\documentclass[12pt]{amsart}
\usepackage{amssymb}
\usepackage{verbatim}
\usepackage[toc,page]{appendix}
\usepackage{mathrsfs}
\usepackage{mathtools}

\usepackage{xcolor}

\newtheorem{thm}{Theorem}[section]

\newtheorem{lem}[thm]{Lemma}

\theoremstyle{definition}

\theoremstyle{remark}

\numberwithin{equation}{section}
\newcommand{\Mod}[1]{\ (\textup{mod}\ #1)}
\newcommand{\Zag}{\mathscr{L}}% Zagier L-series
\newcommand{\MT}{\mathbf{MT}} % main term
\newcommand{\SZE}{\Sigma_{<}}%sum over n<2l
\newcommand{\SIN}{\Sigma_{>}}%sum over n>2l
\newcommand{\MTn}{\Sigma_{0}}%main term n=2l
\newcommand{\Resid}{\mathcal{R}}

\newcommand{\ups}{\upsilon}% coef of Zagier/zeta
\newcommand{\G}{\mathcal{G}}
\newcommand{\Hf}{\mathcal{H}}
\newcommand{\Cc}{\mathcal{C}}% coef in main term
\newcommand{\MAP}{\mathcal{M}}% arith part in main term
\newcommand{\SG}{\mathbf{G}}%sum dm of G fun
\newcommand{\ES}{\mathcal{E}}%sum over e of SG
\newcommand{\TM}{\vartheta}

\newcommand{\R}{\mathbb{R}}

\providecommand{\sgn}{\operatorname{sgn}}

\providecommand{\sym}{\operatorname{sym}}

\providecommand{\Res}{\operatorname{Res}}

\DeclareMathOperator{\res}{res}
\DeclareMathOperator{\arcsinh}{arcsinh}

\newcommand{\HyG}{ {}_2F_1 }
\newcommand{\GenHyG}[5]{ {}_{#1}F_{#2} \left( \begin{matrix} #3 \\ #4 \end{matrix} ; #5 \right) }
\newcommand{\HyGI}{ {}_2\mathrm{I}_1 }
\newcommand{\GenHyGI}[5]{ {}_{#1}\mathrm{I}_{#2} \left( \begin{matrix} #3 \\ #4 \end{matrix} ; #5 \right) }

\mathtoolsset{showonlyrefs}

\begin{document}

\title[]{Non-vanishing of symmetric square $L$-functions in the weight aspect}

\begin{abstract}

We prove a new asymptotic formula for the second moment of symmetric square $L$-functions in the weight aspect on average. This result implies that the associated $L$-function is non-vanishing at the central point for at least $76.69\%$ of holomorphic Hecke cusp forms of bounded weight, improving the previous bound of $70.37\%$.
\end{abstract}

\author{Olga  Balkanova}
\address{Steklov Mathemtical Institute of Russian Academy of Sciences, 8 Gubkina st., Moscow, 119991, Russia}
\email{balkanova@mi-ras.ru}

\author{Dmitry Frolenkov}
\address{HSE University and Steklov Mathematical Institute of Russian Academy of Sciences, 8 Gubkina st., Moscow, 119991, Russia}
\thanks{The work was supported by the Theoretical Physics and Mathematics Advancement Foundation BASIS}
\email{frolenkov@mi-ras.ru}
\keywords{Voronoi summation formula; half-integral weight; L-functions; double Dirichlet series}
\subjclass[2010]{Primary:  11F12, 11L05, 11M06}

\maketitle

%\tableofcontents

%%%%%%%%%%%%%%%%%%%%%%%%%%%%%%%%%%%%%%%%%%%%%%%%%%%%%%%%%%%%%%%%%%%%%%%%%%%%%%%%%%%%%%%%%%%%%%%%%%%%%%%%%%%%%%%%%%%%%%%%%%%%%%%%%%
%%%%%%%%%%%%%%%%%%%%%%%%%%%%%%%%%%%%%%%%%%%%%%%%%%%%%%%%%%%%%%%%%%%%%%%%%%%%%%%%%%%%%%%%%%%%%%%%%%%%%%%%%%%%%%%%%%%%%%%%%%%%%%%%%%

\section{Introduction}
Let $H_{2k}$ be the normalized Hecke basis for the space of holomorphic cusp forms of weight $2k\geq 2$ and level $1$.  It is known that every $f \in H_{2k}$ has a Fourier expansion of the form
\begin{equation}
f(z)=\sum_{n\geq 1}\lambda_f(n)n^{k-1/2}e(nz),\quad e(x)=\exp(2\pi x), \quad \lambda_f(1)=1.
\end{equation}
The associated symmetric square $L$-function  is defined as
\begin{equation}
L(\sym^2f,s)=\zeta(2s)\sum_{n=1}^{\infty}\frac{\lambda_f(n^2)}{n^s},\quad  \Re{s}>1.
\end{equation}
Consider the twisted second moment
\begin{equation}\label{twisted mom def}
M_j(r,k):=\sum_{f \in H_{2k}}^{h}\lambda_f(r^2)L^j(\sym^2f,1/2),
\end{equation}
where the superscript $h$ means that the sum is taken with additional harmonic weight
$\Gamma(2k-1)/((4\pi)^{2k-1} \langle f,f\rangle_1)$ and $\langle f,f\rangle_1$ is the Petersson inner product on the space of level $1$ holomorphic modular forms. In this paper, we establish an asymptotic formula for the averaged second moment
\begin{equation}\label{2 mom averaged}
M^{a}_2(r,K):=\sum_{k}h\left(\frac{k}{K}\right)M_2(r,k)
\end{equation}
provided that $r\ll K^{5/4-\epsilon}.$ As usual,  $h \in C_{0}^{\infty}([a;b])$ is some  non-negative, infinitely differentiable function of compact support such that $0<a<b$.

Previously, an asymptotic formula for \eqref{2 mom averaged} was known only for $r\ll K^{1-\epsilon}$ by a result of Khan \cite[Theorem 3.1]{Khan2010}. Extending this valid range of twists allows us to improve the non-vanishing proportion for $L(\sym^2f,1/2)$, which constitutes the main result of this paper.

%%%%%%%%%%%%%%%%%%%%%%%%%%%%%%%%%%%%%%%%%%%%%%%
\begin{thm}\label{thm:nonvanishing}
For any $0<a<5/8$ we have
\begin{equation}\label{nonvan0}
\sum_{k}h\left(\frac{k}{K}\right)\sum_{\substack{f \in H_{2k}\\L(\sym^2f,1/2)\neq0}}^{h}1\ge\left(1-\frac{1}{(1+a)^3}\right)
\sum_{k}h\left(\frac{k}{K}\right)\sum_{f \in H_{2k}}^{h}1.
\end{equation}
\end{thm}
%%%%%%%%%%%%%%%%%%%%%%%%%%%%%%%%%%%%%%%%%%%%%%%%%%%%%%%%%5
In \cite[Theorem 1.1]{Khan2010}, Khan established \eqref{nonvan0} for $a<1/2$,  thereby showing that the non-vanishing proportion is at least $19/27=0.7037\dots$ Evaluating this for $a$  close to $5/8$ yields the improved proportion  $1685/2197=0.7669\dots$
For comparison, the best conditional result -- obtained under the Generalized Riemann Hypothesis by Iwaniec, Luo, and Sarnak \cite{ILS} -- is $8/9=0.8888\dots$

The non-vanishing proportion is related to the investigation of \eqref{2 mom averaged} via a standard mollification technique \cite[Section 4]{Khan2010}.

%%%%%%%%%%%%%%%%%%%%%%%%%%%%%%%%%%%%%%%%%%%%%%%
\begin{thm}\label{thm:2mom average}
For $r\ll K^{5/4-\epsilon}$ we have
\begin{multline}\label{2momAF0}
M^{a}_2(r,K)=\sum_{e|r^2}
\frac{K}{\sqrt{re_1}}\int_{0}^{\infty}h(x)\Biggl(
\frac{1}{2}\log\frac{xKe_2}{r}\log^2\frac{xK}{e_1e_2}-\frac{1}{6}\log^3\frac{xK}{e_1e_2}+\\+
\log\frac{xKe_2}{r}P_1(\log\frac{xK}{e_1e_2})+P_2(\log\frac{xK}{e_1e_2})\Biggr)dx+O\left(\frac{r^{3/2}}{K^{3/2-\epsilon}}+\frac{(rK)^{\epsilon}}{\sqrt{r}}\right),
\end{multline}
where $P_j(x)$ are  polynomials of degree $j$, and $e=e_1e_2^2$ with $e_1$ square-free.
\end{thm}
%%%%%%%%%%%%%%%%%%%%%%%%%%%%%%%%%%%%%%%%%%%%%%%%%%%%%%%%%5

The main term in \eqref{2momAF0} coincides with that in \cite[Theorem 3.1]{Khan2010} if we replace $K$ by $K/2$ in \eqref{2momAF0}. This difference arises because we consider forms $f \in H_{2k}$, whereas Khan considered  $f \in H_{k}$ with $k\equiv0\Mod{2}.$

%%%%%%%%%%%%%%%%%%%%%%%%%%%%%%%%%%%%%%%%%%%%%%%%%%%%%%

The proof of Theorem \ref{thm:2mom average} differs substantially from that of \cite[Theorem 3.1]{Khan2010}. In \cite{Khan2010}, both \(L\)-functions in \eqref{2 mom averaged} were replaced by an approximate functional equation, followed by an application of the Petersson trace formula. The main term arises from the diagonal contribution in the Petersson formula, along with an additional summand that cancels with a corresponding term from the off-diagonal part. To handle the off-diagonal contribution, Khan first evaluates the sum over \(k\) of the Bessel functions \(J_{2k-1}\). He then applies the Poisson summation formula to evaluate the sum of Kloosterman sums, which, in our view, constitutes the core of his proof of \cite[Theorem 3.1]{Khan2010}.

Our strategy for proving Theorem \ref{thm:2mom average} is conceptually different.
%and aligns with the methodology from our previous works \cite{BFsym2Maass2026, Frolsym2} on the second moments of symmetric square \(L\)-functions for Maass forms.
We begin by applying an approximate functional equation to only one of the \(L\)-functions in \eqref{2 mom averaged}, thereby reducing the problem to the study of a first twisted moment. We then apply a reciprocity-type formula \cite{BF2018}—originally due to Zagier \cite{Zag}—which represents this twisted moment in terms of the first moments of Zagier's \(L\)-series, weighted by certain hypergeometric functions:
\begin{equation}
\sum_{1\leq n<2l}\Zag_{n^2-4l^2}(1/2)\Phi_k\left(\frac{n^2}{4l^2}\right),\quad
\sum_{n>2l}\Zag_{n^2-4l^2}(1/2)\sqrt{n}\Psi_k\left( \frac{4l^2}{n^2}\right).
\end{equation}
We remark that the proof of the reciprocity-type formula incorporates both the Petersson trace formula and a functional equation for the Lerch zeta-function, which plays a role analogous to the Poisson summation formula. Roughly speaking, the problem reduces to estimating
\begin{equation}\label{2mom SZE def0}
\sum_{k}h\left(\frac{k}{K}\right)\sum_{e|r^2}\sum_{m\ll k/(e_1e_2)}
\sum_{1\leq n<2mr/e_2}\Zag_{-n(4mr/e_2-n)}(1/2)f_1\left(\frac{ne_2}{2mr},k\right),
\end{equation}
\begin{equation}\label{2mom SIN def0}
\sum_{k}h\left(\frac{k}{K}\right)\sum_{e|r^2}\sum_{m\ll k/(e_1e_2)}
\sum_{n>0}\Zag_{n(4mr/e_2+n)}(1/2)f_2\left(\frac{ne_2}{2mr},k\right),
\end{equation}
where $f_j(x,k)$ are some hypergeometric functions.

Next, the averages of the hypergeometric functions over \(k\) are evaluated using the asymptotic formulas from \cite{BF2018}. This procedure is analogous to averaging the \(J\)-Bessel functions in Khan's approach. Stopping the analysis of the second moment at this point would only recover Khan's result, since the average values of the hypergeometric functions are very small for $r \ll K^{1-\epsilon}$.
However, when $r \gg K^{1-\epsilon}$, these averages are no longer small. Heuristically, this behavior arises because one of the two hypergeometric functions, $\Psi_k(x)$, can be approximated by the values of \(K\)-Bessel functions, and for $r \gg K^{1-\epsilon}$ we leave the regime where \(K\)-Bessel functions decay exponentially. Consequently, it becomes impossible to estimate the moments of Zagier's \(L\)-series \eqref{2mom SZE def0} and \eqref{2mom SIN def0} solely by analyzing the hypergeometric functions. Another difficulty arises when \(r\) exceeds \(K\), as we are no longer able to evaluate the diagonal contribution in this regime.

To overcome this difficulty, following \cite{BFsym2Maass2026, Frolsym2}, we first perform a series of transformations to convert Zagier's \(L\)-series into a form that is amenable to the Voronoi summation formula. The existence of such a Voronoi formula stems from the fact that Zagier's \(L\)-series appears in the Fourier–Whittaker coefficients of a certain linear combination of Eisenstein series of level $4$ and half-integral weight.

In contrast to the cases considered in \cite{BFsym2Maass2026, Frolsym2}, here the terms with $l \neq 0$ in the Voronoi summation formula can be immediately estimated by analyzing the corresponding integral transforms of the hypergeometric functions $\Phi _{k}$, $\Psi _{k}$ and employing the Weil bound for half-integral weight generalized Kloosterman sums. The core technical challenge in proving Theorem \ref{thm:2mom average} lies in controlling the main term coming from the Voronoi formula. Much like the diagonal main term, we are unable to evaluate its individual contribution directly. To simplify the derivation of these terms, we observe that the mollifier coefficients $x_r$ in \cite[Section 4.1]{Khan2010} are chosen to vanish if $r$ is not square-free. In view of \cite[(4.2)]{Khan2010}, it is therefore sufficient for non-vanishing applications to restrict $r$ in \eqref{twisted mom def} to the form $r=r_1r_2^2$, where $(r_1,r_2)=1$ and $r_1,r_2$ are square-free. After a careful analysis, we establish that the sum of these two main terms can be expressed as a contour integral of an odd function. By oddness, this integral reduces to the residue at $z = 0$, which completes the proof of Theorem \ref{thm:2mom average} for such specific $r$, thereby establishing Theorem \ref{thm:nonvanishing}. Notably, a similar device -- exploiting the symmetry of an odd function -- was previously utilized in various moment computations, including \cite{KMV} and \cite{Sound}. Naturally, Theorem \ref{thm:2mom average} should not be restricted only to these specific values of $r$. For completeness, we extend Theorem \ref{thm:2mom average} to an arbitrary $r$ by induction, using the aforementioned computations as the base case.

%%%%%%%%%%%%%%%%%%%%%%%%%%%%%%%%%%%%%%%%%%%%%%%%

%%%%%%%%%%%%%%%%%%%%%%%%%%%%%%%%%%%%%%%%%%%%%%%%%%

The paper is organized as follows. After establishing the preliminary notation in Section \ref{sec:Prelim}, we review properties of generalized Kloosterman sums of half-integral weight in Section \ref{sec:Kloosterman}. Section \ref{sec:Zagier Voronoi} presents background on Zagier's \(L\)-series, including their Voronoi summation formula, while the reciprocity-type formula for the first moment is given in Section \ref{sec:1st mom}. The analysis of the second moment begins in Section \ref{sec:The second moment} with an application of this reciprocity relation. The diagonal main term is evaluated in Section \ref{sec: MT2}. Next, Sections \ref{sec: SZE} and \ref{sec: SIN} are devoted to the study of multiple series of Zagier's \(L\)-series via the Voronoi formula; the resulting off-diagonal main term is handled in Section \ref{sec:Voronoi MT} (but only in the special case of $r=r_1r_2^2$). We conclude with Section \ref{sec:Proof of Theorem 2mom}, where the joint contribution of both main terms is computed, finalizing the proof of Theorem \ref{thm:2mom average} for this specific choice of $r$. The case of an arbitrary twist is then considered in Section \ref{sec:Proof of Theorem 2mom General}.

%%%%%%%%%%%%%%%%%%%%%%%%%%%%%%%%%%%%%%%%%%%%%%%%%%%%%%%%%%%%%%%%%%%%%%%%%%%%%%%%%%%%%%%%%%%%%%%%%%%%%%%%%%%%%%%%%%%%%%%%%%

%%%%%%%%%%%%%%%%%%%%%%%%%%%%%%%%%%%%%%%%%%%%%%%%%%%%%%%%%%%%%%%%%%%%%%%%
\section{Notation}\label{sec:Prelim}

For an integer $a$, let $\overline{a}_q$ denote its multiplicative inverse modulo $q$, so that $a\overline{a}_q \equiv 1 \pmod{q}$.
When the modulus is clear from the context, we drop the subscript and simply write $\bar{a}$.

Let $\delta_q(n)$ be the indicator function defined by the exponential sum
\begin{equation}\label{delta delta* def}
\delta_q(n) := \frac{1}{q}\sum_{c \pmod{q}} e\left(\frac{nc}{q}\right) = \frac{1}{q}\sum_{k|q} \sideset{}{^*}\sum_{c \pmod{k}} e\left(\frac{nc}{k}\right).
\end{equation}
In particular, $\delta_q(n) = 1$ if $q \mid n$, and $\delta_q(n) = 0$ otherwise. Note that we write an asterisk over a summation sign to denote a sum over a reduced residue system modulo $q$:
\begin{equation}
\sideset{}{^*}\sum_{c \pmod{q}} f(c) = \sum_{\substack{c \pmod{q} \\ (c,q)=1}} f(c).
\end{equation}

As usual, $\Gamma(z)$ denotes the Gamma function and $\psi(z) = \Gamma'(z)/\Gamma(z)$ is its logarithmic derivative (see \cite[Ch.~5]{HMF}).

We define the regularized generalized hypergeometric function ${}_p\mathrm{I}_q$ by scaling it with Gamma factors:
\begin{align}\label{pIq def}
\GenHyGI{p}{q}{a_1, \dots, a_p}{b_1, \dots, b_q}{z}
&:= \frac{\Gamma(a_1) \dots \Gamma(a_p)}{\Gamma(b_1) \dots \Gamma(b_q)} \, \GenHyG{p}{q}{a_1, \dots, a_p}{b_1, \dots, b_q}{z} \nonumber \\
&= \sum_{j=0}^{\infty} \frac{\Gamma(a_1+j) \dots \Gamma(a_p+j)}{\Gamma(b_1+j) \dots \Gamma(b_q+j)} \frac{z^j}{j!}.
\end{align}
In the special case $p=2$ and $q=1$, we use the shorter notation $\HyGI(a,b;c;z)$ and $\HyG(a,b;c;z)$ for these functions.

%%%%%%%%%%%%%%%%%%%%%%%%%%%%%%%%%%%%%%%%%%%%%%%%%%
\section{Kloosterman sums of half-integral weight}\label{sec:Kloosterman}

In this section, we briefly review the definitions and key properties of generalized Kloosterman sums, with a particular focus on the half-integral weight case (for further details, see \cite{Biro2000, IwTopics, Proskurin, Wachter}).

Let $\nu$ be the multiplier system of weight $1/2$ associated with the theta series (see \cite[p.~104]{Biro2000}), defined on $\Gamma_0(4)$ by
\begin{equation}\label{nu def}
\nu(\gamma) = \left(\frac{c}{d}\right) \epsilon_{d}^{-1} \quad \text{for} \quad \gamma = \begin{pmatrix} a & b \\ c & d \end{pmatrix} \in \Gamma_0(4).
\end{equation}
Here, $\left(\frac{c}{d}\right)$ denotes the extended Jacobi symbol (see \cite[Sec.~A1]{Biro2000}), and $\epsilon_d$ is given by
\begin{equation}\label{epsilon def}
\epsilon_{q} =
\begin{cases}
1, & \text{if } q \equiv 1 \pmod{4}, \\
i, & \text{if } q \equiv 3 \pmod{4}.
\end{cases}
\end{equation}

As usual, let $\Gamma_{\mathfrak{a}}$ denote the stabilizer of a cusp $\mathfrak{a}$ in $\Gamma$, and let $\sigma_{\mathfrak{a}}$ be its scaling matrix. For any
\begin{equation}
g = \begin{pmatrix} a & b \\ c & d \end{pmatrix} \in \mathbf{SL}_2(\mathbb{R}),
\end{equation}
we define $j(g,z) := cz+d$ and $j_g(z) := e^{i \arg(cz+d)}$.
Following \cite[(2.40), (2.48)]{IwTopics} (see also \cite[Def.~2.1.1]{Wachter}), we introduce the factor
\begin{equation*}
\omega(\gamma_1,\gamma_2) := \frac{1}{2\pi} \left( \arg j(\gamma_1,\gamma_2 z) + \arg j(\gamma_2,z) - \arg j(\gamma_1\gamma_2,z) \right),
\end{equation*}
and set $\omega_{k}(\gamma_1,\gamma_2) := e(k\omega(\gamma_1,\gamma_2))$.
For two singular cusps $\mathfrak{a}$ and $\mathfrak{b}$, we then define (cf. \cite[(3.4)]{IwTopics}, \cite[Def.~2.1.6]{Wachter})
\begin{equation*}\label{nu ab eq1}
\nu_{\mathfrak{a},\mathfrak{b}}(\gamma) := \nu(\sigma_{\mathfrak{a}}\gamma\sigma_{\mathfrak{b}}^{-1}) \omega_{k}(\sigma_{\mathfrak{a}}^{-1}, \sigma_{\mathfrak{a}}\gamma\sigma_{\mathfrak{b}}^{-1}) \omega_{k}(\gamma\sigma_{\mathfrak{b}}^{-1}, \sigma_{\mathfrak{b}}).
\end{equation*}

The Kloosterman sum associated with the singular cusps $\mathfrak{a}$ and $\mathfrak{b}$ is defined (see \cite[Sec.~A.3]{Biro2000}) by
\begin{equation}\label{Kloos def}
S^{\Gamma}_{\mathfrak{a},\mathfrak{b}}(m,n;c;\nu) :=
\sum_{\gamma = \left(\begin{smallmatrix} a & b \\ c & d \end{smallmatrix}\right) \in \Gamma_{\infty} \setminus \sigma^{-1}_{\mathfrak{a}}\Gamma\sigma_{\mathfrak{b}} / \Gamma_{\infty}}
\overline{\nu_{\mathfrak{a},\mathfrak{b}}(\gamma)} \, e\left(\frac{am+dn}{c}\right).
\end{equation}

By \cite[Lemma~3.1]{BBF2025}, for any odd modulus $(q,2)=1$, we have
\begin{equation}\label{Kl 4N 1/N infty}
S^{\Gamma_0(4)}_{1/1,\infty}(m,n;2q;\nu) =
i \overline{\epsilon_q} \, e\left(\frac{-m}{4}\right)
\sideset{}{^*}\sum_{d \pmod{q}} \left(\frac{d}{q}\right) e\left(\frac{m\overline{4d}+nd}{q}\right),
\end{equation}
while \cite[Lemma~3.2]{BBF2025} yields
\begin{equation}\label{Kl 4N 1/Ninfty 0n}
S^{\Gamma_0(4)}_{1/1,\infty}(0,n;2q;\nu) =
i\sqrt{q} \sum_{d|q} \mu(d) \sum_{t \pmod{q/d}} \delta_{q/d}(t^2-n).
\end{equation}
Furthermore, for $c \equiv 0 \pmod{4}$, it is shown in \cite[Lemma~A.6]{Biro2000} that
\begin{equation}\label{Kl 4N infty infty c=0(4N)}
S^{\Gamma_0(4)}_{\infty,\infty}(m,n;c;\nu) =
\sideset{}{^*}\sum_{d \pmod{c}} \epsilon_d \left(\frac{c}{d}\right) e\left(\frac{m\bar{d}+nd}{c}\right),
\end{equation}
whereas $S^{\Gamma_0(4)}_{\infty,\infty}(m,n;c;\nu) = 0$ if $4 \nmid c$.
If $n \equiv 0, 1 \pmod{4}$, then \cite[Lemma~3.7]{BBF2025} implies that
\begin{equation}\label{Kl 4N infty infty 0n4q n01 q even}
S^{\Gamma_0(4)}_{\infty,\infty}(0,n;4q;\nu) = 2(1+i)\sqrt{q} \sum_{d|q} \mu(d) \sum_{t \pmod{2q/d}} \delta_{4q/d}(t^2-n) \quad \text{if } (q,2)=2,
\end{equation}
\begin{equation}\label{Kl 4N infty infty 0n4q n01 q odd}
S^{\Gamma_0(4)}_{\infty,\infty}(0,n;4q;\nu) = (1+i)\sqrt{q} \sum_{d|q} \mu(d) \sum_{t \pmod{q/d}} \delta_{q/d}(t^2-n) \quad \text{if } (q,2)=1.
\end{equation}

Our subsequent analysis relies on Weil-type bounds for the Kloosterman sums \eqref{Kl 4N 1/N infty} and \eqref{Kl 4N infty infty c=0(4N)}.
It was established in \cite[(3.1), (3.2)]{DasSen} that for $(q,2)=1$,
\begin{equation}\label{DasSen Sest}
S(m,n,q) := \sideset{}{^*}\sum_{d \pmod{q}} \left(\frac{d}{q}\right) e\left(\frac{n\overline{d}+md}{q}\right) \ll 2^{\omega(q)}\sqrt{q(m,n,q)} \ll q^{1/2+\epsilon}\sqrt{(m,n,q)},
\end{equation}
and for $c \equiv 0 \pmod{4}$, \cite[(3.7)]{DasSen} gives
\begin{equation}\label{DasSen Kest}
K_{1}(m,n,c) := \sideset{}{^*}\sum_{d \pmod{c}} \epsilon_{d} \left(\frac{c}{d}\right) e\left(\frac{n\overline{d}+md}{c}\right) \ll 2^{\omega(c)}\sqrt{c(m,n,c)} \ll c^{1/2+\epsilon}\sqrt{(m,n,c)}.
\end{equation}
Combining \eqref{Kl 4N 1/N infty} with \eqref{DasSen Sest}, and \eqref{Kl 4N infty infty c=0(4N)} with \eqref{DasSen Kest}, we arrive at the bounds
\begin{equation}\label{S 1/1inf infinf est}
S^{\Gamma_0(4)}_{1/1,\infty}(m,n;2q;\nu) \ll q^{1/2+\epsilon}\sqrt{(m,n,q)}, \quad \text{and} \quad S^{\Gamma_0(4)}_{\infty,\infty}(m,n;c;\nu) \ll c^{1/2+\epsilon}\sqrt{(m,n,c)},
\end{equation}
which hold for $(q,2)=1$ and $c \equiv 0 \pmod{4}$, respectively.

%%%%%%%%%%%%%%%%%%%%%%%%%%%%%%%%%%%%%%%

%%%%%%%%%%%%%%%%%%%%%%%%%%%%%%%%%%%%%%%%%%%%%%%%%%%%%%%%%%%%%%%%%%%%%%%%%%%%%%%%%%%%%%%%%%%%%%%%%%%%%%%%%%%%%%%%%%%%%%%%%%
\section{Zagier $L$-series and the Voronoi Summation Formula}\label{sec:Zagier Voronoi}
The Zagier $L$-series is defined \cite[sec. 2]{SY}, \cite[Proposition 3]{Zag} for $\Re{s}>1$  as
\begin{equation}\label{Lbyk}
\Zag_{n}(s)=\frac{\zeta(2s)}{\zeta(s)}\sum_{q=1}^{\infty}\frac{\rho_q(n)}{q^{s}}=\zeta(2s)\sum_{q=1}^{\infty}\frac{\ups_q(n)}{q^{s}},
\end{equation}
where
\begin{equation}\label{rho upsilon def}
\rho_q(n):=\#\{x\Mod{2q}:x^2\equiv n\Mod{4q}\},\quad
\ups_q(n):=\sum_{q_2q_3=q}\mu(q_2)\rho_{q_3}(n).
%\lambda_q(n):=\sum_{q_{1}^{2}q_2q_3=q}\mu(q_2)\rho_{q_3}(n).
\end{equation}

It follows from \eqref{rho upsilon def} that $\Zag_n(s)=0$ if $n \equiv 2, 3 \pmod{4}.$ Furthermore, $\Zag_{0}(s)=\zeta(2s-1)$, and for $n=Dl^2$, where $D$ is a fundamental discriminant, we have
\begin{equation}\label{ldecomp}
\Zag_{n}(s)=l^{1/2-s}T_{l}^{(D)}(s)L(s,\chi_D),
\end{equation}
where $L(s,\chi_D)$ is the Dirichlet $L$-function associated with the primitive quadratic character $\chi_D$, and
\begin{equation}\label{eq:td}
T_{l}^{(D)}(s)=\sum_{l_1l_2=l}\chi_D(l_1)\frac{\mu(l_1)}{\sqrt{l_1}}\tau_{s-1/2}(l_2).
\end{equation}

The following estimate is a direct consequence of \eqref{ldecomp} and the Conrey--Iwaniec subconvexity bound
 \cite[Corollary~1.5]{CI} (see \cite[Lemma 4.2]{BF2018}):
\begin{equation}\label{eq:subconvexity}
\Zag_n(1/2)\ll |n|^{\theta},\quad \theta=1/6+\epsilon,\quad
\Zag_{-4l^2}(1/2)\ll |l|^{\epsilon}.
\end{equation}

Using Heath-Brown's large sieve inequality \cite{HB} (see also \cite[Eq.~(3.1)]{KhYoung}), one obtains \cite[Eq.~(3.2)]{KhYoung}
\begin{equation}\label{LZag 2mom estimate}
\sum_{n\le N}|\Zag_n(1/2+it)|^2\ll\left(N+\sqrt{N(1+|t|)}\right)\left(N(1+|t|)\right)^{\epsilon}.
\end{equation}
%%%%%%%%%%%%%%%%%%%%%%%%%%%%%%%%%%%

%%%%%%%%%%%%%%%%%%%%%%%%%%%%%%%%%%%%%%%
To investigate the mean values of $\Zag_{n}(s)$, we will use the Voronoi summation formula established in
\cite{BFVoron}. Let (see \cite[Corollary 1.5]{BFVoron})
\begin{equation}\label{an def}
a_{n}(\rho)=\frac{\pi^{\rho}\Gamma\left(\frac{1}{2}+\rho-\frac{\sgn{n}}{4}\right)}{2^{1+2\rho}|n|^{1/2-\rho}\Gamma(1/2+2\rho)}\Zag_{n}\left(\frac{1}{2}+2\rho\right),
\end{equation}
\begin{equation}\label{a0b0 def}
a_{0}(\rho)=\zeta(1+4\rho),\quad
b_{0}(\rho)=\frac{\sqrt{\pi}\Gamma(2\rho)\zeta(4\rho)}{2^{4\rho}\Gamma(1/2+2\rho)}=
\frac{\pi^{1/2+4\rho}\Gamma(1-4\rho)\zeta(1-4\rho)}{\Gamma(1-2\rho)\Gamma(1/2+2\rho)}.
\end{equation}

Let $\phi(y)$ be a smooth, compactly supported function with support contained either in $\R_+$ or in $\R_-$. We define
\begin{equation}\label{eq:phi+-}
\phi^{+}(s):=\int_{0}^{\infty}\phi(y)y^{s-1}dy,\quad \phi^{-}(s):=\int_{0}^{\infty}\phi(-y)y^{s-1}dy
\end{equation}
and
\begin{equation}\label{Rabpm def}
R^{\pm}(a,x,s):=a\phi^{\pm}(s)\frac{\pi^{s}x^{-2s}}{\Gamma(s\pm 1/4)},
% \quad R_b^{\pm}(x,s)=b_{0}\phi^{\pm}(s)\frac{\pi^{s}x^{-2s}}{\Gamma(s\pm k/2)}.
\end{equation}

We would like to precisely formulate the Voronoi formula for the values $\Zag_{n}(1/2)$. Unfortunately, this was not explicitly written down in \cite{BFVoron}. If the support of $\phi$ is contained in $\R_{\pm}$, the term $\Gamma\left(\frac{1}{2}+\rho\mp\frac{1}{4}\right)$ appears in \eqref{an def}. Taking this into account, the main term in the Voronoi formula will contain (see \cite[(1.7), (1.8)]{BFVoron})
\begin{equation}\label{Rab sum def}
\Resid(x):=
\lim_{\rho\to0}\sum_{\pm}\frac{2^{1+2\rho}\Gamma(1/2+2\rho)}{\pi^{\rho}\Gamma\left(\frac{1}{2}+\rho\mp\frac{1}{4}\right)}
\left(R^{\pm}(a_0(\rho),x,1/2+\rho)+R^{\pm}(b_0(\rho),x,1/2-\rho)\right).
\end{equation}
%%%%%%%%%%%%%%%%%%%%%%%%%%%%%%%%
\begin{lem}\label{lem:Voronoi MT}
The following formula holds:
\begin{equation}\label{Voronoi MTeq0}
\Resid(x)=\int_0^{\infty}\frac{\phi(y)}{x\sqrt{2y}}\left(\log\frac{2y}{\pi x^2}-\frac{\pi}{2}+3\gamma\right)dy+
\int_0^{\infty}\frac{\phi(-y)}{x\sqrt{2y}}\left(\log\frac{2y}{\pi x^2}+\frac{\pi}{2}+3\gamma\right)dy.
\end{equation}
\end{lem}
\begin{proof}
Substituting \eqref{a0b0 def} and \eqref{Rabpm def} into \eqref{Rab sum def}, we obtain
\begin{equation}\label{Voronoi MTeq1}
\Resid(x)=
\frac{2\Gamma(1/2)}{x\Gamma\left(1/4\right)}
\int_0^{\infty}\frac{\phi(y)}{\sqrt{y}}l_+(x,y)dy+
\frac{2\Gamma(1/2)}{x\Gamma\left(3/4\right)}\int_0^{\infty}\frac{\phi(-y)}{\sqrt{y}}
l_-(x,y)dy,
\end{equation}
where
\begin{equation}\label{Voronoi MT l+}
l_+(x,y)=
\lim_{\rho\to0}
\left(
\frac{\zeta(1+4\rho)\sqrt{\pi}}{\Gamma(3/4+\rho)}\left(\frac{\pi y}{x^2}\right)^{\rho}+
\frac{\pi^{1+3\rho}\Gamma(1-4\rho)\zeta(1-4\rho)}{\Gamma(1-2\rho)\Gamma(1/2+2\rho)\Gamma(3/4-\rho)}
\left(\frac{y}{x^2}\right)^{-\rho}
\right),
\end{equation}
\begin{equation}\label{Voronoi MT l-}
l_-(x,y)=
\lim_{\rho\to0}
\left(
\frac{\zeta(1+4\rho)\sqrt{\pi}}{\Gamma(1/4+\rho)}\left(\frac{\pi y}{x^2}\right)^{\rho}+
\frac{\pi^{1+3\rho}\Gamma(1-4\rho)\zeta(1-4\rho)}{\Gamma(1-2\rho)\Gamma(1/2+2\rho)\Gamma(1/4-\rho)}
\left(\frac{y}{x^2}\right)^{-\rho}\right).
\end{equation}
Applying \cite[5.5.3,5.5.5]{HMF}, we have
\begin{equation*}
\Gamma(1/2+2\rho)\Gamma(3/4-\rho)=\frac{\sqrt{\pi}2^{-1/2+2\rho}}{\sin(\pi/4+\pi\rho)}\Gamma(3/4+\rho),
\end{equation*}
\begin{equation*}
\Gamma(1/2+2\rho)\Gamma(1/4-\rho)=\frac{\sqrt{\pi}2^{-1/2+2\rho}}{\sin(\pi/4-\pi\rho)}\Gamma(1/4+\rho),
\end{equation*}
and therefore
\begin{multline}\label{Voronoi MT l+2}
l_+(x,y)=
\lim_{\rho\to0}
\left(
\frac{\zeta(1+4\rho)\sqrt{\pi}}{\Gamma(3/4+\rho)}\left(\frac{\pi y}{x^2}\right)^{\rho}+
\frac{\sqrt{2\pi}\Gamma(1-4\rho)\sin(\pi/4+\pi\rho)\zeta(1-4\rho)}{\Gamma(1-2\rho)\Gamma(3/4+\rho)}
\left(\frac{4y}{\pi^3x^2}\right)^{-\rho}
\right)=\\=
\frac{2\gamma\sqrt{\pi}}{\Gamma(3/4)}+
\frac{\sqrt{\pi}}{4\Gamma(3/4)}\frac{d}{d\rho}\left(
\left(\frac{\pi y}{x^2}\right)^{\rho}-
\frac{\sqrt{2}\Gamma(1-4\rho)\sin(\pi/4+\pi\rho)}{\Gamma(1-2\rho)}
\left(\frac{4y}{\pi^3x^2}\right)^{-\rho}
\right)\Bigl|_{\rho=0}=\\=
\frac{\sqrt{\pi}}{4\Gamma(3/4)}
\left(
\log\left(\frac{4y^2}{\pi^2x^4}\right)-\pi+6\gamma
\right),
\end{multline}
where we have used the fact that \cite[5.4.12]{HMF} $\psi(1)=-\gamma$. Similarly,
\begin{multline}\label{Voronoi MT l-2}
l_-(x,y)=
\lim_{\rho\to0}
\left(
\frac{\zeta(1+4\rho)\sqrt{\pi}}{\Gamma(1/4+\rho)}\left(\frac{\pi y}{x^2}\right)^{\rho}+
\frac{\sqrt{2\pi}\Gamma(1-4\rho)\sin(\pi/4-\pi\rho)\zeta(1-4\rho)}{\Gamma(1-2\rho)\Gamma(1/4+\rho)}
\left(\frac{4y}{\pi^3x^2}\right)^{-\rho}\right)=\\=
\frac{\sqrt{\pi}}{4\Gamma(1/4)}
\left(
\log\left(\frac{4y^2}{\pi^2x^4}\right)+\pi+6\gamma
\right).
\end{multline}
Substituting \eqref{Voronoi MT l+2} and \eqref{Voronoi MT l-2} to \eqref{Voronoi MTeq1}, we obtain \eqref{Voronoi MTeq0}.
\end{proof}
%%%%%%%%%%%%%%%%%%%%%%%%%%%%%%%%%%%%%%%%%%%%%%%%%%%%%%%%%%%
Furthermore, to state the Voronoi summation formula, we need to define the following integral transforms (see \cite[Eqs.~(1.11), (1.12)]{BFVoron}):
\begin{equation}\label{phi hat+ to Phipm def0}
\widehat{\phi}(y):=\int_0^{\infty}\left(\frac{\phi(x)}{x}\Phi^{(+,+)}(xy)+\frac{\phi(-x)}{x}\Phi^{(+,-)}(xy)\right)dx,
\end{equation}
\begin{equation}\label{phi hat- to Phipm def0}
\widehat{\phi}(-y):=\int_0^{\infty}\left(\frac{\phi(x)}{x}\Phi^{(-,+)}(xy)+\frac{\phi(-x)}{x}\Phi^{(-,-)}(xy)\right)dx,
\end{equation}
where $y>0$, and the kernel functions $\Phi^{(+,+)}(x)$ and $\Phi^{(-,-)}(x)$ are defined(see \cite[(1.9)]{BFVoron}) by
\begin{multline}\label{Phi++--def}
\Phi^{(\pm,\pm)}(x):=\lim_{\rho\to0}\left(\frac{\cos\pi(1/4\mp\rho)}{\sin(2\pi\rho)}\sqrt{x}J_{-2\rho}(2\sqrt{x})-
\frac{\cos\pi(1/4\pm\rho)}{\sin(2\pi\rho)}\sqrt{x}J_{2\rho}(2\sqrt{x})\right)=\\=
\frac{\sqrt{x}}{\sqrt{2}}\lim_{\rho\to0}
\frac{J_{-2\rho}(2\sqrt{x})-J_{2\rho}(2\sqrt{x})}{2\sin(\pi\rho)}
\pm\frac{\sqrt{x}}{\sqrt{2}}J_{0}(2\sqrt{x})
=\\=
-\frac{\sqrt{x}}{\sqrt{2}}\left(Y_0(2\sqrt{x})\mp J_{0}(2\sqrt{x})\right),
\end{multline}
where we have used \cite[10.2.4]{HMF}. The functions $\Phi^{(+,-)}(x)$ and $\Phi^{(-,+)}(x)$ are given (see \cite[(1.10)]{BFVoron}) by
\begin{equation}\label{Phi+--+def}
\Phi^{(\pm,\mp)}(x):=\frac{2\sqrt{x}K_{0}(2\sqrt{x})}{\Gamma^2(1/2\mp 1/4)}.
\end{equation}
Let $\sgn(\phi)=\pm1$ depending on whether the support of $\phi$ belongs to $\R_{\pm}$.
According to \cite[Theorems 1.1, 1.2, 1.4]{BFVoron} we have the following formulas.
%%%%%%%%%%%%%%%%%%%%%%%%%%%%%%%%%%%%%%%%%%%%%%%%%%%%%5
\begin{lem}\label{Thm Voronoi an c0mod4}
For $c\equiv0\Mod{4}$ and $ad\equiv1\Mod{c}$ we have
\begin{multline}\label{Thm.eq Voronoi an c0mod4}
\sum_{n=-\infty}^{\infty}\frac{\Zag_{n}\left(1/2\right)}{\sqrt{|n|}}\phi(n)e\left( \frac{an}{c}\right)=
\TM(M_0)e(1/8)\Resid(c)+\\+
\TM(M_0)e(1/8)\sum_{n\neq0}\frac{\Gamma\left(\frac{1}{2}-\frac{\sgn{n}}{4}\right)}{\Gamma\left(\frac{1}{2}-\frac{\sgn{\phi}}{4}\right)}
\frac{\Zag_{n}\left(1/2\right)}{\sqrt{|n|}}\widehat{\phi}\left(\frac{4\pi^2n}{c^2}\right)e\left( -\frac{dn}{c}\right),
\end{multline}
where $\TM(M_0)=\bar{\epsilon}_{d}\left(\frac{c}{d}\right)$.
\end{lem}
%%%%%%%%%%%%%%%%%%%%%%%%%%%%%%%%%%%%%%%%%%%%%5
\begin{lem}\label{Thm Voronoi an codd}
For $(c,2)=1$ and $4ad\equiv -1\Mod{c}$ we have
\begin{multline}\label{Thm.eq Voronoi an codd}
\sum_{n=-\infty}^{\infty}\frac{\Zag_{n}\left(1/2\right)}{\sqrt{|n|}}\phi(n)e\left( \frac{an}{c}\right)=
\TM(M_1)\sqrt{2}\Resid(4c)+\\+
\TM(M_1)\sqrt{2}\sum_{n\neq0}
\frac{\Gamma\left(\frac{1}{2}-\frac{\sgn{n}}{4}\right)}{\Gamma\left(\frac{1}{2}-\frac{\sgn{\phi}}{4}\right)}
\frac{\Zag_{4n}\left(1/2\right)}{\sqrt{|4n|}}
\widehat{\phi}\left(\frac{\pi^2n}{c^2}\right)e\left(\frac{dn}{c}\right),
\end{multline}
where $\TM(M_1)=\bar{\epsilon}_{c}\left(\frac{4d}{c}\right)$.
\end{lem}

%%%%%%%%%%%%%%%%%%%%%%%%%%%%%%%%%%%%%%%%%%%%%%%%%%%%%%%%%%%%%%%%%%%%%%%%%%%%%%%%%

\begin{lem}\label{Thm Voronoi an c2mod4}
For  $c\equiv2\Mod{4}$ let $c_1=c/2$. For $2ad_5\equiv-1\Mod{c_1}$,  $8ad_4\equiv-1\Mod{c_1}$ we have
\begin{multline}\label{Thm.eq Voronoi an c2mod4}
\sum_{n=-\infty}^{\infty}\frac{\Zag_{n}\left(1/2\right)}{\sqrt{|n|}}\phi(n)e\left( \frac{an}{c}\right)=
\frac{\sqrt{2}}{\TM(M_4)}\Resid(4c_1)-
\TM(M_5)\sqrt{2}\Resid(4c_1)\\+
\frac{\sqrt{2}}{\TM(M_4)}\sum_{n\neq0}\frac{\Gamma\left(\frac{1}{2}-\frac{\sgn{n}}{4}\right)}{\Gamma\left(\frac{1}{2}-\frac{\sgn{\phi}}{4}\right)}
\frac{\Zag_{n}\left(1/2\right)}{\sqrt{|n|}}
\widehat{\phi}\left(\frac{\pi^2n}{4c_1^2}\right)e\left(\frac{d_4n}{c_1}\right)-\\-
\TM(M_5)\sqrt{2}\sum_{n\neq0}\frac{\Gamma\left(\frac{1}{2}-\frac{\sgn{n}}{4}\right)}{\Gamma\left(\frac{1}{2}-\frac{\sgn{\phi}}{4}\right)}
\frac{\Zag_{4n}\left(1/2\right)}{\sqrt{|4n|}}
\widehat{\phi}\left(\frac{\pi^2n}{c_1^2}\right)e\left(\frac{d_5n}{c_1}\right),
\end{multline}
where $\TM^{-1}(M_4)=\epsilon_{c_1}\left(\frac{8a}{c_1}\right)$ and $\TM(M_5)=\bar{\epsilon}_{c_1}\left(\frac{4d_5}{c_1}\right)$.
\end{lem}
%%%%%%%%%%%%%%%%%%%%%%%%%%%%%%%%%%%%%%%%%%%%%%%%%%%%%%%%%%%%%%%%%%%%%%%%%%%%%%%%%
%%%%%%%%%%%%%%%%%%%%%%%%%%%%%%%%%%%%%%%%%%%%%%%%%%%%%%%%%%%%%%%%%%%%%%%%%%%%%%%%%

%%%%%%%%%%%%%%%%%%%%%%%%%%%%%%%%%%%%%%%%%%%%%%%%%%%%%%%%%%%%%%%%%%%%%%%%%%%%%%%%%%%%%%%%%%%%%%%%%%%%%%%%%%%%%%%%%%%%%%%%%%

\section{The first moment}\label{sec:1st mom}
Let ${}_2F_{1}(a,b,c;x)$ be the Gauss hypergeometric function and
\begin{equation}\label{defphi}
\Phi_k(x):=\frac{\Gamma(k-1/4)\Gamma(3/4-k)}{\Gamma(1/2)}{}_2F_{1}\left(k-\frac{1}{4},\frac{3}{4}-k,1/2;x \right),
\end{equation}
\begin{equation}
\Psi_k(x):=x^k\frac{\Gamma(k-1/4)\Gamma(k+1/4)}{\Gamma(2k)}{}_2F_{1}\left(k-\frac{1}{4},k+\frac{1}{4},2k;x \right).
\end{equation}
In \cite[Theorem 2.1]{BF2018}, the following formula for the first twisted moment $M_1(l,k)$ \eqref{twisted mom def} was obtained.
%%%%%%%%%%%%%%%%%%%%%%%%%%%%%%%%%%%%%%%%55
\begin{thm}\label{thm:M1explicitformula}
For any $l \geq 1$ we have
\begin{multline}\label{M1mainformula}
M_1(l,k)=\MT_1(l,k)+
\frac{\sqrt{2\pi}(-1)^k}{2\sqrt{l}}\frac{\Gamma(k-1/4)}{\Gamma(k+1/4)}\mathscr{L}_{-4l^2}(1/2)+\\
\frac{1}{\sqrt{l}} \sum_{1\leq n<2l}
\Zag_{n^2-4l^2}(1/2)\Phi_k\left(\frac{n^2}{4l^2}\right)+
\frac{1}{l\sqrt{2}} \sum_{n>2l}\Zag_{n^2-4l^2}(1/2)\sqrt{n}\Psi_k\left( \frac{4l^2}{n^2}\right),
\end{multline}
\begin{equation}\label{MT1 def1}
\MT_1(l,k)=
\frac{1}{2\sqrt{l}}\biggl(-2\log{l}-3\log{2\pi}+\frac{\pi}{2}+3\gamma+\psi(k-1/4)+ \psi(k+1/4) \biggr).
\end{equation}
\end{thm}
%%%%%%%%%%%%%%%%%%%%%%%%%%%%%%%%%5
For future applications, it is convenient to rewrite the sums in \eqref{M1mainformula} as
\begin{equation}\label{M1sum n<2l eq2}
\sum_{1\leq n<2l}\Zag_{n^2-4l^2}(1/2)\Phi_k\left(\frac{n^2}{4l^2}\right)=
\sum_{1\leq n<2l}\Zag_{-n(4l-n)}(1/2)\Phi_k\left(\left(1-\frac{n}{2l}\right)^2\right),
\end{equation}
\begin{equation}\label{M1sum n>2l eq2}
 \sum_{n>2l}\Zag_{n^2-4l^2}(1/2)\sqrt{n}\Psi_k\left( \frac{4l^2}{n^2}\right)=\sum_{n>0}\Zag_{n(4l+n)}(1/2)\sqrt{n+2l}\Psi_k\left( \frac{1}{(1+n/(2l))^2}\right).
\end{equation}
%%%%%%%%%%%%%%%%%%%%%%%%%%%%%%%%%5
We also need the following asymptotic formulas for $\Phi_k(x)$ (see \cite[Theorem 6.10]{BF2018}) and  $\Psi_k(x)$ (see \cite[Theorem 6.17]{BF2018}).
\begin{thm}\label{thm:apprphi}
For $\xi \in (0,\pi^2/4)$ we have
\begin{equation}\label{Phik LG}
\Phi_k(\cos^2\sqrt{\xi})=\frac{-\pi}{\xi^{1/4}(\sin\sqrt{\xi})^{1/2}}\left(Z_J(\xi)+C_YZ_Y(\xi)+C_JZ_J(\xi)\right),
\end{equation}
where $Z_Y$ and $Z_J$ are given by
\begin{multline}\label{ZYxi}
Z_Y(\xi)=\sqrt{\xi}Y_0((2k-1)\sqrt{\xi})\sum_{n=0}^{N}\frac{A_Y(n;\xi)}{(2k-1)^{2n}}-\\
\frac{\xi}{2k-1}Y_1((2k-1)\sqrt{\xi})\sum_{n=0}^{N-1}\frac{B_Y(n;\xi)}{(2k-1)^{2n}}+O\left( \frac{\sqrt{\xi}|Y_0((2k-1)\sqrt{\xi})|}{k^{2N+1}}\right),
\end{multline}
\begin{multline}\label{ZJxi}
Z_J(\xi)=\sqrt{\xi}J_0((2k-1)\sqrt{\xi})\sum_{n=0}^{N}\frac{A_J(n;\xi)}{(2k-1)^{2n}}-\\
\frac{\xi}{2k-1}J_1((2k-1)\sqrt{\xi})\sum_{n=0}^{N-1}\frac{B_J(n;\xi)}{(2k-1)^{2n}}+
O\left( \frac{\sqrt{\xi}|J_0((2k-1)\sqrt{\xi})|}{k^{2N+1}}\right)
\end{multline}
with $A_Y,B_Y,A_J,B_J\ll 1$ and
\begin{equation}\label{CY CJ}
C_Y=1+O\left(\frac{1}{k}\right), \quad C_J=O\left(\frac{1}{k^2} \right).
\end{equation}
\end{thm}
Note that one can write down  more terms in the asymptotic expansions \eqref{CY CJ}.
%%%%%%%%%%%%%%%%%%%%%%%%%%%%%%%%%%%%%
\begin{thm}\label{thm:approxPsi}
For  $\xi \in (0, \infty)$ the following equality holds:
\begin{equation}\label{Psik LG}
\Psi_k\left(\frac{1}{\cosh^2{\sqrt{\xi}/2}} \right)\left( \xi\sinh^2{\sqrt{\xi}}\right)^{1/4}=
C_KZ_K(\xi),
\end{equation}
where $Z_K(\xi)$ is given by
\begin{multline}\label{ZKxi}
Z_K(\xi)=\sqrt{\xi}K_0((k-1/2)\sqrt{\xi})\sum_{n=0}^{N}\frac{A_K(n;\xi)}{(k-1/2)^{2n}}-\\\frac{\xi}{k-1/2}K_1((k-1/2)\sqrt{\xi})
\sum_{n=0}^{N-1}\frac{B_K(n;\xi)}{(k-1/2)^{2n}}+O\left(\frac{\sqrt{\xi}K_0((k-1/2)\sqrt{\xi})}{k^{2N+1}}
%\min\left(\sqrt{\xi}, \frac{1}{\xi}\right)
\right),
\end{multline}
with $A_K,B_K\ll 1$ and  $C_K=2+O(k^{-1})$.
\end{thm}
Again, one can write down more terms in the asymptotic formula for $C_K$.
%%%%%%%%%%%%%%%%%%%%%%%%%%%%%%%%%%%%%%%%%5
%%%%%%%%%%%%%%%%%%%%%%%%%%%%%%%%%%%%%%%%%%%%%%%%%%%%%%%%%%%%%%%%%%%%%%%%%%%%%%%%%%%%%%%%%%%%%%%%%%%%%%%%%%%%%%%%%%%%%%%%%%%%%%%%%%%%%%%%%%%%
\section{The second moment}\label{sec:The second moment}
%%%%%%%%%%%%%%%%%%%%%%%%%%%%%%%%%%%%%%%%%%%%%%%%%%%%%%%%%%%%%%%%%%%%%%%%%%%%%%%%%%%%%%%%%%%%%%%%%%%%%%%%%%%%%%%%%%%%%%%%%%%%%%%%%%%%%%%%%%%%
We start by writing the approximate functional equation for the symmetric square $L$-function, see \cite[(1.5) and Lemma 2.2]{Khan2010}.

%%%%%%%%%%%%%%%%%%%%%%%%%%%%%%%%%%%%%%%%%%%%%%%%%%%%%%%%%%%%%%%%%%%%%%%%%%%%%%
\begin{lem}
The following approximate functional equation holds:
\begin{equation}\label{approx.func.eq.}
L(\sym^2f,1/2)=2\sum_{m=1}^{\infty}\frac{\lambda_f(m^2)}{m^{1/2}}V_k\left(m\right),
\end{equation}
where for any $y>0$ and $a>0$
\begin{equation}\label{approx.fun.eq.Vdef}
V_k(y)=\frac{1}{2\pi i}\int_{(a)}\frac{L_{\infty}(\sym^2f,1/2+z)}{L_{\infty}(\sym^2f,1/2)}\zeta(1+2z)y^{-z}\frac{dz}{z},
\end{equation}
\begin{multline}\label{L.infinity}
L_{\infty}(\sym^2f,s)=\pi^{-3s/2}\Gamma\left(\frac{s+1}{2}\right)\Gamma\left(\frac{s+2k-1}{2}\right)\Gamma\left(\frac{s+2k}{2}\right)=\\=
\pi^{1/2-3s/2}2^{2-s-2k}\Gamma\left(\frac{s+1}{2}\right)\Gamma(s+2k-1),
\end{multline}
\begin{equation}\label{Linf(z)/Linfe q0}
\frac{L_{\infty}(\sym^2f,1/2+z)}{L_{\infty}(\sym^2f,1/2)}=
\pi^{-3z/2}2^{-z}\frac{\Gamma\left(3/4+z/2\right)(2k)^{z}}{\Gamma\left(3/4\right)}\left(1+O((1+|z|)/k)\right).
\end{equation}
For $j\ge0$ and any $A>0$ we have
\begin{equation}\label{approx.fun.eq.Vest}
y^jV_k^{(j)}(y)\ll \left(\frac{k}{y}\right)^{A}.
\end{equation}
\end{lem}
It follows from \eqref{approx.fun.eq.Vdef} and \eqref{L.infinity} that due to the exponential decay of the integrand, we can restrict the integration in \eqref{approx.fun.eq.Vdef} to $|z| \ll k^{\epsilon}$ up to a negligible error term.

%%%%%%%%%%%%%%%%%%%%%%%%%%%%%%%%%%%%%%%%%%%%%%%%%%%%%%%%%%%%%%%%%%%%%%%%%%%%%%
\begin{lem}\label{lem:V j-deriv}
For $y \ll k^{1+\epsilon}$ and $u \asymp 1$, the following estimate holds:
\begin{equation}\label{approx.fun.eq.V uK est}
\frac{d^j}{du^j}V_{uK}(y)\ll \left(\frac{K}{y}\right)^{A}.
\end{equation}
\end{lem}
\begin{proof}
According to \eqref{approx.fun.eq.Vdef} and \eqref{L.infinity}, for $|z| \ll k^{\epsilon}$ it is necessary to estimate
\begin{equation}
\frac{\mathrm{d}^j}{\mathrm{d}u^j} \frac{\Gamma(2uk + z - 1/2)}{\Gamma(2uk - 1/2)}.
\end{equation}
Let $\psi^{(n)}(z)$ be the classical polygamma function of order $n$, and define
\begin{equation}
D_n(uk,z) := \psi^{(n)}(2uk + z - 1/2) - \psi^{(n)}(2uk - 1/2).
\end{equation}
Applying the asymptotic expansion of the polygamma function, we obtain
\begin{equation}
D_0(uK,z)\ll \log(2uK+z-1/2)-\log(2uK-1/2)\ll\frac{1+|z|}{K},
\end{equation}
\begin{equation}
D_n(uK,z)\ll (2uK+z-1/2)^{-n}-(2uK-1/2)^{-n}\ll\frac{1+|z|}{K^{n+1}}.
\end{equation}
Then a direct computation leads to
\begin{multline}
\frac{d^j}{du^j}\frac{\Gamma(2uK+z-1/2)}{\Gamma(2uK-1/2)}\ll (2K)^j\frac{\Gamma(2uK+z-1/2)}{\Gamma(2uK-1/2)}\sum_{k_1+2k_2+\ldots+jk_j=j}D_0^{k_1}D_1^{k_2}\cdot\ldots\cdot D_{j-1}^{k_j}\ll\\\ll
(1+|z|)^j\frac{\Gamma(2uK+z-1/2)}{\Gamma(2uK-1/2)}.
\end{multline}
Therefore,
\begin{equation}\label{approx.fun.eq.V uK est2}
\frac{d^j}{du^j}V_{uK}(y)\ll \int_{(A)}\Gamma\left(3/4+z/2\right)\zeta(1+2z)
(1+|z|)^j\frac{\Gamma(2uK+z-1/2)}{\Gamma(2uK-1/2)}
\frac{dz}{y^z}\ll \left(\frac{K}{y}\right)^{A}.
\end{equation}
\end{proof}
%%%%%%%%%%%%%%%%%%%%%%%%%%%%%%%%%%%%%%%%%%%%%%%%%%%%%%%%%%%%%%%%%%%%%%%%%%%%%%
Consider $M_2(r,k)$ given by \eqref{twisted mom def}. Expressing one of the $L$-functions via \eqref{approx.func.eq.} and using the multiplicity of the Fourier coefficients,
\begin{equation}\label{eq:mult}
\lambda_f(m)\lambda_f(n) = \sum_{e \mid (m,n)} \lambda_f\left( \frac{mn}{e^2} \right),
\end{equation}
we obtain
\begin{equation*}
M_2(r,k)=2\sum_{m=1}^{\infty}\frac{V_k(m)}{m^{1/2}}\sum_{e|(m^2,r^2)}M_1\left(\frac{mr}{e},k\right).
\end{equation*}
Let $e = e_1 e_2^2$, where $e_1$ is square-free. Then $e$ divides $m^2$ if and only if $e_1 e_2 \mid m$. Therefore,
\begin{equation}\label{2mom to 1mom}
M_2(r,k)=2\sum_{e|r^2}\sum_{m=1}^{\infty}\frac{V_k(me_1e_2)}{\sqrt{me_1e_2}}M_1\left(\frac{mr}{e_2},k\right).
\end{equation}
Using \eqref{2mom to 1mom} and \eqref{M1mainformula}, we obtain the following expression for the second averaged moment \eqref{2 mom averaged}:
\begin{equation}\label{2 mom averaged to Zag}
M^{a}_2(r,K)=\MT_2(r,K)+\MTn(r,K)+\SZE(r,K)+\SIN(r,K),
\end{equation}
where
\begin{equation}\label{2mom MT2def}
\MT_2(r,K)=2\sum_{k}h\left(\frac{k}{K}\right)\sum_{e|r^2}\sum_{m=1}^{\infty}\frac{V_k(me_1e_2)}{\sqrt{me_1e_2}}\MT_1\left(\frac{mr}{e_2},k\right),
\end{equation}
\begin{equation}\label{2mom MTn def}
\MTn(r,K)=\sum_{k}h\left(\frac{k}{K}\right)\sum_{e|r^2}\sum_{m=1}^{\infty}\frac{V_k(me_1e_2)}{\sqrt{me_1e_2}}
\frac{\sqrt{2\pi}(-1)^k}{\sqrt{mr/e_2}}\frac{\Gamma(k-1/4)}{\Gamma(k+1/4)}\Zag_{-4m^2r^2/e_2^2}(1/2),
\end{equation}
by \eqref{M1sum n<2l eq2}, we get
\begin{multline}\label{2mom SZE def}
\SZE(r,K)=2\sum_{k}h\left(\frac{k}{K}\right)\sum_{e|r^2}\sum_{m=1}^{\infty}\frac{V_k(me_1e_2)}{m\sqrt{re_1}}
\sum_{1\leq n<2mr/e_2}\Zag_{-n(4mr/e_2-n)}(1/2)\\\times\Phi_k\left(\left(1-\frac{ne_2}{2mr}\right)^2\right),
\end{multline}
and by \eqref{M1sum n>2l eq2}, we get
\begin{multline}\label{2mom SIN def}
\SIN(r,K)=2\sum_{k}h\left(\frac{k}{K}\right)\sum_{e|r^2}\sum_{m=1}^{\infty}\frac{V_k(me_1e_2)e_2}{mr\sqrt{2me_1e_2}}
\sum_{n>0}\Zag_{n(4mr/e_2+n)}(1/2)\times\\\times\sqrt{n+2mr/e_2}\Psi_k\left(\left(1+\frac{ne_2}{2mr}\right)^{-2}\right).
\end{multline}
Now we can immediately estimate \eqref{2mom MTn def}. Using \eqref{eq:subconvexity} and \eqref{approx.fun.eq.Vest}, we show that
\begin{multline}\label{2mom MTn est}
\MTn(r,K) \ll \sum_{k} h\left(\frac{k}{K}\right) \sum_{e \mid r^2} \sum_{m \ll K^{1+\epsilon}/e_1e_2} \frac{V_k(me_1e_2)}{\sqrt{me_1e_2}} \frac{(mr/e_2)^{\epsilon}}{\sqrt{kmr/e_2}} \\
\ll K^{1/2+\epsilon} \sum_{e \mid r^2} \frac{r^{\epsilon}}{\sqrt{re_1}} \ll \frac{K^{1/2+\epsilon}}{r^{1/2-\epsilon}},
\end{multline}
which is already smaller than the main term, which is roughly $K/\sqrt{r}$.
Nevertheless, one can prove a much stronger bound by invoking the presence of the oscillating factor $(-1)^k$ in \eqref{2mom MTn def}. Indeed, applying the Poisson summation formula, we obtain
\begin{multline}\label{2mom MTn est1}
\sum_{k}h\left(\frac{k}{K}\right)V_k(y)(-1)^k\frac{\Gamma(k-1/4)}{\Gamma(k+1/4)}=\\=
K\sum_{m}\int_{-\infty}^{\infty}h(z)V_{zK}(x)
\frac{\Gamma(Kz-1/4)}{\Gamma(Kz+1/4)}
e^{\pi iKz(2m+1)}dz.
\end{multline}
Arguing as in the proof of Lemma \ref{lem:V j-deriv}, we find
\begin{equation*}
\frac{d^j}{dz^j}\left(h(z)V_{zK}(x)
\frac{\Gamma(Kz-1/4)}{\Gamma(Kz+1/4)}\right)\ll K^{\epsilon}.
\end{equation*}
Therefore, integrating by parts $j$ times in \eqref{2mom MTn est1}, we get
\begin{equation}\label{2mom MTn est2}
\sum_{k}h\left(\frac{k}{K}\right)V_k(y)(-1)^k\frac{\Gamma(k-1/4)}{\Gamma(k+1/4)}\ll K^{-A},
\end{equation}
and thus
\begin{equation}\label{2mom MTn est3}
\MTn(r,K)\ll K^{-A}.
\end{equation}

%%%%%%%%%%%%%%%%%%%%%%%%%%%%%%%%%%%%%%%%%%%%%%%%%%%%%%%%%%%%%%%%%%%%%%%%%%%%%%%%%%%%%%%%%%%%%%%%%%%%%%%%%%%%%%%%%%%%%%%%%%
\section{Analysis of $\MT_2(r,K)$}\label{sec: MT2}
To deal with \eqref{2mom MT2def}, it is convenient to rewrite the main term of the first moment \eqref{M1mainformula}. Applying \cite[Theorem 5.5]{BF2018} with $s = 1/2+u$, we find that the main term equals
\begin{multline}\label{MT1 def2}
\MT_1(l,k)=\lim_{u\to0}\Bigl(
\frac{\zeta(1+2u)}{l^{1/2+u}}+\\+
\frac{(2\pi)^{1/2+u}}{\sqrt{\pi}}\frac{\zeta(2u)}{l^{1/2-u}}
\cos\left(\frac{\pi(1/2+u)}{2}\right)\frac{\Gamma(k-1/4-u/2)\Gamma(k+1/4-u/2)}{\Gamma(k+1/4+u/2)\Gamma(k-1/4+u/2)}\Gamma(u)
\Bigr).
\end{multline}
Applying the functional equation of the Riemann zeta-function \cite{HMF}, followed by \cite{HMF}, we obtain
\begin{multline}\label{MT1 zetaFE}
\zeta(2u)\Gamma(u)=
2^{2u}\pi^{2u-1}\zeta(1-2u)\sin(\pi u)\Gamma(1-2u)\Gamma(u)=\\=
\pi^{2u-3/2}\zeta(1-2u)\Gamma(1/2-u)\sin(\pi u)\Gamma(1-u)\Gamma(u)=
\pi^{2u-1/2}\zeta(1-2u)\Gamma(1/2-u).
\end{multline}
Using \cite[5.5.5]{HMF} twice yields
\begin{equation}\label{MT1 Gamma/Gamma}
\frac{\Gamma(k-1/4-u/2)\Gamma(k+1/4-u/2)}{\Gamma(k+1/4+u/2)\Gamma(k-1/4+u/2)}=2^{2u}\frac{\Gamma(2k-1/2-u)}{\Gamma(2k-1/2+u)}.
\end{equation}
Substituting \eqref{MT1 zetaFE} and \eqref{MT1 Gamma/Gamma} into \eqref{MT1 def2}, we obtain
\begin{equation}\label{MT1 def3}
\MT_1(l,k)=\lim_{u\to0}\left(
\frac{\zeta(1+2u)}{l^{1/2+u}}+
\frac{\zeta(1-2u)}{l^{1/2-u}}\MT_1(u,k)\right),
\end{equation}
where
\begin{equation}\label{MT1(u) def1}
\MT_1(u,k)=
\pi^{3u-1/2}2^{1/2+3u}\Gamma(1/2-u)
\frac{\Gamma(2k-1/2-u)}{\Gamma(2k-1/2+u)}
\cos(\pi(1/4+u/2)).
\end{equation}
Using \eqref{MT1 def3}, we rewrite \eqref{2mom MT2def} as
\begin{multline}\label{2mom MT2 eq1}
\MT_2(r,K)=2\sum_{k}h\left(\frac{k}{K}\right)\sum_{e|r^2}\lim_{u\to0}\sum_{m=1}^{\infty}\frac{V_k(me_1e_2)}{\sqrt{me_1e_2}}\\\times
\biggl(\frac{\zeta(1+2u)e_2^{1/2+u}}{(mr)^{1/2+u}}+
\frac{\zeta(1-2u)e_2^{1/2-u}}{(mr)^{1/2-u}}\MT_1(u,k) \biggr).
\end{multline}
Applying  \eqref{approx.fun.eq.Vdef}, we find
\begin{multline}\label{2mom MT2 eq2}
\MT_2(r,K)=2\sum_{k}h\left(\frac{k}{K}\right)\sum_{e|r^2}
\frac{1}{\sqrt{re_1}}
\lim_{u\to0}\frac{1}{2\pi i}\int_{(a)}
\frac{L_{\infty}(\sym^2f,1/2+z)}{L_{\infty}(\sym^2f,1/2)}\frac{\zeta(1+2z)}{(e_1e_2)^{z}}\\\times
\biggl(\frac{\zeta(1+2u)\zeta(1+z+u)}{(r/e_2)^{u}}+
\frac{\zeta(1-2u)\zeta(1+z-u)}{(r/e_2)^{-u}}\MT_1(u,k) \biggr)\frac{dz}{z}.
\end{multline}
Evaluating the limit by L'Hôpital's rule, we show that
\begin{multline}\label{2mom MT2 lim}
\lim_{u\to0}\biggl(\frac{\zeta(1+2u)\zeta(1+z+u)}{(r/e_2)^{u}}+
\frac{\zeta(1-2u)\zeta(1+z-u)}{(r/e_2)^{-u}}\MT_1(u,k) \biggr)=
\zeta'(1+z)+\\+
\zeta(1+z)\biggl(2\gamma-\log\frac{r}{e_2}-\frac{1}{2}\frac{d}{du}\MT_1(u,k)\bigl|_{u=0}\biggr)=\\=
\zeta'(1+z)+\zeta(1+z)\biggl(2\gamma-\log\frac{r}{e_2}-\frac{3}{2}\log(2\pi)+\frac{\pi}{4}+\frac{\psi(1/2)}{2}+\psi(2k-1/2)\biggr).
\end{multline}
Therefore, substituting \eqref{2mom MT2 lim} into \eqref{2mom MT2 eq2} and using \cite{HMF}, we find
\begin{multline}\label{2mom MT2 eq3}
\MT_2(r,K)=2\sum_{k}h\left(\frac{k}{K}\right)\sum_{e|r^2}
\frac{1}{\sqrt{re_1}}
\frac{1}{2\pi i}\int_{(a)}
\frac{L_{\infty}(\sym^2f,1/2+z)}{L_{\infty}(\sym^2f,1/2)}\frac{\zeta(1+2z)}{(e_1e_2)^{z}}\\\times
\Biggl(\zeta'(1+z)+\zeta(1+z)\biggl(\frac{3\gamma}{2}-\log\frac{r}{e_2}-\frac{3}{2}\log(2\pi)+\log(k)+\frac{\pi}{4}\biggr)\Biggr)
\frac{dz}{z}+O\left(\frac{(rK)^{\epsilon}}{\sqrt{r}}\right).
\end{multline}
Unfortunately, we are unable to obtain an asymptotic formula for this expression. The primary reason is that $\frac{L_{\infty}(\operatorname{sym}^2f,1/2+z)}{L_{\infty}(\operatorname{sym}^2f,1/2)}\sim k^z$ (see \eqref{Linf(z)/Linfe q0}), which yields the factor $(k/(e_1e_2))^{z}$ in the integrand. If $e_1e_2 \gg k^{1+\epsilon}$ or $e_1e_2 \ll k^{1-\epsilon}$, we can shift the line of integration to the left or to the right to obtain an asymptotic formula; however, this approach fails when $e_1e_2 \sim k$. This transitional case may occur since $e_1e_2 \mid r$ and $r \gg k$.
Note that the condition $e \mid r^2$ is equivalent to $e_1e_2 \mid r$, where $e = e_1 e_2^2$ and $e_1$ is square-free. We proceed to evaluate the sum over $e$ in \eqref{2mom MT2 eq3}. This will allow us to compute, in Section \ref{sec:Proof of Theorem 2mom}, the total sum of \eqref{2mom MT2 eq3} combined with the main terms arising from the analysis of $\SZE(r,K)$ and $\SIN(r,K)$ (see \eqref{2mom SZE def} and \eqref{2mom SIN def}).

%%%%%%%%%%%%%%%%%%%%%%%%%%%%%%%%%%%%%%%%%%%%%%%%%%%%%%%%%%%%%%%%%%%%%%%%%%%%%%%%%%%%%%%%%%%%
Let
\begin{equation}\label{C1 def}
\Cc_1(z):=\log\frac{k^2}{8\pi^3}+2\frac{\zeta'(1+z)}{\zeta(1+z)}+3\gamma+\frac{\pi}{2}.
\end{equation}
Using this notation, \eqref{2mom MT2 eq3} can be written as
\begin{multline}\label{2mom MT2 eq4}
\MT_2(r,K)=\frac{1}{\sqrt{r}}\sum_{k}h\left(\frac{k}{K}\right)
\frac{1}{2\pi i}\int_{(a)}\frac{L_{\infty}(\sym^2f,1/2+z)}{L_{\infty}(\sym^2f,1/2)}\zeta(1+2z)\zeta(1+z)\\\times
\sum_{e_1e_2|r}\frac{|\mu(e_1)|}{(e_1e_2)^{z}\sqrt{e_1}}\left(\Cc_1(z)+2\log\frac{e_2}{r}\right)
\frac{dz}{z}+O\left(\frac{(rK)^{\epsilon}}{\sqrt{r}}\right).
\end{multline}
We now evaluate the inner sum:
\begin{multline}\label{MT sum e def}
\sum_{e_1e_2|r}\frac{|\mu(e_1)|}{(e_1e_2)^{z}\sqrt{e_1}}\left(\Cc_1(z)+2\log\frac{e_2}{r}\right)=\\=
\sum_{e_2|r}\sum_{e_1|r/e_2}\frac{|\mu(e_1)|}{(e_1e_2)^{z}\sqrt{e_1}}\left(\Cc_1(z)+2\log\frac{e_2}{r}\right)=\\=
\sum_{e_3|r}\sum_{e_1|e_3}\frac{|\mu(e_1)|}{\sqrt{e_1}}\left(\frac{e_3}{e_1r}\right)^{z}\left(\Cc_1(z)-2\log e_3\right)=\\=
r^{-z}
\sum_{e_3|r}
\left(\Cc_1(z)-2\frac{d}{du}\right)e_3^{z+u}\Bigl|_{u=0}
\sum_{e_1|e_3}\frac{|\mu(e_1)|}{e_1^{1/2+z}}:=r^{-z}\MAP(z,r).
\end{multline}
Using multiplicativity and writing $r = r_1 r_2^2$, where $r_1$ and $r_2$ are coprime, square-free numbers, we get
\begin{multline}\label{MT sum e eq2}
\sum_{e_3|r}e_3^{z+u}\sum_{e_1|e_3}\frac{|\mu(e_1)|}{e_1^{1/2+z}}=
\prod_{p|r_1}\left(1+p^{z+u}(1+p^{-1/2-z})\right)\\\times
\prod_{p|r_2}\left(1+p^{z+u}(1+p^{-1/2-z})+p^{2z+2u}(1+p^{-1/2-z})\right).
\end{multline}
Substituting \eqref{MT sum e eq2} into \eqref{MT sum e def}, we obtain
\begin{multline}\label{MT sum e eq3}
r^{-z}\MAP(z,r)=r^{-z}
\prod_{p|r_1}\left(1+p^{z}+p^{-1/2}\right)
\prod_{p|r_2}\left(1+p^{z}+p^{-1/2}+p^{2z}+p^{-1/2+z})\right)\\\times
\Biggl(
\Cc_1(z)-2\sum_{p|r_1}\frac{p^{z}+p^{-1/2}}{1+p^{z}+p^{-1/2}}\log p
-2\sum_{p|r_2}\frac{p^{z}+p^{-1/2}+2p^{2z}+2p^{-1/2+z}}{1+p^{z}+p^{-1/2}+p^{2z}+p^{-1/2+z}}\log p
\Biggr).
\end{multline}
Finally, substitution of \eqref{MT sum e eq3} and \eqref{Linf(z)/Linfe q0} into \eqref{2mom MT2 eq4} yields
\begin{multline}\label{2mom MT2 eq5}
\MT_2(r,K)=\frac{1}{\sqrt{r}}\sum_{k}h\left(\frac{k}{K}\right)
\frac{1}{2\pi i}\int_{(a)}
\pi^{-3z/2}\frac{\Gamma\left(3/4+z/2\right)}{\Gamma\left(3/4\right)}\\\times
\zeta(1+2z)\zeta(1+z)\MAP(z,r)\left(\frac{k}{r}\right)^{z}\frac{dz}{z}
+O\left(\frac{(rK)^{\epsilon}}{\sqrt{r}}\right).
\end{multline}

%%%%%%%%%%%%%%%%%%%%%%%%%%%%%%%%%%%%%%%%%%%%%%%%%%%%%%%%%%%%%%%%%%%%%%%%%%%%%%%%%%%%%%%%%%%%%%%%%%%%%%%%%%%%%%%%%%%%%%%%%%
\section{Analysis of $\SZE(r,K)$}\label{sec: SZE}
It follows from  \eqref{approx.fun.eq.Vest} that one can truncate the sum over $m$ in \eqref{2mom SZE def} to $m\ll K^{1+\epsilon}/(e_1e_2)$ up to a negligible error term. Let us consider the contribution of $n\gg K^{\epsilon-2}mr/e_2$ to \eqref{2mom SZE def}. We are going to show that after averaging over $k$, such summands also become negligible. To simplify the exposition, let $l:=mr/e_2.$ We replace $\Phi_k(x)$ in \eqref{2mom SZE def} by \eqref{Phik LG}. To this end, let
\begin{equation}\label{Phik xi}
\cos^2\sqrt{\xi}=\left(1-\frac{n}{2l}\right)^2\Rightarrow \sin\frac{\sqrt{\xi}}{2}=\sqrt{\frac{n}{4l}} \Rightarrow \xi=4\arcsin^2\sqrt{\frac{n}{4l}}.
\end{equation}
Therefore, for $n \gg lK^{\epsilon-2}$, we have $K\sqrt{\xi} \gg K^{\epsilon}$ and we can replace the Bessel functions in \eqref{ZYxi} and \eqref{ZJxi} by their asymptotic expansions \cite[formulas 8.451.1 and 8.451.2]{GR}. Thus, taking sufficiently many terms in the asymptotic expansions \eqref{ZYxi}, \eqref{ZJxi}, and \eqref{CY CJ}, we see that it is enough to show that the following sum is negligible for $n \gg lK^{\epsilon-2}$:
\begin{equation}\label{SZE trunc1}
\sum_{k}h\left(\frac{k}{K}\right)V_k(x)e^{\pm2ki\sqrt{\xi}},
\end{equation}
where $x=me_1e_2$ and $\xi$ is given by \eqref{Phik xi}. Applying the Poisson summation formula, we have
\begin{equation*}
\sum_{k}h\left(\frac{k}{K}\right)V_k(x)e^{\pm2ki\sqrt{\xi}}=
K\sum_{m}\int_{-\infty}^{\infty}h(z)V_{zK}(x)e^{2iKz(\pm\sqrt{\xi}+\pi m)}dz.
\end{equation*}
Since $n < 2mr / e_2 = 2l$ in \eqref{2mom SZE def}, we infer $0 < \xi < \pi^2 / 4$, and thus, for $m \neq 0$, we have $K z(\pm\sqrt{\xi} + \pi m) \gg K$. Integrating by parts $j$ times with the use of \eqref{approx.fun.eq.V uK est}, we get that
\begin{equation*}\label{SZE trunc2}
\sum_{k}h\left(\frac{k}{K}\right)V_k(x)e^{\pm(2k-1)i\sqrt{\xi}}=
K\int_{-\infty}^{\infty}h(z)V_{zK}(x)e^{\pm2iKz\sqrt{\xi}}dz+O(K^{-A}).
\end{equation*}
For $K\sqrt{\xi} \gg K^{\epsilon}$, integrating by parts $j$ times once again, we finally obtain
\begin{equation*}
\sum_{k}h\left(\frac{k}{K}\right)V_k(x)e^{\pm2ki\sqrt{\xi}}\ll K^{-A}.
\end{equation*}
%%%%%%%%%%%%%%%%%%%5
Thus, up to a negligible error term, it holds that
\begin{multline}\label{SZE eq2}
\SZE(r,K)=2\sum_{k}h\left(\frac{k}{K}\right)\sum_{e|r^2}\sum_{m\ll K^{1+\epsilon}(e_1e_2)^{-1}}\frac{V_k(me_1e_2)}{m\sqrt{re_1}}\times\\
\sum_{1\leq n\ll K^{-2+\epsilon}mr/e_2}\Zag_{-n(4mr/e_2-n)}(1/2)\Phi_k\left(\left(1-\frac{ne_2}{2mr}\right)^2\right).
\end{multline}
Now we change the order of summation over $m$ and $n$, obtaining that the sums are over
\begin{equation}
1\le n\ll \frac{r}{K^{1-\epsilon}e},\quad
\frac{K^{2-\epsilon}ne_2}{r}\ll m\ll \frac{K^{1+\epsilon}}{e_1e_2},
\end{equation}
where $e=e_1e_2^2$. The next step is to perform the change of variables $4mr / e_2 - n = q$, that is,
\begin{equation}\label{SZE m to q}
m=\frac{q+n}{4re_2^{-1}},\quad q\equiv -n\Mod{\frac{4r}{e_2}}.
\end{equation}
This results in
\begin{multline}\label{SZE eq3}
\SZE(r,K)=8\sum_{k}h\left(\frac{k}{K}\right)\sum_{e|r^2}\frac{\sqrt{r}}{\sqrt{e}}\sum_{n\ll rK^{-1+\epsilon}e^{-1}}\sum_{\substack{nK^{2-\epsilon}\ll q\ll rK^{1+\epsilon}/e\\q\equiv -n\Mod{4r/e_2}}}
V_k\left(\frac{(q+n)e}{4r}\right)\times\\
\frac{\Zag_{-qn}(1/2)}{q+n}\Phi_k\left(\left(1-\frac{2n}{q+n}\right)^2\right).
\end{multline}
Now we perform one more change of variables, $q = l / n$, obtaining the congruence condition $l + n^2 \equiv 0 \pmod{4rn / e_2}$. To detect this congruence, we apply the decomposition \eqref{delta delta* def} of the delta function:
\begin{equation}\label{deltaq(l)}
\delta_{4rn/e_2}(l+n^2)=\frac{e_2}{4rn}\sum_{c|4rn/e_2}
\mathop{{\sum}^*}_{a \Mod{c}}e\left(\frac{a(l+n^2)}{c}\right).
\end{equation}
Using  \eqref{deltaq(l)}, we rewrite \eqref{SZE eq3} as follows:
\begin{multline}\label{SZE eq4}
\SZE(r,K)=2\sum_{k}h\left(\frac{k}{K}\right)\sum_{e|r^2}\frac{e_2}{\sqrt{re}}\sum_{n\ll rK^{-1+\epsilon}e^{-1}}\sum_{c|4rn/e_2}
\mathop{{\sum}^*}_{a \Mod{c}}
e\left(\frac{an^2}{c}\right)\times\\
\sum_{n^2K^{2-\epsilon}\ll l\ll nrK^{1+\epsilon}/e}e\left(\frac{al}{c}\right)V_k\left(\frac{(l+n^2)e}{4rn}\right)
\frac{\Zag_{-l}(1/2)}{l+n^2}\Phi_k\left(\frac{(l-n^2)^2}{(l+n^2)^2}\right).
\end{multline}
To simplify the analysis of the integral transforms in the Voronoi summation formula, we decompose the sum over $l$ into dyadic intervals. Furthermore, we need to split the sum over $c$ in \eqref{SZE eq4} into three cases:
\begin{equation}\label{c cases}
c\equiv0\pmod{4},\quad (c,2)=1,\quad c\equiv2\pmod{4}.
\end{equation}
This yields
\begin{equation}\label{SZE jL def}
\SZE(r,K)=\sum_{j=0}^2\sum_{n^2K^{2-\epsilon}\ll L\ll nrK^{1+\epsilon}/e}\SZE^{(j)}(r,K,L),
\end{equation}
\begin{multline}\label{SZE jL eq1}
\SZE^{(j)}(r,K,L)=\\=2\sum_{k}h\left(\frac{k}{K}\right)\sum_{e|r^2}\frac{e_2}{\sqrt{re}}\sum_{n\ll rK^{-1+\epsilon}e^{-1}}
\sum_{\substack{c|4rn/e_2\\c\equiv \pm j\Mod{4}}}
\mathop{{\sum}^*}_{a \Mod{c}}
e\left(\frac{an^2}{c}\right)\SZE^{(j)}(a,c,r,k,L),
\end{multline}
\begin{equation}\label{SZE jcL eq1}
\SZE^{(j)}(a,c,r,k,L)=
\sum_{l}e\left(\frac{al}{c}\right)V_k\left(\frac{(l+n^2)e}{4rn}\right)U\left(\frac{l}{L}\right)
\frac{\Zag_{-l}(1/2)}{l+n^2}\Phi_k\left(\frac{(l-n^2)^2}{(l+n^2)^2}\right).
\end{equation}
where $U(x)$ is a smooth compactly supported function.
It is convenient to rewrite the sum over $l$ in \eqref{SZE jcL eq1} in the form of \eqref{Thm.eq Voronoi an c0mod4}. To do this, we replace $l$ by $-l$ in \eqref{SZE jcL eq1} and $a$ by $-a$ in \eqref{SZE jL eq1}, obtaining
\begin{multline}\label{SZE jL eq2}
\SZE^{(j)}(r,K,L)=\\=
2\sum_{k}h\left(\frac{k}{K}\right)\sum_{e|r^2}\frac{e_2}{\sqrt{re}}\sum_{n\ll rK^{-1+\epsilon}e^{-1}}
\sum_{\substack{c|4rn/e_2\\c\equiv \pm j\Mod{4}}}
\mathop{{\sum}^*}_{a \Mod{c}}
e\left(\frac{-an^2}{c}\right)\SZE^{(j)}(a,c,r,k,L),
\end{multline}
\begin{equation}\label{SZE jcL eq2}
\SZE^{(j)}(a,c,r,k,L)=
\sum_{l}\frac{\Zag_{l}\left(1/2\right)}{\sqrt{|l|}}g_1(l)e\left(\frac{al}{c}\right),
\end{equation}
\begin{equation}\label{SZE g1 def}
g_1(l)=
V_k\left(\frac{(-l+n^2)e}{4rn}\right)U\left(\frac{-l}{L}\right)
\frac{\sqrt{-l}}{-l+n^2}\Phi_k\left(\frac{(-l-n^2)^2}{(-l+n^2)^2}\right).
\end{equation}
Now we apply the Voronoi formula to \eqref{SZE jcL eq2} (specifically, we apply \eqref{Thm.eq Voronoi an c0mod4} if $j=0$, \eqref{Thm.eq Voronoi an codd} if $j=1$, and \eqref{Thm.eq Voronoi an c2mod4} if $j=2$). First, we estimate the integral transforms \eqref{phi hat+ to Phipm def0} and \eqref{phi hat- to Phipm def0} of $g_1(x)$ arising on the right-hand side of the Voronoi formula.

%%%%%%%%%%%%%%%%%%%%%%%%%%%%%%%%%%%%%%%%%%%%%%%%%%%%%5
\begin{lem}\label{lem:g1 est}
For $m \in \mathbb{Z}$, $L \gg n^2 K^{2-\epsilon}$, $\alpha \in \{\pi^2, 4\pi^2\}$, and any $A > 1$,it holds that
\begin{equation}\label{g1 est1}
\widehat{g_1}\left(\frac{\alpha m}{c^2}\right)\ll\frac{1}{(L|m|/c^2)^{A}} \quad\hbox{if}\quad \frac{L|m|}{c^2}\gg k^{\epsilon},
\end{equation}
\begin{equation}\label{g1 est2}
\widehat{g_1}\left(\frac{\alpha m}{c^2}\right)\ll k^{\epsilon}\sqrt{\frac{|m|}{c^2}}.
\end{equation}
\end{lem}
\begin{proof}
Let $y: = \alpha m / c^2$. It follows from \eqref{SZE g1 def}, \eqref{phi hat+ to Phipm def0}, \eqref{phi hat- to Phipm def0}, \eqref{Phi++--def}, and \eqref{Phi+--+def} that
\begin{equation}\label{g1 hat eq1}
\widehat{g_1}(y) \ll \int_0^{\infty} \frac{g_1(-x)}{x} \sqrt{x |y|} B_{0}(2\sqrt{x |y|}) \, dx,
\end{equation}
where $B_0$ is either the $Y_0$, $J_0$, or $K_0$ Bessel function.
Substituting  \eqref{SZE g1 def} into \eqref{g1 hat eq1}, we obtain
\begin{equation}\label{g1 hat eq2}
\widehat{g_1}(y)\ll\sqrt{|y|}\int_0^{\infty}
V_k\left(\frac{(x+n^2)e}{4rn}\right)U\left(\frac{x}{L}\right)
\Phi_k\left(\frac{(x-n^2)^2}{(x+n^2)^2}\right)
\frac{B_{0}(2\sqrt{x|y|})dx}{x+n^2},
\end{equation}
and therefore,
\begin{equation}\label{g1 hat eq3}
\widehat{g_1}(y)\ll \sqrt{|y|}\int_0^{\infty}
V_k\left(\frac{(xL+n^2)e}{4rn}\right)U\left(x\right)
\Phi_k\left(\frac{(xL-n^2)^2}{(xL+n^2)^2}\right)
\frac{B_{0}(2\sqrt{xL|y|})dx}{x+n^2/L}.
\end{equation}
Estimating the integral trivially by taking the absolute value (and using the estimates $B_0(x) \ll x^{\epsilon}$, $V_k(x) \ll k^{\epsilon}$, and $\Phi_k(x) \ll k^{\epsilon}$), we immediately obtain the estimate \eqref{g1 est2}. To establish the first estimate \eqref{g1 est1}, we shall apply the following result (see \cite[Lemma~6.1]{Har}):
\begin{equation}\label{Harcos est}
\int_0^{\infty}F(x)B_0(a\sqrt{x})dx=\pm\left(\frac{2}{a}\right)^j\int_0^{\infty}F^{(j)}(x)
x^{j/2}B_j(a\sqrt{x})dx,
\end{equation}
where $F$ is a smooth compactly supported function.
Thus, to estimate \eqref{g1 hat eq3} by applying \eqref{Harcos est}, we need to investigate the higher derivatives of the functions in the integrand of \eqref{g1 hat eq3}. The only nontrivial case involves the derivatives of the function $\Phi_k$. There are various approaches to estimating the higher derivatives of a hypergeometric function; one of the simplest methods is to use the fact that it satisfies a second-order linear differential equation. It follows from \cite[Section~2.7.2, formulas~(7)–(9)]{BE} that the function
\begin{equation}\label{Y def}
Y_k(x)=x^{1/4}(1-x)^{1/2}\Phi_k(x)
\end{equation}
satisfies the differential equation
\begin{equation}\label{Y dif eq}
Y_k''+\left(\frac{3}{16x^2}+\frac{1}{4(1-x)^2}+\frac{(2k-1)^2+3/4}{4x(1-x)}\right)Y_k=0.
\end{equation}
Note that we have the following estimates for the argument of $\Phi_k$ in \eqref{g1 hat eq3}:
\begin{equation}\label{Y arg}
0<1-\frac{(xL-n^2)^2}{(xL+n^2)^2}=\frac{4xLn^2}{(xL+n^2)^2}\ll\frac{n^2}{L}\ll k^{\epsilon-2}
.\end{equation}
Thus, the term in the brackets in \eqref{Y dif eq} is smaller than $K^{\epsilon} / (1-x)^2$.

Roughly speaking, for $1-x \ll K^{\epsilon-2}$, we have $Y_k''(x) \ll K^{\epsilon} Y_k(x) / (1-x)^2$. Using \eqref{Y def}, we rewrite \eqref{g1 hat eq3} as
\begin{multline}\label{g1 hat eq4}
\widehat{g_1}(y) \ll \sqrt{\frac{L|y|}{n^2}} \int_0^{\infty} V_k\left(\frac{(xL+n^2)e}{4rn}\right) U(x) Y_k\left(\frac{(xL-n^2)^2}{(xL+n^2)^2}\right) \\
\times \left(\frac{x + n^2/L}{x - n^2/L}\right)^{1/2} B_{0}\bigl(2\sqrt{xL|y|}\,\bigr) \, \frac{dx}{\sqrt{x}}.
\end{multline}
Let
\begin{equation}\label{g1 F def}
F(x):=V_k\left(\frac{(xL+n^2)e}{4rn}\right)\left(\frac{x+n^2/L}{x-n^2/L}\right)^{1/2}\frac{U\left(x\right)}{\sqrt{x}}.
\end{equation}
Then, applying \eqref{Harcos est}, we obtain
\begin{equation}\label{g1 hat eq5}
\widehat{g_1}(y)\ll \frac{\sqrt{L|y|/n^2}}{(L|y|)^{j/2}}\int_0^{\infty}
\frac{d^j}{dx^j}\left(F(x)Y_k\left(\frac{(xL-n^2)^2}{(xL+n^2)^2}\right)\right)
x^{j/2}B_j(2\sqrt{xL|y|})dx.
\end{equation}
Using \eqref{approx.fun.eq.Vest}, we infer that $F^{(n)}(x)\ll k^{\epsilon}.$
To estimate $\frac{d^n}{dx^n} Y_k \left( \frac{(xL-n^2)^2}{(xL+n^2)^2} \right)$, we apply \eqref{Y dif eq}. Straightforward calculations then lead to
\begin{equation}\label{Y 1der}
\frac{d}{dx}Y_k\left(\left(1-\frac{2n^2/L}{x+n^2/L}\right)^2\right)=
\left(\frac{4n^2/L}{(x+n^2/L)^2}-\frac{8(n^2/L)^2}{(x+n^2/L)^3}\right)
Y'_k\left(\frac{(xL-n^2)^2,}{(xL+n^2)^2}\right),
\end{equation}
\begin{multline}\label{Y 2der}
\frac{d^2}{dx^2}Y_k\left(\left(1-\frac{2n^2/L}{x+n^2/L}\right)^2\right)=
\left(\frac{24(n^2/L)^2}{(x+n^2/L)^4}-\frac{8n^2/L}{(x+n^2/L)^3}\right)
Y'_k\left(\frac{(xL-n^2)^2}{(xL+n^2)^2}\right)+\\+
\left(\frac{4n^2/L}{(x+n^2/L)^2}-\frac{8(n^2/L)^2}{(x+n^2/L)^3}\right)^2
Y''_k\left(\frac{(xL-n^2)^2}{(xL+n^2)^2}\right).
\end{multline}
Expressing $Y_k''$ by means of \eqref{Y dif eq}, we obtain
\begin{equation}\label{Y 2der2}
\frac{d^2}{dx^2}Y_k\left(\left(1-\frac{2n^2/L}{x+n^2/L}\right)^2\right)=
h_{1,1}(x)Y'_k\left(\frac{(xL-n^2)^2}{(xL+n^2)^2}\right)+h_{2,1}(x)
Y_k\left(\frac{(xL-n^2)^2}{(xL+n^2)^2}\right),
\end{equation}
where $h_{1,1}\ll n^2/L$ and $h_{2,1}\ll 1+\frac{k^2n^2}{L}\ll k^{\epsilon}$. Repeating this process, we infer
\begin{equation}\label{Y 2der3}
\frac{d^n}{dx^n}Y_k\left(\left(1-\frac{2n^2/L}{x+n^2/L}\right)^2\right)=
h_{1,n}(x)Y'_k\left(\frac{(xL-n^2)^2}{(xL+n^2)^2}\right)+h_{2,n}(x)
Y_k\left(\frac{(xL-n^2)^2}{(xL+n^2)^2}\right),
\end{equation}
where $h_{1,n}\ll k^{\epsilon}n^2/L$ and $h_{2,n}\ll k^{\epsilon}$. Therefore, \eqref{g1 hat eq5} can be rewritten as
\begin{multline}\label{g1 hat eq6}
\widehat{g_1}(y)\ll \frac{\sqrt{L|y|/n^2}}{(L|y|)^{j/2}}\int_0^{\infty}
\left(
h_{1,j}(x)Y'_k\left(\frac{(xL-n^2)^2}{(xL+n^2)^2}\right)+h_{2,j}(x)
Y_k\left(\frac{(xL-n^2)^2}{(xL+n^2)^2}\right)
\right)\\\times
B_j(2\sqrt{xL|y|})dx,
\end{multline}
where $h_{1,j}\ll k^{\epsilon}n^2/L$ and $h_{2,j}\ll k^{\epsilon}$.

By virtue of the estimates
\begin{equation}\label{Y est1}
Y_k\left(\frac{(xL-n^2)^2}{(xL+n^2)^2}\right) \ll \sqrt{\frac{n^2}{L}} \Phi_k\left(\frac{(xL-n^2)^2}{(xL+n^2)^2}\right) \ll K^{\epsilon} \sqrt{\frac{n^2}{L}},
\end{equation}
one can immediately evaluate the term involving $Y_k$ in \eqref{g1 hat eq6}, which yields
\begin{equation}\label{g1 hat eq7}
\widehat{g_1}(y) \ll \frac{K^{\epsilon} \sqrt{|y|}}{(L|y|)^{j/2+1/4}} + \frac{\sqrt{L|y|/n^2}}{(L|y|)^{j/2}} \int_0^{\infty} h_{1,j}(x) Y'_k\left(\frac{(xL-n^2)^2}{(xL+n^2)^2}\right) B_j\bigl(2\sqrt{xL|y|}\,\bigr) \, dx.
\end{equation}
Using \eqref{Y 1der} along with the estimate $h_{1,j} \ll K^{\epsilon} n^2 / L$, we integrate the last integral by parts to obtain
\begin{equation}
\widehat{g_1}(y)\ll
\frac{k^{\epsilon}\sqrt{|y|}}{(L|y|)^{j/2+1/4}}+
\frac{\sqrt{L|y|/n^2}}{(L|y|)^{j/2}}\int_0^{\infty}
Y_k\left(\frac{(xL-n^2)^2}{(xL+n^2)^2}\right)\frac{d}{dx}\left(h_{3,j}(x)
B_j(2\sqrt{xL|y|})\right)dx,
\end{equation}
where $h_{3,j} \ll K^{\epsilon}$. Applying \cite[~10.6.2 or 10.29.2]{HMF} to evaluate the derivative of the Bessel function and using \eqref{Y est1}, we find that
\begin{equation}\label{g1 hat eq8}
\widehat{g_1}(y) \ll \frac{K^{\epsilon} \sqrt{|y|}}{(L|y|)^{j/2+1/4}} + \frac{K^{\epsilon} \sqrt{|y|}}{(L|y|)^{j/2-1/4}} \ll \frac{1}{(L|y|)^{A}},
\end{equation}
provided that $L|y| \gg K^{\epsilon}$.
\end{proof}
%%%%%%%%%%%%%%%%%%%%%%%%%%%%%%%%%%%%%%%%%%%%%%%%%%%%%5
Using lemma \ref{lem:g1 est}, we deduce the next result.
\begin{lem}\label{lem:SZE sumj MT+ET}
The following asymptotic formula holds:
\begin{multline}\label{SZE sumj MT+ET}
\sum_{j=0}^2\mathop{{\sum}^*}_{a \Mod{c}}
\sum_{\substack{c|4rn/e_2\\c\equiv \pm j\Mod{4}}}
e\left(\frac{-an^2}{c}\right)
\SZE^{(j)}(a,c,r,k,L)=\\=\int_0^{\infty}g_1(-y)
\Biggl(\sum_{\substack{q|rn/e_2\\(q,2)=1}}\left(
\frac{(1-i)}{8q}S^{\Gamma_0(4)}_{\infty,\infty}(0,n^2;4q;\nu)-
\frac{i}{4q}S^{\Gamma_0(4)}_{1/1,\infty}(0,n^2;2q;\nu)\right)+\\+
\sum_{\substack{q|rn/e_2\\(q,2)=2}}\frac{(1-i)}{8q}S^{\Gamma_0(4)}_{\infty,\infty}(0,n^2;4q;\nu)
\Biggr)\left(\log\frac{2y}{\pi (4q)^2}+\frac{\pi}{2}+3\gamma\right)\frac{dy}{\sqrt{y}}
+O\left(\frac{k^{\epsilon}(rn/e_2)^{3/2+\epsilon}}{L}\right).
\end{multline}
\end{lem}
\begin{proof}
In the case $j = 0$ (that is, $c \equiv 0 \pmod{4}$), we apply \eqref{Thm.eq Voronoi an c0mod4} and Lemma~\ref{lem:Voronoi MT} to obtain
\begin{multline}\label{SZE j=0 eq1}
\mathop{{\sum}^*}_{a \Mod{c}}e\left(\frac{-an^2}{c}\right)
\SZE^{(0)}(a,c,r,k,L)=\\=
\frac{ e(1/8)}{c}
\mathop{{\sum}^*}_{a \Mod{c}}e\left(\frac{-an^2}{c}\right)\TM(M_0)\int_0^{\infty}g_1(-y)
\left(\log\frac{2y}{\pi c^2}+\frac{\pi}{2}+3\gamma\right)\frac{dy}{\sqrt{2y}}+\\+
e(1/8)\sum_{l\neq0}\widehat{g_1}\left(\frac{4\pi^2l}{c^2}\right)
\frac{\Gamma\left(\frac{1}{2}-\frac{\sgn{l}}{4}\right)\Zag_{l}\left(\frac{1}{2}\right)}{\Gamma(3/4)\sqrt{|l|}}
\sum_{\substack{a\Mod{c}\\ad\equiv1\Mod{c}}}e\left(\frac{-an^2-dl}{c}\right)\TM(M_0),
\end{multline}
where $\TM(M_0)=\bar{\epsilon}_{d}\left(\frac{c}{d}\right)$.
Since $c > 0$, it holds that $\left(\frac{c}{-d} \right) = \left(\frac{c}{d} \right)$, and applying \eqref{Kl 4N infty infty c=0(4N)}, we obtain
\begin{multline}\label{SZE j=0 eq2}
\sum_{\substack{a\Mod{c}\\ad\equiv1\Mod{c}}}e\left(\frac{-an^2-dl}{c}\right)\TM(M_0)=\sum_{\substack{d\Mod{c}\\ad\equiv1\Mod{c}}}e\left(\frac{an^2+dl}{c}\right)
\bar{\epsilon}_{-d}\left(\frac{c}{-d}\right)=\\=-i\mathop{{\sum}^*}_{d \Mod{c}}e\left(\frac{n^2\bar{d}+ld}{c}\right)
\epsilon_{d}\left(\frac{c}{d}\right)=-iS^{\Gamma_0(4)}_{\infty,\infty}(l,n^2;c;\nu).
\end{multline}
To establish the last equality, observe that since $c \equiv 0 \pmod{4}$, it holds that $ad \equiv 1 \pmod{4}$, which implies $\epsilon_d = \epsilon_a$.
Moreover, applying the identity $\left(\frac{c}{d+cA}\right) = \left(\frac{c}{d}\right)$ from \cite[Section~A1]{Biro2000}, which holds for $c \equiv 0 \pmod{4}$, we obtain $\left(\frac{c}{d}\right) = \left(\frac{c}{a}\right)$.

Substituting \eqref{SZE j=0 eq2} to \eqref{SZE j=0 eq1}, we find
\begin{multline}\label{SZE j=0 eq3}
\mathop{{\sum}^*}_{a \Mod{c}}e\left(\frac{-an^2}{c}\right)
\SZE^{(0)}(a,c,r,k,L)=\\=
\frac{(1-i)}{2c}S^{\Gamma_0(4)}_{\infty,\infty}(0,n^2;c;\nu)\int_0^{\infty}g_1(-y)
\left(\log\frac{2y}{\pi c^2}+\frac{\pi}{2}+3\gamma\right)\frac{dy}{\sqrt{y}}+\\+
e(-1/8)\sum_{l\neq0}\widehat{g_1}\left(\frac{4\pi^2l}{c^2}\right)
\frac{\Gamma\left(\frac{1}{2}-\frac{\sgn{l}}{4}\right)\Zag_{l}\left(\frac{1}{2}\right)}{\Gamma(3/4)\sqrt{|l|}}
S^{\Gamma_0(4)}_{\infty,\infty}(l,n^2;c;\nu).
\end{multline}
To estimate the series over $l$ in \eqref{SZE j=0 eq3}, we employ \eqref{S 1/1inf infinf est}, \eqref{g1 est1}, and \eqref{g1 est2}. It follows from \eqref{g1 est1} that the terms with $|l| \gg c^2 K^{\epsilon} / L$ are negligibly small. Applying \eqref{S 1/1inf infinf est}, \eqref{g1 est2}, and the Cauchy–Schwarz inequality, followed by \eqref{LZag 2mom estimate}, we obtain
\begin{equation}\label{SZE j=0 eq4}
\sum_{l\neq0}\ldots\ll \frac{k^{\epsilon}}{c^{1/2-\epsilon}}\sum_{|l|\ll c^2k^{\epsilon}/L}
|\Zag_{l}\left(1/2\right)|\sqrt{(|l|,n^2,c)}\ll \frac{k^{\epsilon}c^{3/2+\epsilon}}{L}.
\end{equation}
Therefore, using \eqref{SZE j=0 eq4}, we rewrite \eqref{SZE j=0 eq3} as
\begin{multline}\label{SZE j=0 eq5}
\mathop{{\sum}^*}_{a \Mod{c}}e\left(\frac{-an^2}{c}\right)
\SZE^{(0)}(a,c,r,k,L)=\\=
\frac{(1-i)}{2c}S^{\Gamma_0(4)}_{\infty,\infty}(0,n^2;c;\nu)\int_0^{\infty}g_1(-y)
\left(\log\frac{2y}{\pi c^2}+\frac{\pi}{2}+3\gamma\right)\frac{dy}{\sqrt{y}}+
O\left(\frac{k^{\epsilon}c^{3/2+\epsilon}}{L}\right).
\end{multline}
It remains to perform the summation over $c \mid (4rn / e_2)$ with $c \equiv 0 \pmod{4}$. Setting $c = 4q$, we obtain
\begin{multline}\label{SZE j=0 eq6}
\sum_{\substack{c|4rn/e_2\\c\equiv0\Mod{4}}}
\mathop{{\sum}^*}_{a \Mod{c}}e\left(\frac{-an^2}{c}\right)
\SZE^{(0)}(a,c,r,k,L)=O\left(\frac{k^{\epsilon}(rn/e_2)^{3/2+\epsilon}}{L}\right)+\\+
\sum_{q|rn/e_2}\frac{(1-i)}{8q}S^{\Gamma_0(4)}_{\infty,\infty}(0,n^2;4q;\nu)
\int_0^{\infty}g_1(-y)\left(\log\frac{2y}{\pi (4q)^2}+\frac{\pi}{2}+3\gamma\right)\frac{dy}{\sqrt{y}}.
\end{multline}
%%%%%%%%%%%%5
Consider the case $j = 1$ (that is, $(c, 2) = 1$). Applying \eqref{Thm.eq Voronoi an codd} and Lemma~\ref{lem:Voronoi MT} to \eqref{SZE jcL eq2}, we obtain
\begin{multline}\label{SZE j=1 eq1}
\mathop{{\sum}^*}_{a \Mod{c}}e\left(\frac{-an^2}{c}\right)
\SZE^{(1)}(a,c,r,k,L)=\\=
\frac{\sqrt{2}}{4c}
\mathop{{\sum}^*}_{a \Mod{c}}e\left(\frac{-an^2}{c}\right)\TM(M_1)
\int_0^{\infty}g_1(-y)
\left(\log\frac{2y}{\pi (4c)^2}+\frac{\pi}{2}+3\gamma\right)\frac{dy}{\sqrt{2y}}
+\\+
2^{-1/2}\sum_{l\neq0}\widehat{g_1}\left(\frac{\pi^2l}{c^2}\right)
\frac{\Gamma\left(\frac{1}{2}-\frac{\sgn{l}}{4}\right)\Zag_{l}\left(\frac{1}{2}\right)}{\Gamma(3/4)\sqrt{|l|}}
\sum_{\substack{a\Mod{c}\\4ad\equiv-1\Mod{c}}}e\left(\frac{-an^2+dl}{c}\right)\TM(M_1),
\end{multline}
where $\TM(M_1)=\bar{\epsilon}_{c}\left(\frac{4d}{c}\right)$. Let us evaluate the sum over $a$. Since $ad \equiv -1 \pmod{c}$, it holds that $\left(\frac{4d}{c}\right) = \left(\frac{-a}{c}\right) = (-1)^{\frac{c-1}{2}} \left(\frac{a}{c}\right)$, and thus we obtain
\begin{multline}\label{SZE j=1 eq2}
\sum_{\substack{a\Mod{c}\\4ad\equiv-1\Mod{c}}}e\left(\frac{-an^2+dl}{c}\right)\TM(M_1)=
\mathop{{\sum}^*}_{a \Mod{c}}e\left(\frac{-an^2-\overline{4a}l}{c}\right)\epsilon_{c}\left(\frac{a}{c}\right)=\\=
\bar{\epsilon}_{c}\mathop{{\sum}^*}_{d \Mod{c}}e\left(\frac{dn^2+\overline{4d}l}{c}\right)\left(\frac{d}{c}\right)
=-iS^{\Gamma_0(4)}_{1/1,\infty}(l,n^2;2c;\nu)e\left(\frac{l}{4}\right),
\end{multline}
where the last assertion follows from \eqref{Kl 4N 1/N infty}.
Applying \eqref{SZE j=1 eq2} and estimating the sum over $l$ in \eqref{SZE j=1 eq1} similarly to the previous case $j = 0$ (again by applying \eqref{S 1/1inf infinf est}), we obtain
\begin{multline}\label{SZE j=1 eq3}
\mathop{{\sum}^*}_{a \Mod{c}}e\left(\frac{-an^2}{c}\right)
\SZE^{(1)}(a,c,r,k,L)=\\=
\frac{1}{4ic}S^{\Gamma_0(4)}_{1/1,\infty}(0,n^2;2c;\nu)
\int_0^{\infty}g_1(-y)
\left(\log\frac{2y}{\pi (4c)^2}+\frac{\pi}{2}+3\gamma\right)\frac{dy}{\sqrt{y}}
+O\left(\frac{k^{\epsilon}c^{3/2+\epsilon}}{L}\right).
\end{multline}
It remains to perform the summation over $c|4rn/e_2$, $c\equiv\pm1\Mod{4}$:
\begin{multline}\label{SZE j=1 eq4}
\sum_{\substack{c|4rn/e_2\\c\equiv\pm1\Mod{4}}}
\mathop{{\sum}^*}_{a \Mod{c}}e\left(\frac{-an^2}{c}\right)
\SZE^{(1)}(a,c,r,k,L)=O\left(\frac{k^{\epsilon}(rn/e_2)^{3/2+\epsilon}}{L}\right)+\\+
\sum_{\substack{q|rn/e_2\\(q,2)=1}}\frac{1}{4iq}S^{\Gamma_0(4)}_{1/1,\infty}(0,n^2;2q;\nu)
\int_0^{\infty}g_1(-y)
\left(\log\frac{2y}{\pi (4q)^2}+\frac{\pi}{2}+3\gamma\right)\frac{dy}{\sqrt{y}}.
\end{multline}
%%%%%%%%%%%%%%%%%%%%%
Finally, consider the case $j = 2$ (that is, $c \equiv 2 \pmod{4}$). Applying \eqref{Thm.eq Voronoi an c2mod4} and Lemma~\ref{lem:Voronoi MT} to \eqref{SZE jcL eq2}, we obtain
\begin{multline}\label{SZE j=2 eq1}
\mathop{{\sum}^*}_{a \Mod{c}}e\left(\frac{-an^2}{c}\right)
\SZE^{(2)}(a,c,r,k,L)=\\=
\frac{\sqrt{2}}{2c}
\mathop{{\sum}^*}_{a \Mod{c}}e\left(\frac{-an^2}{c}\right)\left(\TM(M_4)^{-1}-\TM(M_5)\right)
\int_0^{\infty}g_1(-y)
\left(\log\frac{2y}{\pi (2c)^2}+\frac{\pi}{2}+3\gamma\right)\frac{dy}{\sqrt{2y}}+\\+
2^{1/2}\sum_{l\neq0}\widehat{g_1}\left(\frac{\pi^2l}{c^2}\right)
\frac{\Gamma\left(\frac{1}{2}-\frac{\sgn{l}}{4}\right)\Zag_{l}\left(\frac{1}{2}\right)}{\Gamma(3/4)\sqrt{|l|}}
\sum_{\substack{a\Mod{c}\\8ad_4\equiv-1\Mod{c_1}}}e\left(\frac{d_4l}{c_1}-\frac{an^2}{c}\right)\TM^{-1}(M_4)-\\-
2^{-1/2}\sum_{l\neq0}\widehat{g_1}\left(\frac{4\pi^2l}{c^2}\right)
\frac{\Gamma\left(\frac{1}{2}-\frac{\sgn{l}}{4}\right)\Zag_{4l}\left(\frac{1}{2}\right)}{\Gamma(3/4)\sqrt{|l|}}
\sum_{\substack{a\Mod{c}\\2ad_5\equiv-1\Mod{c_1}}}e\left(\frac{d_5l}{c_1}-\frac{an^2}{c}\right)\TM(M_5),
\end{multline}
where $c_1=c/2$, $\TM^{-1}(M_4)=\epsilon_{c_1}\left(\frac{8a}{c_1}\right)$ and $\TM(M_5)=\bar{\epsilon}_{c_1}\left(\frac{4d_5}{c_1}\right)$.
In the sum over $a \pmod{2c_1}$, we perform the change of variables $a = 2a_1 + a_2 c_1$, obtaining
\begin{multline}\label{SZE j=2 eq2}
\sum_{\substack{a\Mod{c}\\8ad_4\equiv-1\Mod{c_1}}}e\left(\frac{d_4l}{c_1}-\frac{an^2}{c}\right)\TM^{-1}(M_4)=
\epsilon_{c_1}\sum_{\substack{a\Mod{2c_1}\\8ad\equiv-1\Mod{c_1}}}e\left(\frac{dl}{c_1}-\frac{an^2}{c}\right)\left(\frac{8a}{c_1}\right)=\\=
\epsilon_{c_1}\sum_{\substack{a_1\Mod{c_1}\\16a_1d\equiv-1\Mod{c_1}}}\mathop{{\sum}^*}_{a_2 \Mod{2}}
e\left(\frac{dl}{c_1}-\frac{a_1n^2}{c_1}-\frac{a_2n^2}{2}\right)\left(\frac{16a_1}{c_1}\right)=\\=
\epsilon_{c_1}e\left(-\frac{n^2}{2}\right)\mathop{{\sum}^*}_{d \Mod{c_1}}
e\left(\frac{dl+\overline{16d}n^2}{c_1}\right)\left(\frac{-d}{c_1}\right)=\\=
\bar{\epsilon}_{c_1}e\left(-\frac{n^2}{2}\right)\mathop{{\sum}^*}_{d \Mod{c_1}}
e\left(\frac{dl\overline{4}_{c_1}+\overline{4d}n^2}{c_1}\right)\left(\frac{d}{c_1}\right)=
-ie\left(-\frac{n^2}{4}\right)S^{\Gamma_0(4)}_{1/1,\infty}(n^2,\overline{4}_{c_1}l;2c_1;\nu),
\end{multline}
where the last assertion follows from \eqref{Kl 4N 1/N infty}.
For the sums in \eqref{SZE j=2 eq1} involving $\TM(M_5)$, we have
\begin{multline}\label{SZE j=2 eq3}
\sum_{\substack{a\Mod{c}\\2ad_5\equiv-1\Mod{c_1}}}e\left(\frac{d_5l}{c_1}-\frac{an^2}{c}\right)\TM(M_5)=
\bar{\epsilon}_{c_1}\sum_{\substack{a\Mod{2c_1}\\2ad\equiv-1\Mod{c_1}}}e\left(\frac{dl}{c_1}-\frac{an^2}{2c_1}\right)\left(\frac{4d}{c_1}\right)=\\=
\bar{\epsilon}_{c_1}\sum_{\substack{a_1\Mod{c_1}\\4a_1d\equiv-1\Mod{c_1}}}\mathop{{\sum}^*}_{a_2 \Mod{2}}
e\left(\frac{dl-a_1n^2}{c_1}-\frac{a_2n^2}{2}\right)\left(\frac{4d}{c_1}\right)=\\=
\bar{\epsilon}_{c_1}e\left(-\frac{n^2}{2}\right)\mathop{{\sum}^*}_{d \Mod{c_1}}\left(\frac{d}{c_1}\right)
e\left(\frac{dl+\overline{4d}n^2}{c_1}\right)=
-ie\left(-\frac{n^2}{4}\right)S^{\Gamma_0(4)}_{1/1,\infty}(n^2,l;2c_1;\nu).
\end{multline}
It follows from \eqref{SZE j=2 eq2} and \eqref{SZE j=2 eq3} that the main term in \eqref{SZE j=2 eq1} vanishes, i.e.,
\begin{multline*}
\mathop{{\sum}^*}_{a \Mod{c}}e\left(\frac{-an^2}{c}\right)\left(\TM(M_4)^{-1}-\TM(M_5)\right)=
-ie\left(-\frac{n^2}{4}\right)S^{\Gamma_0(4)}_{1/1,\infty}(n^2,0;2c_1;\nu)+\\
+ie\left(-\frac{n^2}{4}\right)S^{\Gamma_0(4)}_{1/1,\infty}(n^2,0;2c_1;\nu)=0.
\end{multline*}
Therefore, applying \eqref{SZE j=2 eq2} and \eqref{SZE j=2 eq3}, and estimating the sums over $l$ in \eqref{SZE j=2 eq1} similarly to the previous case $j = 0$ (again by applying \eqref{S 1/1inf infinf est}), we obtain
\begin{equation*}
\mathop{{\sum}^*}_{a \Mod{c}}e\left(\frac{-an^2}{c}\right)
\SZE^{(2)}(a,c,r,k,L)\ll\frac{k^{\epsilon}c^{3/2+\epsilon}}{L}.
\end{equation*}
Performing the summation over $c \mid (4rn / e_2)$ with $c \equiv 2 \pmod{4}$ yields
\begin{equation}\label{SZE j=2 eq5}
\sum_{\substack{c|4rn/e_2\\c\equiv2\Mod{4}}}
\mathop{{\sum}^*}_{a \Mod{c}}e\left(\frac{-an^2}{c}\right)
\SZE^{(2)}(a,c,r,k,L)\ll\frac{k^{\epsilon}(rn/e_2)^{3/2+\epsilon}}{L}.
\end{equation}
Combining \eqref{SZE j=0 eq6}, \eqref{SZE j=1 eq4}, and \eqref{SZE j=2 eq5}, we establish \eqref{SZE sumj MT+ET}.

\end{proof}
%%%%%%%%%%%%%%%%%%%%%%%%%%%%%%%%%%%%%%%%%%%%%%%%%%%%%5
The main term in \eqref{SZE sumj MT+ET} can be rewritten in terms of the number of solutions to the quadratic congruences \eqref{rho upsilon def}, as we now show.

\begin{lem}\label{lem:SZE sumc MT+ET}
The following equality holds:
\begin{multline}\label{SZE sumc MT+ET}
\sum_{j=0}^2
\sum_{\substack{c|4rn/e_2\\c\equiv \pm j\Mod{4}}}\mathop{{\sum}^*}_{a \Mod{c}}
e\left(\frac{-an^2}{c}\right)
\SZE^{(j)}(a,c,r,k,L)=\\=
\int_0^{\infty}g_1(-y)
\sum_{q|rn/e_2}\frac{\ups_q(n^2)}{\sqrt{q}}
\left(\log\frac{2y}{\pi (4q)^2}+\frac{\pi}{2}+3\gamma\right)\frac{dy}{2\sqrt{y}}
+O\left(\frac{k^{\epsilon}(rn/e_2)^{3/2+\epsilon}}{L}\right),
\end{multline}
where $\ups_q(n^2)$ is defined in \eqref{rho upsilon def}.
\end{lem}
\begin{proof}
Using \eqref{Kl 4N infty infty 0n4q n01 q even} and \eqref{rho upsilon def}, we obtain
\begin{multline}\label{SZE Kl+KL q even eq1}
\sum_{\substack{q|rn/e_2\\(q,2)=2}}\frac{(1-i)}{4q}S^{\Gamma_0(4)}_{\infty,\infty}(0,n^2;4q;\nu)=
\sum_{\substack{q|rn/e_2\\(q,2)=2}}\frac{1}{\sqrt{q}}
\sum_{d|q}\mu(d)\sum_{t\Mod{ 2q/d}}\delta_{4q/d}(t^2-n^2)=\\=
\sum_{\substack{q|rn/e_2\\(q,2)=2}}\frac{\ups_q(n^2)}{\sqrt{q}}.
\end{multline}
Applying  \eqref{Kl 4N 1/Ninfty 0n} and \eqref{Kl 4N infty infty 0n4q n01 q odd} yields:
\begin{multline*}
\sum_{\substack{q|rn/e_2\\(q,2)=1}}\left(
\frac{(1-i)}{4q}S^{\Gamma_0(4)}_{\infty,\infty}(0,n^2;4q;\nu)-
\frac{i}{2q}S^{\Gamma_0(4)}_{1/1,\infty}(0,n^2;2q;\nu)\right)=\\=
\sum_{\substack{q|rn/e_2\\(q,2)=1}}
\frac{1}{\sqrt{q}}\sum_{d|q}\mu(d)\sum_{t\Mod{ q/d}}\delta_{q/d}(t^2-n^2).
\end{multline*}
To rewrite the last expression in terms of $\ups_q(n^2)$, we apply the identity

\begin{equation}\label{sum b delta b=sum 2b delta 4b}
\sum_{t\Mod{ b}}\delta_{b}(t^2-n)=\sum_{t\Mod{ 2b}}\delta_{4b}(t^2-n),
\end{equation}
which holds (see \cite[(4.7)]{BBF2025}) for  odd $b$ and $n\equiv0,1\Mod{4}$. Therefore,
\begin{equation}\label{SZE Kl+KL q odd eq2}
\sum_{\substack{q|rn/e_2\\(q,2)=1}}\left(
\frac{(1-i)}{4q}S^{\Gamma_0(4)}_{\infty,\infty}(0,n^2;4q;\nu)-
\frac{i}{2q}S^{\Gamma_0(4)}_{1/1,\infty}(0,n^2;2q;\nu)\right)=
\sum_{\substack{q|rn/e_2\\(q,2)=1}}\frac{\ups_q(n^2)}{\sqrt{q}}.
\end{equation}
Finally, \eqref{SZE sumc MT+ET} follows from \eqref{SZE sumj MT+ET}, \eqref{SZE Kl+KL q even eq1}, and \eqref{SZE Kl+KL q odd eq2}.
\end{proof}
%%%%%%%%%%%%%%%%%%%%%%%%%%%%%%%%%%%%%%%%%%%%%%%%%%%%%5
Using Lemma~\ref{lem:SZE sumc MT+ET}, we get the following expression for $\SZE(r,K)$ from \eqref{SZE jL def}.

\begin{lem}\label{lem:SZE rK MT+ET}
We have
\begin{multline}\label{SZE rK MT+ET}
\SZE(r,K)=
\sum_{k}h\left(\frac{k}{K}\right)\sum_{e|r^2}\frac{e_2}{\sqrt{re}}\sum_{n\ll rK^{-1+\epsilon}e^{-1}}
\sum_{q|rn/e_2}\frac{\ups_q(n^2)}{\sqrt{q}}\\\times
\sum_{n^2K^{2-\epsilon}\ll L\ll nrK^{1+\epsilon}/e}
\int_0^{\infty}g_1(-y)\left(\log\frac{2y}{\pi (4q)^2}+\frac{\pi}{2}+3\gamma\right)\frac{dy}{\sqrt{y}}
+O\left(\frac{r^{3/2}}{K^{3/2-\epsilon}}\right),
\end{multline}
where $e=e_1e_2^2$, with $e_1$ being square-free and $e_1e_2 \mid r$.
\end{lem}
\begin{proof}
It follows from \eqref{SZE jL def}, \eqref{SZE jL eq2} and \eqref{SZE sumc MT+ET} that
\begin{multline}\label{SZE rK eq1}
\SZE(r,K)=\sum_{k}h\left(\frac{k}{K}\right)\sum_{e|r^2}\frac{e_2}{\sqrt{re}}\sum_{n\ll rK^{-1+\epsilon}e^{-1}}
\sum_{q|rn/e_2}\sum_{n^2K^{2-\epsilon}\ll L\ll nrK^{1+\epsilon}/e}
\frac{\ups_q(n^2)}{\sqrt{q}}\\\times
\int_0^{\infty}\left(\log\frac{2y}{\pi (4q)^2}+\frac{\pi}{2}+3\gamma\right)\frac{dy}{\sqrt{y}}
+O\left(\sum_{e|r^2}\frac{e_2}{\sqrt{re}}\sum_{n\ll rK^{-1+\epsilon}e^{-1}}\frac{(rn/e_2)^{3/2+\epsilon}}{n^2K^{1-\epsilon}}\right).
\end{multline}
The error term can be estimated as
\begin{equation}\label{SZE rK eq2}
\sum_{e|r^2}\frac{e_2}{\sqrt{re}}\sum_{n\ll rK^{-1+\epsilon}e^{-1}}\frac{(rn/e_2)^{3/2+\epsilon}}{n^2K^{1-\epsilon}}\ll
\sum_{e|r^2}\frac{r^{3/2}}{K^{3/2-\epsilon}e \sqrt{e_2}}\ll\frac{r^{3/2}}{K^{3/2-\epsilon}},
\end{equation}
which yields the desired estimate for the lemma.
\end{proof}
%%%%%%%%%%%%%%%%%%%%%%%%%%%%%%%%%%%%%%%%%%%%%%%%%%%%%5

%%%%%%%%%%%%%%%%%%%%%%%%%%%%%%%%%%%%%%%%%%%%%%%%%%%%%%%%%%%%%%%%%%%%%%%%%%%%%%%%%%%%%%%%%%%%%%%%%%%%%%%%%%%%%%%%%%%%%%%%%%
\section{Analysis of $\SIN(r,K)$}\label{sec: SIN}
It follows from \eqref{approx.fun.eq.Vest} that, up to a negligible error, the sum over $m$ in \eqref{2mom SIN def} can be truncated to $m \ll K^{1+\varepsilon} / (e_1 e_2)$. To simplify the subsequent discussion, let $l = mr / e_2$. We next show that the contribution of $n \gg K^{\varepsilon-2} l$ to \eqref{2mom SIN def} is  small. To this end, we apply \eqref{Psik LG} to $\Psi_k(x)$ in \eqref{2mom SIN def} by setting
\begin{equation}\label{Psik xi}
\cosh\frac{\sqrt{\xi}}{2}=1+\frac{n}{2l}\Rightarrow\sinh\frac{\sqrt{\xi}}{4}=\sqrt{\frac{n}{4l}}\Rightarrow \xi=16\arcsinh^2\sqrt{\frac{n}{4l}}.
\end{equation}
Therefore, for $n \gg l K^{\varepsilon-2}$, we have $K \sqrt{\xi} \gg K^{\varepsilon}$, which implies that the $K$-Bessel function decays exponentially in \eqref{ZKxi}. Consequently, up to a negligible error, \eqref{2mom SIN def} can be rewritten as
\begin{multline}\label{SIN eq2}
\SIN(r,K)=2\sum_{k}h\left(\frac{k}{K}\right)\sum_{e|r^2}\sum_{m\ll K^{1+\epsilon}(e_1e_2)^{-1}}\frac{V_k(me_1e_2)e_2}{mr\sqrt{2me_1e_2}}\times\\
\sum_{1\leq n\ll K^{-2+\epsilon}mr/e_2}\Zag_{n(4mr/e_2+n)}(1/2)\sqrt{n+2mr/e_2}\Psi_k\left(\left(1+\frac{ne_2}{2mr}\right)^{-2}\right).
\end{multline}
Proceeding as in Section~\ref{sec: SZE}, we first obtain
\begin{multline}\label{SIN eq3}
\SIN(r,K)=8\sum_{k}h\left(\frac{k}{K}\right)\sum_{e|r^2}\frac{\sqrt{r}}{\sqrt{e}}\sum_{n\ll rK^{-1+\epsilon}e^{-1}}\sum_{\substack{nK^{2-\epsilon\ll q\ll rK^{1+\epsilon}/e}\\q\equiv n\Mod{4r/e_2}}}V_k\left(\frac{(q-n)e}{4r}\right)\times\\
\frac{\Zag_{qn}(1/2)\sqrt{q+n}}{(q-n)^{3/2}}\Psi_k\left(\left(1+\frac{2n}{q-n}\right)^{-2}\right),
\end{multline}
and then
\begin{multline}\label{SIN eq4}
\SIN(r,K)=2\sum_{k}h\left(\frac{k}{K}\right)\sum_{e|r^2}\frac{e_2}{\sqrt{re}}\sum_{n\ll rK^{-1+\epsilon}e^{-1}}\sum_{c|4rn/e_2}\mathop{{\sum}^*}_{a \Mod{c}}
e\left(\frac{-an^2}{c}\right)\times\\
\sum_{n^2K^{2-\epsilon\ll l\ll nrK^{1+\epsilon}/e}}e\left(\frac{al}{c}\right)V_k\left(\frac{(l-n^2)e}{4rn}\right)
\frac{\Zag_{l}(1/2)\sqrt{l+n^2}}{(l-n^2)^{3/2}}\Psi_k\left(\frac{(l-n^2)^2}{(l+n^2)^2}\right).
\end{multline}
Splitting the analysis into the three cases given in \eqref{c cases}, we find that
\begin{equation}\label{SIN jL def}
\SIN(r,K)=\sum_{j=0}^2\sum_{n^2K^{2-\epsilon}\ll L\ll nrK^{1+\epsilon}/e}\SIN^{(j)}(r,k,L),
\end{equation}
\begin{multline}\label{SIN jL eq1}
\SIN^{(j)}(r,k,L)=2\sum_{k}h\left(\frac{k}{K}\right)\sum_{e|r^2}\frac{e_2}{\sqrt{re}}\sum_{n\ll rK^{-1+\epsilon}e^{-1}}
\sum_{\substack{c|4rn/e_2\\c\equiv \pm j\Mod{4}}}\mathop{{\sum}^*}_{a \Mod{c}}
e\left(\frac{-an^2}{c}\right)\\\times
\SIN^{(j)}(a,c,r,k,L),
\end{multline}
\begin{equation}\label{SIN jcL eq1}
\SIN^{(j)}(a,c,r,k,L)=
\sum_{l}\frac{\Zag_{l}(1/2)}{\sqrt{l}}g_2(l)e\left(\frac{al}{c}\right),
\end{equation}
\begin{equation}\label{SIN g2 def}
g_2(l)=
V_k\left(\frac{(l-n^2)e}{4rn}\right)U\left(\frac{l}{L}\right)
\frac{\sqrt{l(l+n^2)}}{(l-n^2)^{3/2}}\Psi_k\left(\frac{(l-n^2)^2}{(l+n^2)^2}\right).
\end{equation}

%%%%%%%%%%%%%%%%%%%%%%%%%%%%%%%%%%%%%%%%%%%%%%%%%%%%%5
%%%%%%%%%%%%%%%%%%%%%%%%%%%%%%%%%%%%%%%%%%%%%%%%%%%%%5
%%%%%%%%%%%%%%%%%%%%%%%%%%%%%%%%%%%%%%%%%%%%%%%%%%%%%%%%
The following lemma is a direct analogue of Lemma~\ref{lem:g1 est}.
\begin{lem}\label{lem:g2 est}
For $m \in \mathbb{Z}$, $L \gg n^2 K^{2-\varepsilon}$,  $\alpha \in \{\pi^2, 4\pi^2\}$, and any $A > 1$, we have
\begin{equation}\label{g2 est1}
\widehat{g_2}\left(\frac{\alpha m}{c^2}\right)\ll\frac{1}{(L|m|/c^2)^{A}} \quad\hbox{if}\quad \frac{L|m|}{c^2}\gg k^{\epsilon},
\end{equation}
\begin{equation}\label{g2 est2}
\widehat{g_2}\left(\frac{\alpha m}{c^2}\right)\ll k^{\epsilon}\sqrt{\frac{|m|}{c^2}}.
\end{equation}
\end{lem}
\begin{proof}
Let $y=\alpha m/c^2.$ It follows from  \eqref{SIN g2 def}, \eqref{phi hat+ to Phipm def0}, \eqref{phi hat- to Phipm def0}, \eqref{Phi++--def} and \eqref{Phi+--+def} that
\begin{equation}\label{g2 hat eq1}
\widehat{g_2}(y)\ll\int_0^{\infty}\frac{g_2(x)}{x}\sqrt{x|y|}B_{0}(2\sqrt{x|y|})dx,
\end{equation}
where $B_0$ represents one of the Bessel functions $Y_0$, $J_0$, or $K_0$. The substitution of \eqref{SIN g2 def} into \eqref{g2 hat eq1} yields
\begin{equation}\label{g2 hat eq2}
\widehat{g_2}(y)\ll\sqrt{|y|}\int_0^{\infty}
V_k\left(\frac{(x-n^2)e}{4rn}\right)U\left(\frac{x}{L}\right)
\Psi_k\left(\frac{(x-n^2)^2}{(x+n^2)^2}\right)
B_{0}(2\sqrt{x|y|})\frac{\sqrt{x+n^2}}{(x-n^2)^{3/2}}dx,
\end{equation}
which implies that
\begin{equation}\label{g2 hat eq3}
\widehat{g_2}(y)\ll \sqrt{|y|}\int_0^{\infty}
V_k\left(\frac{(xL-n^2)e}{4rn}\right)U\left(x\right)
\Psi_k\left(\frac{(x-n^2/L)^2}{(x+n^2/L)^2}\right)
B_{0}(2\sqrt{xL|y|})\frac{\sqrt{x+n^2/L}}{(x-n^2/L)^{3/2}}dx.
\end{equation}
Estimating the integral trivially by taking the absolute value of the integrand immediately yields the estimate \eqref{g2 est2}. To prove \eqref{g2 est1}, we proceed as in the proof of Lemma~\ref{lem:g1 est}. The only difference is that, by \cite[Section~2.7.2, Equations~(7)–(9)]{BE}, the function
\begin{equation}
\label{Y2_def}
Y_k(x) = (1-x)^{1/2} \Psi_k(x)
\end{equation}
satisfies the differential equation
\begin{equation}\label{Y2 dif eq}
Y_k''+\left(\frac{1-(2k-1)^2}{4x^2}+\frac{1}{4(1-x)^2}+\frac{5/4-(2k-1)^2}{4x(1-x)}\right)Y_k=0.
\end{equation}

However, this equation differs only slightly from \eqref{Y dif eq}. For example, we still have $Y_k''(x)\ll k^{\epsilon}Y_k(x)/(1-x)^2$ for $1-x \ll k^{\varepsilon-2}$. Using \eqref{Y2_def}, we rewrite \eqref{g2 hat eq3} as
\begin{multline}\label{g2 hat eq4}
\widehat{g_2}(y)\ll \sqrt{\frac{L|y|}{n^2}}\int_0^{\infty}
V_k\left(\frac{(xL-n^2)e}{4rn}\right)U\left(x\right)
Y_k\left(\frac{(x-n^2/L)^2}
{(x+n^2/L)^2}\right)\\\times\left(\frac{x+n^2/L}{x-n^2/L}\right)^{3/2}
B_{0}(2\sqrt{xL|y|})\frac{dx}{\sqrt{x}}.
\end{multline}
This expression is a direct analogue of \eqref{g1 hat eq4}, and it suffices to repeat the subsequent calculations from the proof of Lemma~\ref{lem:g1 est}.
\end{proof}
%%%%%%%%%%%%%%%%%%%%%%%%%%%%%%%%%%%%%%%%%%%%%%%%%%%%%5
Next, proceeding as in Section~\ref{sec: SZE}, with the sole difference that the second integral in \eqref{Voronoi MTeq0} now vanishes (requiring $+\pi/2$ under the integral to be replaced by $-\pi/2$), we obtain the following analogue of Lemma~\ref{lem:SZE rK MT+ET}.
\begin{lem}\label{lem:SIN rK MT+ET}
We have
\begin{multline}\label{SIN rK MT+ET}
\SIN(r,K)=
\sum_{k}h\left(\frac{k}{K}\right)\sum_{e|r^2}\frac{e_2}{\sqrt{re}}\sum_{n\ll rK^{-1+\epsilon}e^{-1}}
\sum_{q|rn/e_2}\frac{\ups_q(n^2)}{\sqrt{q}}\\\times
\sum_{n^2K^{2-\epsilon}\ll L\ll nrK^{1+\epsilon}/e}
\int_0^{\infty}g_2(y)\left(\log\frac{2y}{\pi (4q)^2}-\frac{\pi}{2}+3\gamma\right)\frac{dy}{\sqrt{y}}
+O\left(\frac{r^{3/2}}{K^{3/2-\epsilon}}\right),
\end{multline}
where $e=e_1e_2^2$ with $e_1$ being square-free and  $e_1e_2|r$.
\end{lem}

%%%%%%%%%%%%%%%%%%%%%%%%%%%%%%%%%%%%%%%%%%%%%%%%%%%%%%%%%%%%%%%%%%%%%%%%%%%%%%%%%%%%%%%%%%%%%%%%%%%%%%%%%%%%%%%%%%%%%%%%%%
\section{The off-diagonal terms}\label{sec:Voronoi MT}
Consider the integrals (see \eqref{SZE g1 def}, \eqref{SZE rK MT+ET})
\begin{multline}\label{I<Ldef}
I_{<}(k,L,r,e,n)=\int_0^{\infty}g_1(-y)\left(\log\frac{y}{8\pi q^2}+\frac{\pi}{2}+3\gamma\right)\frac{dy}{\sqrt{y}}=\\=
\int_0^{\infty}\left(\log\frac{y}{8\pi q^2}+\frac{\pi}{2}+3\gamma\right)
V_k\left(\frac{(y+n^2)e}{4rn}\right)
\frac{U\left(\frac{y}{L}\right)}{y+n^2}\Phi_k\left(\frac{(y-n^2)^2}{(y+n^2)^2}\right)dy
\end{multline}
and (see \eqref{SIN g2 def}, \eqref{SIN rK MT+ET})
\begin{multline}\label{I>Ldef}
I_{>}(k,L,r,e,n)=\int_0^{\infty}g_2(y)\left(\log\frac{y}{8\pi q^2}-\frac{\pi}{2}+3\gamma\right)\frac{dy}{\sqrt{y}}=\\=
\int_0^{\infty}\left(\log\frac{y}{8\pi q^2}-\frac{\pi}{2}+3\gamma\right)
V_k\left(\frac{(y-n^2)e}{4rn}\right)U\left(\frac{y}{L}\right)
\frac{\sqrt{y+n^2}}{(y-n^2)^{3/2}}\Psi_k\left(\frac{(y-n^2)^2}{(y+n^2)^2}\right)dy.
\end{multline}
It follows from \eqref{SZE rK MT+ET} and \eqref{SIN rK MT+ET} that
\begin{multline}\label{SZE+SIN rK MT+ET}
\SZE(r,K)+\SIN(r,K)=
\sum_{k}h\left(\frac{k}{K}\right)\sum_{e|r^2}\frac{e_2}{\sqrt{re}}\sum_{n\ll rK^{-1+\epsilon}e^{-1}}
\sum_{q|rn/e_2}\frac{\ups_q(n^2)}{\sqrt{q}}\\\times
\sum_{n^2K^{2-\epsilon}\ll L\ll nrK^{1+\epsilon}/e}\left(I_{<}(k,L,r,e,n)+I_{>}(k,L,r,e,n)\right)
+O\left(\frac{r^{3/2}}{K^{3/2-\epsilon}}\right).
\end{multline}
Note that, estimating $I_{<}(k,L,r,e,n) + I_{>}(k,L,r,e,n)$ trivially by $K^{\varepsilon}$, we obtain the estimate
\begin{multline}\label{SZE+SIN rK MT+ET triv}
\SZE(r,K)+\SIN(r,K)\ll\frac{r^{3/2}}{K^{3/2-\epsilon}}+
K^{1+\epsilon}\sum_{e|r^2}\frac{e_2}{\sqrt{re}}\sum_{n\ll rK^{-1+\epsilon}e^{-1}}
\sum_{q|rn/e_2}\frac{|\ups_q(n^2)|}{\sqrt{q}}\ll\\\ll
\frac{r^{3/2}}{K^{3/2-\epsilon}}+K^{\epsilon}\sum_{e|r^2}\frac{e_2\sqrt{r}}{e^{3/2}}\ll\frac{r^{3/2}}{K^{3/2-\epsilon}}+
r^{1/2}K^{\epsilon},
\end{multline}
which is insufficient since it is smaller than the main term only when $r \ll K^{1-\varepsilon}$. For the further evaluation of these integrals, it is convenient to replace $\log y$ by $\log(y+n^2)$ in \eqref{I<Ldef} and by $\log(y-n^2)$ in \eqref{I>Ldef}.
Since
\begin{equation*}
\log(1\pm n^2/y)\ll n^2/y\ll n^2/L\ll K^{-2+\epsilon}
\end{equation*}
and in view of \eqref{SZE+SIN rK MT+ET triv}, we find that the error introduced by this approximation is $O(\frac{r^{1/2}}{K^{2-\epsilon}})$, which is smaller than the error term in \eqref{SZE+SIN rK MT+ET}.
Consider
\begin{equation}\label{sum I<+I>eq1}
\sum_{k}h\left(\frac{k}{K}\right)\sum_{n^2K^{2-\epsilon}\ll L\ll nrK^{1+\epsilon}/e}\left(I_{<}(k,L,r,e,n)+I_{>}(k,L,r,e,n)\right).
\end{equation}
We now note that the summation over $n$ can be extended to all $n > 0$, since the sum over $L$ is empty for $n \gg r K^{-1+\epsilon} e^{-1}$. Furthermore, up to a negligible error, the summation over $L$ can be extended to all $L \gg n^2 K^{2-\epsilon}$. Indeed, for $L \gg n r K^{1+\varepsilon} / e$, relation \eqref{approx.fun.eq.Vest} implies that $V_k\left(\frac{(y \pm n^2)e}{4rn}\right) \ll K^{-A}$. Similarly, for $L \ll n^2 K^{2-\epsilon}$, we have $\Psi_k\left(\frac{(y-n^2)^2}{(y+n^2)^2}\right) \ll K^{-A}$, which follows from \eqref{Psik LG} by setting

\begin{equation*}
\cosh\frac{\sqrt{\xi}}{2}=1+\frac{2n^2}{y-n^2}\Rightarrow\sinh\frac{\sqrt{\xi}}{4}=\sqrt{\frac{n^2}{y-n^2}}\Rightarrow \xi=16\arcsinh^2\sqrt{\frac{n^2}{y-n^2}}.
\end{equation*}

Therefore, for $y \ll n^2 K^{2-\epsilon}$, we have $K \sqrt{\xi} \gg Kn / \sqrt{y} \gg Kn / \sqrt{L} \gg K^{\epsilon}$, which implies that the $K$-Bessel function decays exponentially in \eqref{ZKxi}. Proceeding as in the beginning of Section~\ref{sec: SZE}, we can show that
\begin{equation*}
\sum_{k}h\left(\frac{k}{K}\right)I_{<}(k,L)
\end{equation*}
is negligible for $L \ll n^2 K^{2-\epsilon}$.
 Consequently, \eqref{SZE+SIN rK MT+ET} can be rewritten as
\begin{multline}\label{SZE+SIN rK MT+ET2}
\SZE(r,K)+\SIN(r,K)=
\sum_{k}h\left(\frac{k}{K}\right)\sum_{e|r^2}\frac{e_2}{\sqrt{re}}\sum_{n=1}^{\infty}
\sum_{q|rn/e_2}\frac{\ups_q(n^2)}{\sqrt{q}}\\\times
\left(I_{<}(k,n,e)+I_{>}(k,n,e)\right)
+O\left(\frac{r^{3/2}}{K^{3/2-\epsilon}}\right),
\end{multline}
where
\begin{equation}\label{I<def}
I_{<}(k,n,e)=
\int_{n^2}^{\infty}
\left(\log\frac{y+n^2}{8\pi q^2}+\frac{\pi}{2}+3\gamma\right)
V_k\left(\frac{(y+n^2)e}{4rn}\right)
\Phi_k\left(\frac{(y-n^2)^2}{(y+n^2)^2}\right)\frac{dy}{(y+n^2)},
\end{equation}
\begin{equation}\label{I>def}
I_{>}(k,n,e)=
\int_{n^2}^{\infty}\left(\log\frac{y-n^2}{8\pi q^2}-\frac{\pi}{2}+3\gamma\right)
V_k\left(\frac{(y-n^2)e}{4rn}\right)
\frac{\sqrt{y+n^2}}{(y-n^2)^{3/2}}\Psi_k\left(\frac{(y-n^2)^2}{(y+n^2)^2}\right)dy.
\end{equation}
Consider (see \cite[Lemma 5.1]{BF2018})
\begin{equation}\label{Ik def}
I_k(x):=\frac{1}{2\pi i}\int_{(\Delta)}\frac{\Gamma(k-1/2+w/2)}{\Gamma(k+1/2-w/2)}\Gamma(1/2-w)\sin\left( \pi \frac{1/2+w}{2}\right)x^wdw
\end{equation}
with $1-2k<\Delta<1/2.$ It follows from \cite[Lemma 5.2, 5.4]{BF2018} that
\begin{equation}\label{PhikPsik toIk}
\Psi_k(y)=(-1)^k\sqrt{\pi y}I_k\left(\frac{2}{\sqrt{y}}\right),\quad
\Phi_k(y)=(-1)^k\sqrt{\pi}\frac{I_k(2\sqrt{y})}{y^{1/4}}.
\end{equation}
According to \eqref{I<def}, \eqref{I>def}, and \eqref{PhikPsik toIk}, we have
\begin{equation}\label{I< eq1}
I_{<}(k,n,e)=(-1)^k\sqrt{\pi}
\int_{n^2}^{\infty}
V_k\left(\frac{(y+n^2)e}{4rn}\right)I_k\left(2\frac{y-n^2}{y+n^2}\right)
\frac{\left(\log\frac{y+n^2}{8\pi q^2}+\frac{\pi}{2}+3\gamma\right)dy}{(y-n^2)^{1/2}(y+n^2)^{1/2}},
\end{equation}
\begin{equation}\label{I> eq1}
I_{>}(k,n,e)=(-1)^k\sqrt{\pi}
\int_{n^2}^{\infty}
V_k\left(\frac{(y-n^2)e}{4rn}\right)
I_k\left(2\frac{y+n^2}{y-n^2}\right)\frac{\left(\log\frac{y-n^2}{8\pi q^2}-\frac{\pi}{2}+3\gamma\right)dy}{(y-n^2)^{1/2}(y+n^2)^{1/2}}.
\end{equation}
Introducing the change of variable $y = n^2 x$ and utilizing the integral representation \eqref{approx.fun.eq.Vdef} for $V_k(x)$, we find that
\begin{equation}\label{I<k to I<z}
I_{<}(k,n,e)=
\frac{1}{2\pi i}\int_{(a)}\frac{L_{\infty}(\sym^2f,1/2+z)}{L_{\infty}(\sym^2f,1/2)}\zeta(1+2z)
\left(\frac{ne}{4r}\right)^{-z}I_{<}(k,z)\frac{dz}{z},
\end{equation}
\begin{equation}\label{I<zdef}
I_{<}(k,z)=(-1)^k\sqrt{\pi}
\int_{1}^{\infty}I_k\left(2\frac{x-1}{x+1}\right)\frac{\left(\log\frac{n^2(x+1)}{8\pi q^2}+\frac{\pi}{2}+3\gamma\right)dx}{(x-1)^{1/2}(x+1)^{1/2+z}},
\end{equation}
\begin{equation}\label{I>k to I>z}
I_{>}(k,n,e)=
\frac{1}{2\pi i}\int_{(a)}\frac{L_{\infty}(\sym^2f,1/2+z)}{L_{\infty}(\sym^2f,1/2)}\zeta(1+2z)
\left(\frac{ne}{4r}\right)^{-z}I_{>}(k,z)\frac{dz}{z},
\end{equation}
\begin{equation}\label{I>zdef}
I_{>}(k,z)=(-1)^k\sqrt{\pi}
\int_{1}^{\infty}I_k\left(2\frac{x+1}{x-1}\right)\frac{\left(\log\frac{n^2(x-1)}{8\pi q^2}-\frac{\pi}{2}+3\gamma\right)dx}{(x-1)^{1/2+z}(x+1)^{1/2}}.
\end{equation}
For convenience, we introduce the following additional notation:
\begin{equation}\label{I<0zdef}
I_{<,0}(k,z):=(-1)^k\sqrt{\pi}
\int_{1}^{\infty}I_k\left(2\frac{x-1}{x+1}\right)\frac{dx}{(x-1)^{1/2}(x+1)^{1/2+z}},
\end{equation}
\begin{equation}\label{I>z0def}
I_{>,0}(k,z):=(-1)^k\sqrt{\pi}
\int_{1}^{\infty}I_k\left(2\frac{x+1}{x-1}\right)\frac{dx}{(x-1)^{1/2+z}(x+1)^{1/2}}.
\end{equation}
Then
\begin{equation}\label{I<z to I<z0}
I_{<}(k,z)=\left(\log\frac{n^2}{8\pi q^2}+\frac{\pi}{2}+3\gamma-\frac{\partial}{\partial z}\right)I_{<,0}(k,z),
\end{equation}
\begin{equation}\label{I>z to I>z0}
I_{>}(k,z)=\left(\log\frac{n^2}{8\pi q^2}-\frac{\pi}{2}+3\gamma-\frac{\partial}{\partial z}\right)I_{>,0}(k,z).
\end{equation}

%%%%%%%%%%%%%%%%%%%%%%%
To evaluate $I_{<,0}(k,z)$ and $I_{>,0}(k,z)$, we apply the following formula from \cite[Equation~5.12.3]{HMF}:
\begin{equation}\label{beta +0infinity}
\int_0^{\infty}x^c(1+x)^ddx=\frac{\Gamma(c+1)\Gamma(-1-c-d)}{\Gamma(-d)},
\end{equation}
which holds  for $\Re{d}<0$ and $-1<\Re{c}<-1-\Re{d}$. Relation \eqref{beta +0infinity} directly implies that
\begin{equation}\label{beta +0infinity 2}
\int_1^{\infty}(x+1)^a(x-1)^bdx=2^{1+a+b}\frac{\Gamma(1+b)\Gamma(-1-a-b)}{\Gamma(-a)},
\end{equation}
which holds  for $\Re{a}<0$ and $-1<\Re{b}<-1-\Re{a}$.
%%%%%%%%%%%%%%%%%%%%%%%%%%%%%%%%%%%%%%%%%%%%%%%%%%%%%5
\begin{lem}\label{lem:I>z}
For $0<\Re{z}<2k-1/2$ we have
\begin{equation}\label{I>z eq0}
I_{>,0}(k,z)=\frac{\Gamma^2(z)\Gamma(2k-1/2-z)}{\Gamma(2k-1/2+z)}.
\end{equation}
\end{lem}
\begin{proof}
Substituting \eqref{Ik def} into \eqref{I>z0def}, we obtain
\begin{multline}\label{I>z eq1}
I_{>,0}(k,z)=
\frac{(-1)^k\sqrt{\pi}}{2\pi i}\int_{(\Delta)}\frac{\Gamma(k-1/2+w/2)}{\Gamma(k+1/2-w/2)}\Gamma(1/2-w)\sin\left( \pi \frac{1/2+w}{2}\right)2^w\\\times
\int_{1}^{\infty}(x-1)^{-w-z-1/2}(x+1)^{w-1/2}dxdw.
\end{multline}
By virtue of \eqref{beta +0infinity 2}, we find that
\begin{equation}\label{I>z eq2}
\int_{1}^{\infty}(x+1)^{w-1/2}(x-1)^{-w-z-1/2}dx=
2^{-z}\frac{\Gamma(1/2-w-z)\Gamma(z)}{\Gamma(1/2-w)}
\end{equation}
for  $0<\Re{z}<1/2-\Re{w}.$ Substituting \eqref{I>z eq2} into \eqref{I>z eq1} yields
\begin{multline}\label{I>z eq3}
I_{>,0}(k,z)=
\frac{2^{-z}\Gamma(z)(-1)^k\sqrt{\pi}}{2\pi i}\int_{(\Delta)}\frac{\Gamma(k-1/2+w/2)}{\Gamma(k+1/2-w/2)}\Gamma(1/2-w-z)\sin\left( \pi \frac{1/2+w}{2}\right)2^wdw.
\end{multline}
Moving the line of integration in \eqref{I>z eq3} to the left, we pass the poles at $w = 1 - 2k - 2j$ for $j = 0, 1, 2, \dots$. Evaluating the residues, we find that
\begin{equation}\label{I>z eq4}
I_{>,0}(k,z)=
2^{2-z-2k}\Gamma(z)(-1)^k\sqrt{\pi}\sum_{j=0}^{\infty}\frac{(-1)^j}{j!2^{2j}}
\frac{\Gamma(2k+2j-z-1/2)}{\Gamma(2k+j)}\sin\left( \frac{3\pi}{4}-\pi(k+j)\right).
\end{equation}
Using \cite[5.5.5]{HMF}, \eqref{pIq def}, and \cite[ 15.4.20]{HMF}, we obtain
\begin{multline}\label{I>z eq5}
I_{>,0}(k,z)=
2^{-2z}\Gamma(z)\sum_{j=0}^{\infty}\frac{1}{j!}
\frac{\Gamma(k-z/2-1/4+j)\Gamma(k-z/2+1/4+j)}{\Gamma(2k+j)}=\\=
2^{-2z}\Gamma(z)\GenHyGI{2}{1}{k-z/2-1/4,k-z/2+1/4}{2k}{1}=\\=
2^{-2z}\frac{\Gamma^2(z)\Gamma(k-1/4-z/2)\Gamma(k+1/4-z/2)}{\Gamma(k-1/4+z/2)\Gamma(k+1/4+z/2)}.
%\frac{\Gamma(k-z/2-1/4)\Gamma(k-z/2+1/4)}{\Gamma(2k)}.
\end{multline}
By applying \cite[Equation~5.5.5]{HMF} twice to the numerator and denominator of the last fraction in \eqref{I>z eq5}, we find that
\begin{equation*}
I_{>,0}(k,z)=
\frac{\Gamma^2(z)\Gamma(2k-1/2-z)}{\Gamma(2k-1/2+z)}.
\end{equation*}
\end{proof}
%%%%%%%%%%%%%%%%%%%%%%%%%%%%%%%%%%%%%%%%%%%%%%%%%%%%%5
\begin{lem}\label{lem:I<z}
For $0<\Re{z}<2k-1/2$ we have
\begin{equation}\label{I<z eq01}
I_{<,0}(k,z)=
2^{-2z}\Gamma(z)
\GenHyGI{3}{2}{k-1/4,3/4-k,1}{1/2+z/2,1+z/2}{1},
\end{equation}
\begin{multline}\label{I<z eq02}
I_{<,0}(k,z)=
\frac{\Gamma^2(z)\Gamma(2k-1/2-z)\cos(\pi z)}{\Gamma(2k-1/2+z)}-\\-
2^{-2z}\Gamma(z)\sin(\pi z)
\GenHyGI{3}{2}{1-z/2,3/2-z/2,1}{k+5/4,9/4-k}{1},
\end{multline}
\end{lem}
\begin{proof}
Substituting \eqref{Ik def} to \eqref{I<zdef}, we obtain
\begin{multline}\label{I<z eq1}
I_{<,0}(k,z)=
\frac{(-1)^k\sqrt{\pi}}{2\pi i}\int_{(\Delta)}\frac{\Gamma(k-1/2+w/2)}{\Gamma(k+1/2-w/2)}\Gamma(1/2-w)\sin\left( \pi \frac{1/2+w}{2}\right)2^w\\\times
\int_{1}^{\infty}(x-1)^{w-1/2}(x+1)^{-1/2-w-z}dxdw.
\end{multline}
Applying \eqref{beta +0infinity 2}, we show that
\begin{equation}\label{I<z eq2}
\int_{1}^{\infty}(x+1)^{-1/2-w-z}(x-1)^{w-1/2}dx=
2^{-z}\frac{\Gamma(1/2+w)\Gamma(z)}{\Gamma(1/2+w+z)}
\end{equation}
for  $\Re{w}>-1/2,$ $\Re{z}>0$. Substitution of \eqref{I<z eq2} into \eqref{I<z eq1} yields
\begin{multline}\label{I<z eq3}
I_{<,0}(k,z)=\\=
\frac{2^{-z}\Gamma(z)(-1)^k\sqrt{\pi}}{2\pi i}\int_{(\Delta)}
\frac{\Gamma(k-1/2+w/2)}{\Gamma(k+1/2-w/2)}\frac{\Gamma(1/2+w)\Gamma(1/2-w)}{\Gamma(1/2+w+z)}\sin\left( \pi \frac{1/2+w}{2}\right)2^wdw.
\end{multline}
Moving the line of integration in \eqref{I<z eq3} to the right, we pass the poles at $w = 1/2 + j$ for $j = 0, 1, 2, \dots$. Evaluating the residues, we find that
\begin{multline}\label{I<z eq4}
I_{<,0}(k,z)=
2^{-z}\Gamma(z)(-1)^k\sqrt{\pi}\sum_{j=0}^{\infty}\frac{(-1)^j2^{1/2+j}}{j!}
\frac{\Gamma(k-1/4+j/2)\Gamma(1+j)}{\Gamma(k+1/4-j/2)\Gamma(1+z+j)}\cos\frac{\pi j}{2}=\\=
2^{-z}\Gamma(z)(-1)^k\sqrt{\pi}\sum_{m=0}^{\infty}\frac{(-1)^m2^{1/2+2m}}{\Gamma(1+z+2m)}\frac{\Gamma(k-1/4+m)}{\Gamma(k+1/4-m)}.
\end{multline}
The application of \cite[Equations~5.5.3 and 5.5.5]{HMF} together with \eqref{pIq def} yields
\begin{multline}\label{I<z eq5}
I_{<,0}(k,z)=
2^{-2z}\Gamma(z)\sum_{m=0}^{\infty}\frac{\Gamma(k-1/4+m)\Gamma(3/4-k+m)}{\Gamma(1/2+z/2+m)\Gamma(1+z/2+m)}=\\=
2^{-2z}\Gamma(z)\GenHyGI{3}{2}{k-1/4,3/4-k,1}{1/2+z/2,1+z/2}{1},
\end{multline}
which completes the proof of \eqref{I<z eq01}.

Next, moving the line of integration in \eqref{I<z eq3} to the left, we pass the poles at $w = 1 - 2k - 2j$ and $w = -1/2 - j$ for $j = 0, 1, 2, \dots$. Evaluating the residues, we find that
\begin{multline}\label{I<z eq6}
\frac{I_{<,0}(k,z)}{2^{-z}\Gamma(z)}=
2\sqrt{\pi}\sum_{j=0}^{\infty}\frac{1}{j!}\frac{\Gamma(-1/2+2k+2j)\Gamma(3/2-2k-2j)}{\Gamma(2k+j)\Gamma(3/2+z-2k-2j)}2^{1/2-2j-2k}+\\+
(-1)^k\sqrt{\pi}\sum_{j=0}^{\infty}\frac{(-1)^j2^{-1/2-j}}{j!}
\frac{\Gamma(k-3/4-j/2)\Gamma(1+j)}{\Gamma(k+3/4+j/2)\Gamma(z-j)}\sin\frac{-\pi j}{2}.
\end{multline}
Applying \cite[Equation~5.5.3]{HMF} twice to the first sum in \eqref{I<z eq6}, we obtain
\begin{multline}\label{I<z eq7}
\frac{I_{<,0}(k,z)}{2^{-z}\Gamma(z)}=
2\sqrt{\pi}\sum_{j=0}^{\infty}\frac{1}{j!}\frac{\Gamma(-1/2-z+2k+2j)\cos(\pi z)}{\Gamma(2k+j)}2^{1/2-2j-2k}+\\+
(-1)^k\sqrt{\pi}\sum_{m=0}^{\infty}
\frac{(-1)^m2^{-3/2-2m}\Gamma(k-5/4-m)}{\Gamma(k+5/4+m)\Gamma(z-1-2m)}.
\end{multline}
Application of \cite[Equation~5.5.5]{HMF} to the first sum and a twofold application of \cite[Equation~5.5.3]{HMF} to the second sum in \eqref{I<z eq7} yield
\begin{multline}\label{I<z eq8}
\frac{I_{<,0}(k,z)}{2^{-z}\Gamma(z)}=
2^{-z}\cos(\pi z)\sum_{j=0}^{\infty}\frac{\Gamma(k-1/4-z/2+j)\Gamma(k+1/4-z/2+j)}{\Gamma(2k+j)j!}-\\-
\sqrt{\pi}\sum_{m=0}^{\infty}
\frac{2^{-1-2m}\Gamma(2-z+2m)\sin(\pi z)}{\Gamma(k+5/4+m)\Gamma(9/4-k+m)}.
\end{multline}
By virtue of \eqref{pIq def} and applying \cite[Equation~5.5.5]{HMF} to the second sum, we obtain
\begin{multline}\label{I<z eq9}
\frac{I_{<,0}(k,z)}{2^{-z}\Gamma(z)}=
2^{-z}\cos(\pi z)\GenHyGI{2}{1}{k-1/4-z/2,k+1/4-z/2}{2k}{1}-\\-
2^{-z}\sin(\pi z)\sum_{m=0}^{\infty}
\frac{\Gamma(1-z/2+m)\Gamma(3/2-z/2+m)}{\Gamma(k+5/4+m)\Gamma(9/4-k+m)}.
\end{multline}
Combining \eqref{pIq def} with \cite[Equation~15.4.20]{HMF}, we find that
\begin{multline}\label{I<z eq10}
\frac{I_{<,0}(k,z)}{2^{-z}\Gamma(z)}=
2^{-z}\cos(\pi z)
\frac{\Gamma(z)\Gamma(k-1/4-z/2)\Gamma(k+1/4-z/2)}{\Gamma(k-1/4+z/2)\Gamma(k+1/4+z/2)}-\\-
2^{-z}\sin(\pi z)
\GenHyGI{3}{2}{1-z/2,3/2-z/2,1}{k+5/4,9/4-k}{1},
\end{multline}
Finally, applying \cite[Equation~5.5.5]{HMF}, we prove \eqref{I<z eq02}.
\end{proof}
%%%%%%%%%%%%%%%%%%%%%%%%%%%%%%%%%%%%%%%%%%%%%%%%%%%%%5
According to \cite[Equation~5.5.5]{HMF}, we have
\begin{equation*}
\Gamma\left(1-\frac{z}{2}\right) \Gamma\left(\frac{3}{2}-\frac{z}{2}\right) = \pi^{1/2} 2^{z-1} (1-z) \Gamma(1-z),
\end{equation*}
and thus \eqref{I<z eq02} can be rewritten in the form
\begin{multline}\label{I<z eq03}
I_{<,0}(k,z)=
\frac{\Gamma^2(z)\Gamma(2k-1/2-z)\cos(\pi z)}{\Gamma(2k-1/2+z)}-\\-
\frac{\pi^{3/2} 2^{-1-z}(1-z)}{\Gamma(k+5/4)\Gamma(9/4-k)}
\GenHyG{3}{2}{1-z/2,3/2-z/2,1}{k+5/4,9/4-k}{1}.
\end{multline}
It follows from \eqref{I>z eq0} and \eqref{I<z eq03} that
\begin{multline}\label{I<z+I>z eq1}
I_{<,0}(k,z)+I_{>,0}(k,z)=
\frac{\Gamma^2(z)\Gamma(2k-1/2-z)(1+\cos(\pi z))}{\Gamma(2k-1/2+z)}-\\-
\frac{\pi^{3/2} 2^{-1-z}(1-z)}{\Gamma(k+5/4)\Gamma(9/4-k)}
\GenHyG{3}{2}{1-z/2,3/2-z/2,1}{k+5/4,9/4-k}{1}.
\end{multline}
Finally, applying \eqref{SZE+SIN rK MT+ET2}, \eqref{I<k to I<z}, \eqref{I>k to I>z}, \eqref{I<z to I<z0}, and \eqref{I>z to I>z0}, we obtain
\begin{multline}\label{SZE+SIN rK MT+ET3}
\SZE(r,K)+\SIN(r,K)=
\sum_{k}h\left(\frac{k}{K}\right)\sum_{e|r^2}\frac{e_2}{\sqrt{re}}\sum_{n=1}^{\infty}
\sum_{q|rn/e_2}\frac{\ups_q(n^2)}{\sqrt{q}}\frac{1}{2\pi i}\int_{(a)}\left(\frac{ne}{4r}\right)^{-z}\\\times
\frac{L_{\infty}(\sym^2f,1/2+z)}{L_{\infty}(\sym^2f,1/2)}\zeta(1+2z)
\Biggl(
\left(\log\frac{n^2}{8\pi q^2}+3\gamma-\frac{\partial}{\partial z}\right)\left(I_{<,0}(k,z)+I_{>,0}(k,z)\right)+\\+
\frac{\pi}{2}\left(I_{<,0}(k,z)-I_{>,0}(k,z)\right)\Biggr)
\frac{dz}{z}+O\left(\frac{r^{3/2}}{K^{3/2-\epsilon}}\right).
\end{multline}

To compute this expression, we first evaluate the sums over $n$ and $q$. Let $R := r / e_2$ and
\begin{equation}
\label{Vseries_def}
V(R,z,u): = \sum_{n=1}^{\infty} \sum_{q \mid Rn} \frac{\ups_q(n^2)}{n^z \sqrt{q}} \left(\frac{n}{q}\right)^u.
\end{equation}
Then, for any constants $C$ and $D$, we have
\begin{equation}\label{nq series to Vseries}
\sum_{n=1}^{\infty}\sum_{q|Rn}\frac{\ups_q(n^2)}{n^z\sqrt{q}}\left(C\log\frac{n}{q}+D\right)=
V(R,z,0)D+C\frac{d}{du}V(R,z,u)\Biggl|_{u=0}.
\end{equation}
Consider
\begin{multline}\label{series n,q transform1}
V(R,z,u)=
\sum_{q=1}^{\infty}\sum_{Rn\equiv0\Mod{q}}\frac{\ups_q(n^2)}{n^{z-u}q^{1/2+u}}=
\sum_{d|R}\sum_{\substack{q=1\\(q,R)=d}}^{\infty}\sum_{n\equiv0\Mod{q/d}}\frac{\ups_q(n^2)}{n^{z-u}q^{1/2+u}}=\\=
\sum_{d|R}\sum_{\substack{q=1\\(q,R/d)=1}}^{\infty}\sum_{n\equiv0\Mod{q}}\frac{\ups_{qd}(n^2)}{n^{z-u}(qd)^{1/2+u}}=
\sum_{d|R}\sum_{q=1}^{\infty}\sum_{m|(q,R/d)}\mu(m)\sum_{n=1}^{\infty}\frac{\ups_{qd}(n^2q^2)}{(nq)^{z-u}(qd)^{1/2+u}}=\\=
\sum_{d|R}\sum_{m|R/d}\frac{\mu(m)}{m^{1/2+z}d^{1/2+u}}\sum_{n,q=1}^{\infty}\frac{\ups_{mdq}(m^2n^2q^2)}{n^{z-u}q^{1/2+z}}.
\end{multline}

We next evaluate the inner sum over $n$ and $q$. To simplify the analysis, we first restrict our attention to the special case where $r = r_1 r_2^2$ with $(r_1, r_2) = 1$, where both $r_1$ and $r_2$ are square-free. (This case suffices to establish Theorem~\ref{thm:nonvanishing}; the general case will be treated in Section~\ref{sec:Proof of Theorem 2mom General}). Consequently, the parameter $R$ satisfies the same conditions. We may also assume that $m$ is square-free. Furthermore, note that if $\nu_p(d) = 2$, then $\nu_p(R/d) = 0$, which implies that $\nu_p(m) = 0$. For convenience, we introduce the notation
\begin{equation}
\label{G_ups_def}
G(n,q) = \sqrt{q} \ups_q(n),
\end{equation}
which slightly simplifies \eqref{series n,q transform1}:
\begin{equation}\label{series n,q transform2}
V(R,z,u)=
\sum_{d|R}\sum_{m|R/d}\frac{\mu(m)}{m^{1+z}d^{1+u}}\sum_{n,q=1}^{\infty}\frac{G(m^2n^2q^2,mdq)}{n^{z-u}q^{1+z}}.
\end{equation}

Since $\ups_q(n)$ is multiplicative in $q$, the function $G(n,q)$ shares the same property. Let $p^{\alpha} \| n$. In the notation of \cite[(8.3)]{IwMi2001}, we have $\ups_q(n) = G(n,q)$; hence, upon multiplication by $\sqrt{q}$, the relations \cite[(8.5), (8.6)]{IwMi2001} for $\gamma \ge 1$ become
\begin{equation}\label{G(n, p gamma)}
G(n, p^{\gamma})=\left\{
  \begin{array}{ll}
    \phi(p^{\gamma}),                                  & \hbox{if}\quad \gamma\le\alpha, \,\gamma-\hbox{even}, \\
    0,                                                 & \hbox{if}\quad \gamma\le\alpha, \,\gamma-\hbox{odd}, \\
    -p^{\gamma-1},                                     & \hbox{if}\quad \gamma=\alpha+1, \,\gamma-\hbox{even}, \\
    \left(\frac{np^{-\alpha}}{p}\right)p^{\alpha+1/2}, & \hbox{if}\quad \gamma=\alpha+1, \,\gamma-\hbox{odd}, \\
    0,                                                 & \hbox{if}\quad \gamma\ge\alpha+2,
  \end{array}
\right.
\end{equation}
and for  $2^{\alpha}\|n$
\begin{equation}\label{G(n, 2 gamma)}
G(n, 2^{\gamma})=\left\{
  \begin{array}{ll}
    2^{\gamma-1},                                                 & \hbox{if}\quad \gamma\le\alpha-2, \,\gamma-\hbox{even}, \\
    0,                                                            & \hbox{if}\quad \gamma\le\alpha-2, \,\gamma-\hbox{odd}, \\
    -2^{\gamma-1},                                                & \hbox{if}\quad \gamma=\alpha-1,   \,\gamma-\hbox{even}, \\
    0,                                                            & \hbox{if}\quad \gamma=\alpha-1,   \,\gamma-\hbox{odd}, \\
    2^{\gamma-1}\chi_4(n2^{-\alpha}),                             & \hbox{if}\quad \gamma=\alpha,     \,\gamma-\hbox{even}, \\
    0,                                                            & \hbox{if}\quad \gamma=\alpha,     \,\gamma-\hbox{odd}, \\
    0,                                                            & \hbox{if}\quad \gamma=\alpha+1,   \,\gamma-\hbox{even}, \\
    2^{\gamma-3/2}(1+\chi_4(n2^{-\alpha}))\chi_8(n2^{-\alpha}),   & \hbox{if}\quad \gamma=\alpha+1,   \,\gamma-\hbox{odd}, \\
    0,                                                            & \hbox{if}\quad \gamma\ge\alpha+2.
  \end{array}
\right.
\end{equation}

%%%%%%%%%%%%%%%%%%%%%%%%%%%%%%%%%%%%%%%%%%%%%%%%%%%%%5
\begin{lem}\label{lem:Sound series}
Let $d = d_1d_2^2$ with $(d_1,d_2) = 1$, where $d_1$ and $d_2$ are square-free.
Let $m$ be a square-free integer such that $(m,d_2) = 1$. Then
\begin{equation}\label{Sound series 0}
\sum_{n,q=1}^{\infty}\frac{G(m^2n^2q^2,mdq)}{n^{z-u}q^{1+z}}
=\frac{\zeta(z-u)\zeta(2z)}{\zeta(1+2z)}\G(m,d,z,u),
\end{equation}
where
\begin{equation}\label{G(m,d,z,u)def}
\G(m,d,z,u)=\G_{1,0}(z)\G_{0,1}(z,u)\G_{1,1}(z)\G_{0,2}(z,u),
\end{equation}
with
\begin{equation}\label{G10 def}
\G_{1,0}(z)=\prod_{\substack{p|m\\p\nmid d}}p^{1-z}\left(1-\frac{1}{p}\right)\left(1-\frac{1}{p^{1+2z}}\right)^{-1},
\end{equation}
\begin{equation}\label{G01 def}
\G_{0,1}(z,u)=\prod_{\substack{p|d_1\\p\nmid m}}\left(1-\frac{1}{p^{1+2z}}\right)^{-1}\left(
p^{1-z}\left(1-\frac{1}{p}\right)+\sqrt{p}\left(1-\frac{1}{p^{z-u}}\right)\left(1-\frac{1}{p^{2z}}\right)
\right),
\end{equation}
\begin{equation}\label{G11 def}
\G_{1,1}(z)=\prod_{\substack{p|m\\p|d_1}}p^{2}\left(1-\frac{1}{p}\right)\left(1-\frac{1}{p^{1+2z}}\right)^{-1},
\end{equation}
\begin{multline}\label{G02 def}
\G_{0,2}(z,u)=\prod_{p|d_2}\left(1-\frac{1}{p^{1+2z}}\right)^{-1}\\\times\left(
p^{2}\left(1-\frac{1}{p}\right)\left(\frac{1}{p^{z-u}}+\frac{1}{p^{2z}}-\frac{1}{p^{3z-u}}\right)
+p^{3/2-z}\left(1-\frac{1}{p^{z-u}}\right)\left(1-\frac{1}{p^{2z}}\right)
\right).
\end{multline}
\end{lem}
\begin{proof}
First, we prove that the function $G(an^2, q)$ is jointly multiplicative with respect to $n$ and $q$.
That is, for $(n_1q_1, n_2q_2) = 1$, we have
\begin{equation}\label{G multiplicity1}
G(an_1^2n_2^2, q_1q_2) = G(an_1^2, q_1)G(an_2^2, q_2).
\end{equation}
Since $G(n, q)$ is multiplicative in $q$, it follows that
\begin{equation}
G(an_1^2n_2^2, q_1q_2) = G(an_1^2n_2^2, q_1)G(an_1^2n_2^2, q_2).
\end{equation}
We now observe that \eqref{G(n, p gamma)} and \eqref{G(n, 2 gamma)} remain unchanged if $n$ is replaced by $nb^2$ with $(b, p) = 1$.
Since $(q_1, n_2) = 1$, we have $G(an_1^2n_2^2, q_1) = G(an_1^2, q_1)$.
Similarly, since $(q_2, n_1) = 1$, we obtain $G(an_1^2n_2^2, q_2) = G(an_2^2, q_2)$.
This establishes \eqref{G multiplicity1}.

More generally, since \eqref{G(n, p gamma)} and \eqref{G(n, 2 gamma)} are invariant under replacing $n$ by $nb^2$ with $(b, p) = 1$, it follows from \eqref{G multiplicity1} that
\begin{equation}\label{G multiplicity2}
G(m^2n^2q^2, mdq) = \prod_{p} G(p^{\nu_p(m^2n^2q^2)}, p^{\nu_p(mdq)}).
\end{equation}
Applying \eqref{G multiplicity2} then yields the Euler product representation
\begin{equation}\label{Sound series 1}
\sum_{n,q=1}^{\infty} \frac{G(m^2n^2q^2, mdq)}{n^{z-u}q^{1+z}} = \prod_{p} \sum_{n,q=0}^{\infty} \frac{G(p^{2n+2q+2\nu_p(m)}, p^{q+\nu_p(md)})}{p^{n(z-u)+q(1+z)}}.
\end{equation}
The conditions of the lemma imply that we must consider five distinct cases, since $0 \le \nu_p(m) \le 1$ and $0 \le \nu_p(d) \le 2$, with the combination $\nu_p(m) = 1$ and $\nu_p(d) = 2$ being impossible.
Let
\begin{equation}\label{Sound series 2}
S_p(\nu_p(m), \nu_p(d)): = \sum_{n,q=0}^{\infty} \frac{G(p^{2n+2q+2\nu_p(m)}, p^{q+\nu_p(md)})}{p^{n(z-u)+q(1+z)}}.
\end{equation}

For $p > 2$, since $q \le 2n+2q$, it follows from \eqref{G(n, p gamma)} that
\begin{multline}\label{Sp(0,0)}
S_p(0,0)=\sum_{n,q=0}^{\infty}\frac{G(p^{2n+2q},p^{q})}{p^{n(z-u)+q(1+z)}}=\sum_{n=0}^{\infty}\frac{1}{p^{n(z-u)}}\left(1+
\sum_{\substack{q=1\\q\equiv0\Mod{2}}}^{\infty}\frac{\phi(p^q)}{p^{q(1+z)}}\right)=\\=
\frac{1}{1-p^{-z+u}}\left(1+\frac{(1-p^{-1})p^{-2z}}{1-p^{-2z}}\right)=
\frac{1-p^{-1-2z}}{(1-p^{-z+u})(1-p^{-2z})}.
\end{multline}

For $p = 2$, since $q \le 2n+2q-2$ holds for all $q$ when $n \ge 1$, and for $q \ge 2$ when $n = 0$, evaluating \eqref{G(n, 2 gamma)} yields
\begin{multline}\label{S2(0,0)}
S_2(0,0)=\sum_{n,q=0}^{\infty}\frac{G(2^{2n+2q},2^{q})}{2^{n(z-u)+q(1+z)}}=
1+\frac{G(2^{2},2)}{2^{1+z}}+\sum_{\substack{q=2\\q\equiv0\Mod{2}}}^{\infty}\frac{2^{q-1}}{2^{q(1+z)}}+\\+
\sum_{n=1}^{\infty}\frac{1}{2^{n(z-u)}}\left(1+
\sum_{\substack{q=1\\q\equiv0\Mod{2}}}^{\infty}\frac{2^{q-1}}{2^{q(1+z)}}\right)=
1+\frac{2^{-1-2z}}{1-2^{-2z}}+\frac{2^{-z+u}}{1-2^{-z+u}}
\left(1+\frac{2^{-1-2z}}{1-2^{-2z}}\right)=\\=
\left(1+\frac{2^{-1-2z}}{1-2^{-2z}}\right)\left(1+\frac{2^{-z+u}}{1-2^{-z+u}}\right)=
\frac{1-2^{-1-2z}}{(1-2^{-z+u})(1-2^{-2z})}.
\end{multline}
Thus, the evaluation in \eqref{S2(0,0)} matches the formula in \eqref{Sp(0,0)} for the case $p = 2$.

If $p > 2$ and $\nu_p(m) = 1, \nu_p(d) = 0$, since $q+1 \le 2n+2q+2$, it follows from \eqref{G(n, p gamma)} that
\begin{align}\label{Sp(1,0)}
S_p(1,0) &= \sum_{n,q=0}^{\infty} \frac{G(p^{2n+2q+2}, p^{q+1})}{p^{n(z-u)+q(1+z)}} = \sum_{n=0}^{\infty} \frac{1}{p^{n(z-u)}} \sum_{\substack{q=0 \\ q \equiv 1 \pmod{2}}}^{\infty} \frac{\phi(p^{q+1})}{p^{q(1+z)}} \notag \\
&= \frac{p(1-p^{-1})}{1-p^{-z+u}} \frac{p^{-z}}{1-p^{-2z}} = \frac{p^{1-z}(1-p^{-1})}{(1-p^{-z+u})(1-p^{-2z})}.
\end{align}

For the case $p = 2$, since $q+1 \le 2n+2q$ holds for all $q$ when $n \ge 1$, and for $q \ge 1$ when $n = 0$, we obtain from \eqref{G(n, 2 gamma)} that
\begin{multline}\label{S2(1,0)}
S_2(1,0)=\sum_{n,q=0}^{\infty}\frac{G(2^{2n+2q+2},2^{q+1})}{2^{n(z-u)+q(1+z)}}=G(2^{2},2)+
\sum_{n=0}^{\infty}\frac{1}{2^{n(z-u)}}
\sum_{\substack{q=1\\q\equiv1\Mod{2}}}^{\infty}\frac{2^{q}}{2^{q(1+z)}}=\\=
\frac{2^{-z}}{(1-2^{-z+u})(1-2^{-2z})}
\end{multline}
Once again, the evaluation in \eqref{S2(1,0)} is consistent with the formula in \eqref{Sp(1,0)}.
If $p > 2$ and $\nu_p(m) = 0, \nu_p(d) = 1$, since $q+1 \le 2n+2q$ holds for all $q, n \ge 0$ except for $n = q = 0$, applying \eqref{G(n, p gamma)} yields
\begin{multline}\label{Sp(0,1)}
S_p(0,1)=\sum_{n,q=0}^{\infty}\frac{G(p^{2n+2q},p^{q+1})}{p^{n(z-u)+q(1+z)}}=G(1,p)+
\sum_{n=0}^{\infty}\frac{1}{p^{n(z-u)}}
\sum_{\substack{q=0\\q\equiv1\Mod{2}}}^{\infty}\frac{\phi(p^{q+1})}{p^{q(1+z)}}=\\=
\sqrt{p}+\frac{p^{1-z}(1-p^{-1})}{(1-p^{-z+u})(1-p^{-2z})}.
\end{multline}

For the case $p = 2$, the inequality $q+1 \le 2n+2q-2$ holds for all $q$ when $n \ge 2$, for $q \ge 1$ when $n = 1$, and for $q \ge 3$ when $n = 0$.
Thus, from \eqref{G(n, 2 gamma)}, we obtain
\begin{align*}
S_2(0,1) &= \sum_{n,q=0}^{\infty} \frac{G(2^{2n+2q}, 2^{q+1})}{2^{n(z-u)+q(1+z)}} = G(1,2) + \frac{G(2^2, 2^2)}{2^{1+z}} + \frac{G(2^4, 2^3)}{2^{2(1+z)}} + \sum_{\substack{q=3 \\ q \equiv 1 \pmod{2}}}^{\infty} \frac{2^{q}}{2^{q(1+z)}} \\
&\quad + \frac{G(2^2, 2)}{2^{z-u}} + \frac{1}{2^{z-u}} \sum_{\substack{q=1 \\ q \equiv 1 \pmod{2}}}^{\infty} \frac{2^{q}}{2^{q(1+z)}} + \sum_{n=2}^{\infty} \frac{1}{2^{n(z-u)}} \sum_{\substack{q=0 \\ q \equiv 1 \pmod{2}}}^{\infty} \frac{2^{q}}{2^{q(1+z)}}.
\end{align*}
Since
\begin{equation}\label{G(2a,2b)}
G(1,2) = \sqrt{2}, \quad G(2^2, 2^2) = 2, \quad G(2^4, 2^3) = G(2^2, 2) = 0,
\end{equation}
it follows that
\begin{multline}\label{S2(0,1)}
S_2(0,1)=\sqrt{2}+\frac{1}{2^z}+\frac{2^{-3z}+2^{-2z+u}}{1-2^{-2z}}+
\frac{2^{-3z+2u}}{(1-2^{-z+u})(1-2^{-2z})}=\\=
\sqrt{2}+\frac{2^{-z}}{(1-2^{-z+u})(1-2^{-2z})}.
\end{multline}
As observed in the previous cases, the evaluation in \eqref{S2(0,1)} matches the formula in \eqref{Sp(0,1)} for $p = 2$.
If $p > 2$ and $\nu_p(m) = 1, \nu_p(d) = 1$, since $q+2 \le 2n+2q+2$ holds for all $q, n \ge 0$, applying \eqref{G(n, p gamma)} yields
\begin{multline}\label{Sp(1,1)}
S_p(1,1)=\sum_{n,q=0}^{\infty}\frac{G(p^{2n+2q+2},p^{q+2})}{p^{n(z-u)+q(1+z)}}=
\sum_{n=0}^{\infty}\frac{1}{p^{n(z-u)}}
\sum_{\substack{q=0\\q\equiv0\Mod{2}}}^{\infty}\frac{\phi(p^{q+2})}{p^{q(1+z)}}=\\=
\frac{p^{2}(1-p^{-1})}{(1-p^{-z+u})(1-p^{-2z})}.
\end{multline}
For the case $p = 2$, since $q+2 \le 2n+2q$ holds for all $q$ when $n \ge 1$, and for $q \ge 2$ when $n = 0$, it follows from \eqref{G(n, 2 gamma)} that
\begin{align}\label{S2(1,1)}
S_2(1,1) &= \sum_{n,q=0}^{\infty} \frac{G(2^{2n+2q+2}, 2^{q+2})}{2^{n(z-u)+q(1+z)}} \notag \\
&= G(2^2, 2^2) + \frac{G(2^{4}, 2^{3})}{2^{1+z}} + \sum_{\substack{q=2 \\ q \equiv 0 \pmod{2}}}^{\infty} \frac{2^{q+1}}{2^{q(1+z)}} + \sum_{n=1}^{\infty} \frac{1}{2^{n(z-u)}} \sum_{\substack{q=0 \\ q \equiv 0 \pmod{2}}}^{\infty} \frac{2^{q+1}}{2^{q(1+z)}} \notag \\
&= 2 + \frac{2^{1-2z}}{1-2^{-2z}} + \frac{2^{1-z+u}}{(1-2^{-z+u})(1-2^{-2z})} = \frac{2}{(1-2^{-z+u})(1-2^{-2z})},
\end{align}
which matches the evaluation in \eqref{Sp(1,1)}.
If $p > 2$ and $\nu_p(m) = 0, \nu_p(d) = 2$, since $q+2 \le 2n+2q$ holds for all $q$ when $n \ge 1$, and for $q \ge 2$ when $n = 0$, we find from \eqref{G(n, p gamma)} that
\begin{multline}\label{Sp(0,2)}
S_p(0,2)=\sum_{n,q=0}^{\infty}\frac{G(p^{2n+2q},p^{q+2})}{p^{n(z-u)+q(1+z)}}=
\sum_{n=0}^{\infty}\frac{1}{p^{n(z-u)}}
\sum_{\substack{q=0\\q\equiv0\Mod{2}}}^{\infty}\frac{\phi(p^{q+2})}{p^{q(1+z)}}-\phi(p^2)+
G(1,p^2)+\\+\frac{G(p^{2},p^{3})}{p^{1+z}}=
\frac{p^{2}(1-p^{-1})}{(1-p^{-z+u})(1-p^{-2z})}-\phi(p^2)+\frac{p^{5/2}}{p^{1+z}}=\\=
\frac{p^{2}(1-p^{-1})}{(1-p^{-z+u})(1-p^{-2z})}\left(\frac{1}{p^{z-u}}+\frac{1}{p^{2z}}-\frac{1}{p^{3z-u}}\right)
+\frac{p^{3/2}}{p^{z}}.
\end{multline}
For the case $p = 2$, since $q+2 \le 2n+2q-2$ holds for all $q$ when $n \ge 2$, for $q \ge 2$ when $n = 1$, and for $q \ge 4$ when $n = 0$, we obtain from \eqref{G(n, 2 gamma)} and \eqref{G(2a,2b)} that
\begin{align}\label{S2(0,2)}
S_2(0,2) &= \sum_{n,q=0}^{\infty} \frac{G(2^{2n+2q}, 2^{q+2})}{2^{n(z-u)+q(1+z)}} \notag \\
&= \sum_{n=2}^{\infty} \frac{1}{2^{n(z-u)}} \sum_{\substack{q=0 \\ q \equiv 0 \pmod{2}}}^{\infty} \frac{2^{q+1}}{2^{q(1+z)}} + \sum_{\substack{q=2 \\ q \equiv 0 \pmod{2}}}^{\infty} \frac{2^{q+1}}{2^{z-u+q(1+z)}} + \sum_{\substack{q=4 \\ q \equiv 0 \pmod{2}}}^{\infty} \frac{2^{q+1}}{2^{q(1+z)}} \notag \\
&\quad + \sum_{q=0}^{1} \frac{G(2^{2+2q}, 2^{q+2})}{2^{z-u+q(1+z)}} + \sum_{q=0}^{3} \frac{G(2^{2q}, 2^{q+2})}{2^{q(1+z)}} \notag \\
&= \frac{2^{1-2z+2u}}{(1-2^{-z+u})(1-2^{-2z})} + \frac{2^{1-3z+u}}{1-2^{-2z}} + \frac{2^{1-4z}}{1-2^{-2z}} + \frac{2}{2^{z-u}} + \frac{2^{5/2}}{2^{1+z}} + \frac{2^3}{2^{2(1+z)}} \notag \\
&= \frac{2}{(1-2^{-z+u})(1-2^{-2z})} \Biggl( \frac{1}{2^{2z-2u}} + \frac{1}{2^{3z-u}} \left( 1 - \frac{1}{2^{z-u}} \right) + \frac{1}{2^{4z}} \left( 1 - \frac{1}{2^{z-u}} \right) \notag \\
&\quad + \left( \frac{1}{2^{z-u}} + \frac{1}{2^{2z}} \right) \left( 1 - \frac{1}{2^{z-u}} \right) \left( 1 - \frac{1}{2^{2z}} \right) \Biggr) + \frac{2^{3/2}}{2^{z}} \notag \\
&= \frac{2}{(1-2^{-z+u})(1-2^{-2z})} \left( \frac{1}{2^{z-u}} + \frac{1}{2^{2z}} - \frac{1}{2^{3z-u}} \right) + \frac{2^{3/2}}{2^{z}},
\end{align}
which is consistent with \eqref{Sp(0,2)}.

Substituting \eqref{Sound series 2} into \eqref{Sound series 1} and applying the explicit evaluations \eqref{Sp(0,0)}, \eqref{Sp(1,0)}, \eqref{Sp(0,1)}, \eqref{Sp(1,1)}, and \eqref{Sp(0,2)}, we obtain \eqref{Sound series 0} after factoring out the common local component
\[
\frac{1-p^{-1-2z}}{(1-p^{-z+u})(1-p^{-2z})}.
\]
This Euler factor yields the global product ratio
\[
\frac{\zeta(z-u)\zeta(2z)}{\zeta(1+2z)}
\]
appearing in \eqref{Sound series 0}.
\end{proof}

%%%%%%%%%%%%%%%%%%%%%%%%%%%%%%%%%%%%%%%%%%%%%%%%%%%%%%%%%5

It now follows from \eqref{nq series to Vseries}, \eqref{series n,q transform2}, and \eqref{Sound series 0} that
\begin{multline}\label{nq series to zeta}
\sum_{n=1}^{\infty}\sum_{q|Rn}\frac{\ups_q(n^2)}{n^z\sqrt{q}}\left(C\log\frac{n}{q}+D\right)=
\frac{\zeta(z)\zeta(2z)}{\zeta(1+2z)}
\sum_{d|R}\sum_{m|R/d}\frac{\mu(m)}{m^{1+z}d}\Biggl(
C\frac{d}{du}\G(m,d,z,u)\Biggl|_{u=0}\\+
\left(D-C\log d-C\frac{\zeta'(z)}{\zeta(z)}\right)\G(m,d,z,0)
\Biggr).
\end{multline}

%%%%%%%%%%%%%%%%%

Using \eqref{nq series to zeta}, we can rewrite \eqref{SZE+SIN rK MT+ET3} in the form
\begin{multline}\label{SZE+SIN rK MT+ET4a>1}
\SZE(r,K)+\SIN(r,K)=
\sum_{k}h\left(\frac{k}{K}\right)\sum_{e|r^2}\frac{e_2}{\sqrt{re}}
\sum_{d|\frac{r}{e_2}}\sum_{m|\frac{r}{de_2}}\frac{\mu(m)}{dm}\\\times
\frac{1}{2\pi i}\int_{(a)}\frac{L_{\infty}(\sym^2f,1/2+z)}{L_{\infty}(\sym^2f,1/2)}
\zeta(z)\zeta(2z)\G(m,d,z,0)\left(\frac{me}{4r}\right)^{-z}
\Hf(k,z)
\frac{dz}{z}+O\left(\frac{r^{3/2}}{K^{3/2-\epsilon}}\right),
\end{multline}
where $a > 1$ and
\begin{multline}\label{Hf def}
\Hf(k,z)=
\frac{\pi}{2}\left(I_{<,0}(k,z)-I_{>,0}(k,z)\right)+\\+
\Bigl(2\frac{d}{du}\log\G(m,d,z,u)\Biggl|_{u=0}-\log(8\pi)+3\gamma-2\log d-2\frac{\zeta'(z)}{\zeta(z)}-
\frac{\partial}{\partial z}\Bigr)\left(I_{<,0}(k,z)+I_{>,0}(k,z)\right).
\end{multline}

We now shift the line of integration to $\operatorname{Re}(z) = a$ with $1/2 < a < 3/4$.
According to \eqref{I<z+I>z eq1}, we have $I_{<}(k,1) + I_{>}(k,1) = 0$.
This vanishing property, combined with several other unexpected cancellations, ensures that the contribution of the residue at $z = 1$ is negligibly small.

%%%%%%%%%%%%%%%%%%%%%%%%%%%%%%%%%%%%%%%%%%%%%%%%%%%%%%%%%%%%%%%%5

\begin{lem}\label{lem:residue z=1}
For $1/2 < a < 3/4$, we have
\begin{multline}\label{SZE+SIN rK MT+ET4}
\SZE(r,K)+\SIN(r,K)=
\sum_{k}h\left(\frac{k}{K}\right)\sum_{e|r^2}\frac{e_2}{\sqrt{re}}
\sum_{d|\frac{r}{e_2}}\sum_{m|\frac{r}{de_2}}\frac{\mu(m)}{dm}\\\times
\frac{1}{2\pi i}\int_{(a)}\frac{L_{\infty}(\sym^2f,1/2+z)}{L_{\infty}(\sym^2f,1/2)}
\zeta(z)\zeta(2z)\G(m,d,z,0)\left(\frac{me}{4r}\right)^{-z}
\Hf(k,z)
\frac{dz}{z}+O\left(\frac{r^{3/2}}{K^{3/2-\epsilon}}\right),
\end{multline}
\end{lem}

\begin{proof}
It follows from \eqref{Hf def} that
\begin{align}\label{Hf to Hf1}
\zeta(z)\Hf(k,z) &= \Hf_1(k,z) \notag \\
&\quad + \zeta(z) \left( 2\frac{d}{du}\log\G(m,d,z,u)\Biggl|_{u=0} - \log(8\pi) + 3\gamma - 2\log d \right) \left(I_{<,0}(k,z) + I_{>,0}(k,z)\right),
\end{align}
where
\begin{align}\label{Hf1def}
\Hf_1(k,z) &= \frac{\pi\zeta(z)}{2}\left(I_{<,0}(k,z) - I_{>,0}(k,z)\right) - 2\zeta'(z)\left(I_{<,0}(k,z) + I_{>,0}(k,z)\right) \notag \\
&\quad - \zeta(z)\frac{\partial}{\partial z}\left(I_{<,0}(k,z) + I_{>,0}(k,z)\right).
\end{align}
Since $I_{<}(k,1) + I_{>}(k,1) = 0$ (see \eqref{I<z+I>z eq1}), we obtain
\begin{equation}\label{resHf to resHf1}
\Res_{z=1} \zeta(z)\Hf(k,z) = \Res_{z=1} \Hf_1(k,z).
\end{equation}
By virtue of the elementary trigonometric relation
\begin{equation}\label{cos to sin}
1 + \cos(\pi z) = 2\sin^2\frac{\pi(1-z)}{2},
\end{equation}
the function $(1 + \cos(\pi z))\zeta'(z)$ is holomorphic at $z=1$. Consequently, \eqref{I>z eq0}, \eqref{I<z eq03}, and \eqref{I<z+I>z eq1} imply that
\begin{multline}\label{res Hf1 eq1}
\Res_{z=1}\Hf_1(k,z)=\Res_{z=1}\Biggl(
\frac{\pi\zeta(z)\Gamma^2(z)\Gamma(2k-1/2-z)(\cos(\pi z)-1)}{2\Gamma(2k-1/2+z)}+\\
+\frac{\pi^{3/2} \zeta'(z)2^{-z}(1-z)}{\Gamma(k+5/4)\Gamma(9/4-k)}
\GenHyG{3}{2}{1-z/2,3/2-z/2,1}{k+5/4,9/4-k}{1}-
\zeta(z)\frac{\partial}{\partial z}\left(I_{<,0}(k,z)+I_{>,0}(k,z)\right)\Biggr).
\end{multline}
Using \eqref{I<z+I>z eq1} and \eqref{cos to sin}, we find
\begin{multline}\label{res dz(I<z+I>z)}
\Res_{z=1}\zeta(z)\frac{\partial}{\partial z}\left(I_{<,0}(k,z)+I_{>,0}(k,z)\right)=\\=\Res_{z=1}\frac{\pi^{3/2}\zeta(z) 2^{-1-z}}{\Gamma(k+5/4)\Gamma(9/4-k)}
\GenHyG{3}{2}{1-z/2,3/2-z/2,1}{k+5/4,9/4-k}{1}.
\end{multline}
Substituting \eqref{res dz(I<z+I>z)} into \eqref{res Hf1 eq1} then yields
\begin{multline}\label{res Hf1 eq2}
\Res_{z=1}\Hf_1(k,z)=\Res_{z=1}\frac{1}{z-1}\Biggl(
\frac{\pi\Gamma^2(z)\Gamma(2k-1/2-z)(\cos(\pi z)-1)}{2\Gamma(2k-1/2+z)}+\\
+\frac{\pi^{3/2} (2^{-z}-2^{-1-z})}{\Gamma(k+5/4)\Gamma(9/4-k)}
\GenHyG{3}{2}{1-z/2,3/2-z/2,1}{k+5/4,9/4-k}{1}\Biggr)=
\frac{-\pi\Gamma(2k-3/2)}{\Gamma(2k+1/2)}+\\
+\frac{\pi^{3/2}2^{-2}}{\Gamma(k+5/4)\Gamma(9/4-k)}
\GenHyG{3}{2}{1/2,1,1}{k+5/4,9/4-k}{1}.
\end{multline}
Finally, applying \cite{PBMv3}, we obtain
\begin{multline}\label{res Hf1 eq3}
\GenHyG{3}{2}{1/2,1,1}{k+5/4,9/4-k}{1}=\frac{\Gamma(k+5/4)\Gamma(9/4-k)}{\Gamma(1/2)}\GenHyG{3}{2}{k+3/4,7/4-k,1}{2,2}{1}=\\=
\frac{\Gamma(k+5/4)\Gamma(9/4-k)}{\Gamma(1/2)(k-1/4)(3/4-k)}\Biggl(
\frac{\Gamma(1/2)}{\Gamma(5/4-k)\Gamma(k+1/4)}-1
\Biggr).
\end{multline}
Substituting \eqref{res Hf1 eq3} into \eqref{res Hf1 eq2} and applying \cite{HMF}, we obtain
\begin{multline}\label{res Hf1 eq4}
\Res_{z=1}\Hf_1(k,z)=
\frac{\pi^{3/2}/4}{(k-1/4)(3/4-k)\Gamma(5/4-k)\Gamma(k+1/4)}-\\-
\frac{\pi}{(2k-1/2)(2k-3/2)}-
\frac{\pi/4}{(k-1/4)(3/4-k)}=
\frac{\pi^{1/2}\sin(5\pi/4-\pi k)\Gamma(k-1/4)}{(2k-1/2)(3/2-2k)\Gamma(k+1/4)}=\\=
\frac{(-1)^k\pi^{1/2}\Gamma(k-1/4)}{2^{1/2}(2k-1/2)(2k-3/2)\Gamma(k+1/4)}.
\end{multline}
It follows from \eqref{SZE+SIN rK MT+ET4a>1}, \eqref{resHf to resHf1}, and \eqref{res Hf1 eq4} that the contribution of the residue at $z=1$ to $\SZE(r,K) + \SIN(r,K)$ is bounded by
\begin{multline}\label{SZE+SIN res1 eq1}
\sum_{e|r^2}\frac{e_2}{\sqrt{re}}
\sum_{d|\frac{r}{e_2}}\sum_{m|\frac{r}{de_2}}\frac{\mu(m)}{dm}
\G(m,d,1,0)\left(\frac{me}{r}\right)^{-1}
\sum_{k}h\left(\frac{k}{K}\right)\\\times
\frac{(-1)^k\Gamma(k-1/4)}{(2k-1/2)(2k-3/2)\Gamma(k+1/4)}\frac{L_{\infty}(\sym^2f,3/2)}{L_{\infty}(\sym^2f,1/2)}\ll\\\ll
\sum_{e|r^2}\frac{e_2\sqrt{r}}{e^{3/2}}
\sum_{d|\frac{r}{e_2}}\sum_{m|\frac{r}{de_2}}\frac{\mu(m)}{dm^2}\G(m,d,1,0)
\sum_{k}h\left(\frac{k}{K}\right)
\frac{(-1)^k\Gamma(k-1/4)}{(2k-3/2)\Gamma(k+1/4)},
\end{multline}
where we have used \eqref{L.infinity} to derive the last inequality.

Applying the Poisson summation formula and integrating by parts sufficiently many times, we obtain
\begin{multline}\label{SZE+SIN res1 eq2}
\sum_{k}h\left(\frac{k}{K}\right)
\frac{e(k/2)\Gamma(k-1/4)}{(2k-3/2)\Gamma(k+1/4)}=\\=
K\sum_{m}\int_{-\infty}^{\infty}h(y)
\frac{e(Ky(m+1/2))\Gamma(Ky-1/4)}{(2Ky-3/2)\Gamma(Ky+1/4)}dy\ll K^{-A}.
\end{multline}
Using \eqref{G(m,d,z,u)def} and estimating $\G(m,d,1,0)$ trivially, we find that
\begin{equation}\label{G(m,d,1,0) est}
\G(m,d,1,0) \ll \left(\frac{d_1}{(m,d_1)}\right)^{1/2}(m,d_1)^2d_2 m^{\epsilon}d^{\epsilon}.
\end{equation}
Substituting \eqref{SZE+SIN res1 eq2} and \eqref{G(m,d,1,0) est} into \eqref{SZE+SIN res1 eq1}, we obtain that the contribution of the residue is bounded by
\begin{equation}\label{SZE+SIN res1 eq3}
K^{-A}\sum_{e_1e_2|r}\frac{e_2\sqrt{r}}{e_1^{3/2}e_2^3} \sum_{d|\frac{r}{e_2}}\sum_{m|\frac{r}{de_2}}\frac{(m,d_1)^{3/2}}{d^{1/2}m^2}m^{\epsilon}d^{\epsilon} \ll \frac{r^{1/2+\epsilon}}{K^A},
\end{equation}
which completes the proof of \eqref{SZE+SIN rK MT+ET4}.

\end{proof}
%%%%%%%%%%%%%%%%%%%%%%%%%%%%%%%%%%%%%%%%%%%%%%%%%
On the new line of integration, we are able to obtain asymptotic formulas for both $I_{<,0}(k,z)$ and $I_{>,0}(k,z)$.

%%%%%%%%%%%%%%%%%%%%%%%%%%%%%%%%%%%%%%%%%%%%%%%%%

\begin{lem}\label{lem:I<0+I>0 asympt}
For $0<\Re{z}<2k-1/2$, we have
\begin{equation}\label{I>0z asympt}
I_{>,0}(k,z)=\frac{\Gamma^2(z)}{(2k)^{2z}}
\left(1+O\left(\frac{|z|}{k}\right)\right).
\end{equation}
For $0<\Re{z}<3/4$, $|\Im{z}|\ll K^{\epsilon}$, we have
\begin{multline}\label{I<0z asympt}
\sum_{k}h\left(\frac{k}{K}\right)\frac{L_{\infty}(\sym^2f,1/2+z)}{L_{\infty}(\sym^2f,1/2)}I_{<,0}(k,z)=\\=
\sum_{k}h\left(\frac{k}{K}\right)\frac{L_{\infty}(\sym^2f,1/2+z)}{L_{\infty}(\sym^2f,1/2)}\frac{\Gamma^2(z)}{(2k)^{2z}}
\left(\cos(\pi z)-\sin(\pi z)+O\left(\frac{|z|}{K}\right)\right).
\end{multline}
\end{lem}
\begin{proof}

It follows from \eqref{PhikPsik toIk} and \eqref{I<0zdef} that
\begin{multline}\label{I<z to Phi}
I_{<,0}(k,z)=
\int_{1}^{\infty}\Phi_k\left(\left(1-\frac{2}{x+1}\right)^2\right)\frac{dx}{(x+1)^{1+z}}=\\=
\int_{0}^{\pi^2/4}\Phi_k\left(\cos^2\sqrt{\xi}\right)\sin^{2z-1}(\sqrt{\xi}/2)\cos(\sqrt{\xi}/2)\frac{d\xi}{2\sqrt{\xi}}.
\end{multline}
We employ the average over $k$ to demonstrate that the contribution from the region $\xi \gg K^{\epsilon-2}$ is negligible.
This can be established in a manner analogous to the analysis at the beginning of Section~\ref{sec: SZE}.
The single difference now is that instead of the factor $V_k(x)$, the ratio $\frac{L_{\infty}(\sym^2f,1/2+z)}{L_{\infty}(\sym^2f,1/2)}$ appears in \eqref{SZE trunc1}; thus, it suffices to show that
\begin{equation}\label{I< trunc1}
\sum_{k}h\left(\frac{k}{K}\right)\frac{L_{\infty}(\sym^2f,1/2+z)}{L_{\infty}(\sym^2f,1/2)}e^{\pm2ki\sqrt{\xi}}\ll K^{-A}.
\end{equation}
Taking additional terms in the expansion \eqref{Linf(z)/Linfe q0}, this reduces to proving that
\begin{equation}\label{I<trunc2}
\sum_{k}h\left(\frac{k}{K}\right)k^{z}e^{\pm2ki\sqrt{\xi}}\ll K^{-A}.
\end{equation}
Applying the Poisson summation formula and writing $z=\sigma+it$, we obtain
\begin{equation}\label{I<trunc3}
\sum_{k}h\left(\frac{k}{K}\right)k^ze^{\pm2ki\sqrt{\xi}}=
K^{1+z}\sum_{m}\int_{-\infty}^{\infty}h(x)x^{\sigma}e^{2iKx(\pm\sqrt{\xi}+\pi m)+it\log x}\,dx.
\end{equation}
Since $|t|\ll K^{\epsilon}\ll K\sqrt{\xi}$, the phase satisfies
\begin{equation}\label{I<trunc4}
2Kx(\pm\sqrt{\xi}+\pi m)+t\log x\gg \min(K|m|, \pm 2Kx\sqrt{\xi}+t\log x)\gg\min(K|m|, K^{\epsilon}).
\end{equation}
Consequently, integrating by parts in \eqref{I<trunc3} yields \eqref{I<trunc2}.

After performing this truncation  we can apply the asymptotic formula (see \eqref{Phik LG})
\begin{equation}\label{Phik LG1}
\Phi_k(\cos^2\sqrt{\xi})=\frac{-\pi\xi^{1/4}}{(\sin\sqrt{\xi})^{1/2}}\left(J_0((2k-1)\sqrt{\xi})+Y_0((2k-1)\sqrt{\xi})\right)+O(k^{-1+\epsilon}).
\end{equation}
Therefore, on average we can replace $I_{<,0}(k,z)$ (see \eqref{I<z to Phi}) by
\begin{align}\label{I<z to Phi2}
&\frac{-\pi}{2^{3/2}} \int_{0}^{K^{-2+\epsilon}}\left(J_0((2k-1)\sqrt{\xi})+Y_0((2k-1)\sqrt{\xi})\right) \sin^{2z-3/2}(\sqrt{\xi}/2)\cos^{1/2}(\sqrt{\xi}/2)\frac{d\xi}{\xi^{1/4}} \notag \\
&+O(K^{-1-2\Re{z}+\epsilon}).
\end{align}
Next, we make the change of variables $\xi=4x^2$ and replace $\sin x$ by $x$ and $\cos x$ by 1, obtaining
\begin{equation}\label{I<z to Phi3}
-2\pi \int_{0}^{K^{-1+\epsilon}}\left(J_0((4k-2)x)+Y_0((4k-2)x)\right) x^{2z-1}dx+O(K^{-1-2\Re{z}+\epsilon}).
\end{equation}

We now wish to extend the range of integration to $(0,\infty)$.
Again, we can show that, on average over $k$, the integral over $(K^{-1+\epsilon},1/2)$ is negligible.
However, we cannot use these arguments to extend the integral up to infinity.
This is because after applying the Poisson summation formula, we encounter the phase $2\pi m\pm4x$, which can become very small near the points $x_m=\pi m/2$.
Consequently, we must show that the integral
\begin{equation}\label{I<z to Phi4}
\int_{0}^{\infty}\left(J_0((4k-2)x)+Y_0((4k-2)x)\right) x^{2z-1}\chi(x)dx
\end{equation}
is itself negligible. Here, $\chi(x)$ is a smooth function satisfying $\chi(x)=0$ for $x<1/4$ and $\chi(x)=1$ for $x>1/2$.

 First, we choose a large parameter $X=K^{A}$ and perform an additional smooth partition of unity on \eqref{I<z to Phi4} into two integrals over $(0,X)$ and $(X,\infty)$, respectively.
To estimate the integral over $(X,\infty)$, we replace the Bessel functions with their asymptotic expansions \cite{GR} and apply the first derivative test, yielding
\begin{align}\label{I<z Xinf est}
\int_{X}^{\infty}\left(J_0((4k-2)x)+Y_0((4k-2)x)\right) x^{2z-1}\chi_{1}(x)dx &\ll \int_{X}^{\infty}e^{\pm i(4k-2)x+2it\log x} x^{2\sigma-3/2}\chi_{1}(x)dx \notag \\
&\ll \max_{x>X}\frac{x^{2\sigma-3/2}\chi_{1}(x)}{k\pm t/x} \ll X^{2\sigma-3/2}\ll K^{-A},
\end{align}
provided that $\sigma<3/4$. To estimate the remaining integral over $(0,X)$, we apply another partition of unity
\begin{equation}\label{I<z Um partition}
\sum_{-1\le m\ll\log X}\int_{0}^{\infty}\left(J_0((4k-2)x)+Y_0((4k-2)x)\right) x^{2z-1}U\left(\frac{x}{2^m}\right)dx,
\end{equation}
where $U(x)$ is a smooth function supported on $[1/2,2]$.
Performing the change of variables $x=2^m\sqrt{y}$ and applying \eqref{Harcos est}, we obtain -- where $B_0$ denotes either $J_0$ or $Y_0$ --
\begin{align}\label{I<z Um partition2}
\int_{0}^{\infty}B_0((4k-2)x)x^{2z-1}U\left(\frac{x}{2^m}\right)dx &\ll 2^{2m\Re(z)}\int_{0}^{\infty}B_0((4k-2)2^m\sqrt{y})y^{z-1}U\left(\sqrt{y}\right)dy \notag \\
&\ll \frac{2^{2m\Re(z)}}{(2k-1)^j2^{mj}}\int_{0}^{\infty}\frac{d^j}{dy^j} \left(y^{z-1}U\left(\sqrt{y}\right)\right)y^{j/2}B_j((4k-2)2^m\sqrt{y})dy \notag \\
&\ll \frac{2^{2m\Re(z)}(1+|z|)^j}{(2k-1)^j2^{mj}}.
\end{align}
Taking $j$ sufficiently large, we obtain that
\begin{equation}\label{I<z Um partition3}
\sum_{-1\le m\ll\log X}\int_{0}^{\infty}\left(J_0((4k-2)x)+Y_0((4k-2)x)\right) x^{2z-1}U\left(\frac{x}{2^m}\right)dx\ll K^{-A}.
\end{equation}
Therefore, on average we can replace $I_{<,0}(k,z)$ by
\begin{align}\label{I<z to Phi5}
&-2\pi\int_{0}^{\infty}\left(J_0((4k-2)x)+Y_0((4k-2)x)\right) x^{2z-1}dx+O(K^{-1-2\Re{z}+\epsilon}) \notag \\
&=\frac{-\pi}{(2k-1)^{2z}}\frac{\Gamma(z)}{\Gamma(1-z)}\left(1-\cot(\pi z)\right)+O(K^{-1-2\Re{z}+\epsilon}) \notag \\
&=\frac{\cos(\pi z)-\sin(\pi z)}{(2k-1)^{2z}}\Gamma^2(z)+O(K^{-1-2\Re{z}+\epsilon}),
\end{align}
where we have used \cite[6.561.14]{GR} and \cite[6.561.15]{GR} to evaluate the Mellin transform of Bessel functions.

The analysis for $I_{>,0}(k,z)$ is more direct, as it follows from \eqref{I>z eq0} and \cite{HMF} that
\begin{equation}\label{I>z eq0 simple}
I_{>,0}(k,z)=\frac{\Gamma^2(z)\Gamma(2k-1/2-z)}{\Gamma(2k-1/2+z)}=\frac{\Gamma^2(z)}{(2k)^{2z}}\left(1+O((1+|z|)/k)\right).
\end{equation}
\end{proof}
%%%%%%%%%%%%%%%%%%%%%%%%%%%%%%%%%%%%%%%%%%%%%%%%%%%%%%%%%%%%%%%

We now wish to simplify the integrand in \eqref{SZE+SIN rK MT+ET4}.
First, we replace $\frac{L_{\infty}(\sym^2f,1/2+z)}{L_{\infty}(\sym^2f,1/2)}$ by \eqref{Linf(z)/Linfe q0}.
In fact, this approximation can be employed from the beginning of this section, since the contribution of the error term $O(k^{-1})$ from \eqref{Linf(z)/Linfe q0} is bounded by $O(r^{1/2}K^{\epsilon-1})$ (see \eqref{SZE+SIN rK MT+ET triv}).
This error bound is strictly smaller than the main term $r^{-1/2}K^{1+\epsilon}$ for $r < K^{2-\epsilon}$.
Additionally, we have the functional equation
\begin{equation}\label{log zeta FE}
\frac{\zeta'(z)}{\zeta(z)}=\log\pi-\frac{1}{2}\psi\left(\frac{z}{2}\right)-\frac{1}{2}\psi\left(\frac{1-z}{2}\right)-
\frac{\zeta'(1-z)}{\zeta(1-z)}.
\end{equation}
It follows from \eqref{I>0z asympt} and \eqref{I<0z asympt} that, on average, $\frac{\partial}{\partial z}\left(I_{<,0}(k,z)+I_{>,0}(k,z)\right)$ can be replaced by
\begin{multline}\label{dz I<0+I>0}
\frac{\Gamma^2(z)}{(2k)^{2z}}\Biggl(\left(1-\sin(\pi z)+\cos(\pi z)\right)(2\psi(z)-2\log(2k))
-\pi\sin(\pi z)-\pi\cos(\pi z)
\Biggr)\\\times
\left(1+O\left(\frac{|z|}{k}\right)\right).
\end{multline}
Applying \eqref{I>0z asympt}, \eqref{I<0z asympt}, \eqref{dz I<0+I>0}, and \eqref{log zeta FE}, we find that \eqref{Hf def} can be rewritten as
\begin{equation}\label{Hf to Hf2}
\Hf(k,z)=
\frac{\Gamma^2(z)}{(2k)^{2z}}\left(1-\sin(\pi z)+\cos(\pi z)\right)\Hf_2(k,z)\left(1+O\left(\frac{|z|}{k}\right)\right),
\end{equation}
where
\begin{multline}\label{Hf2 def}
\Hf_2(k,z)=
\frac{\pi}{2}\frac{\sin(\pi z)+3\cos(\pi z)-1}{1-\sin(\pi z)+\cos(\pi z)}+
2\frac{d}{du}\log\G(m,d,z,u)\Biggl|_{u=0}-\log(8\pi)+\\+3\gamma-2\log d+2\frac{\zeta'(1-z)}{\zeta(1-z)}-2\log(\pi)+
\psi\left(\frac{z}{2}\right)+\psi\left(\frac{1-z}{2}\right)-2\psi\left(z\right)+2\log(2k).
\end{multline}
Using \cite{HMF}, we obtain the identity
\begin{align}\label{Hf2 psi}
\psi\left(\frac{z}{2}\right)+\psi\left(\frac{1-z}{2}\right)-2\psi\left(z\right) &=\psi\left(\frac{1-z}{2}\right)-\psi\left(\frac{1+z}{2}\right)-2\log2 \notag \\
&= -2\log2-\pi\frac{\sin(\pi z/2)}{\cos(\pi z/2)}.
\end{align}
Substituting \eqref{Hf2 psi} into \eqref{Hf2 def} yields
\begin{align}\label{Hf2 eq1}
\Hf_2(k,z)&=
\frac{\pi}{2}\frac{\sin(\pi z)+3\cos(\pi z)-1}{1-\sin(\pi z)+\cos(\pi z)}-\pi\frac{\sin(\pi z/2)}{\cos(\pi z/2)}+
2\frac{d}{du}\log\G(m,d,z,u)\Biggl|_{u=0}+3\gamma \notag \\
&\quad +2\frac{\zeta'(1-z)}{\zeta(1-z)}+\log\frac{k^2}{8\pi^3d^2}.
\end{align}

A straightforward calculation leads to
\begin{equation}\label{Hf2 eq2}
\Hf_2(k,z)=
\log\frac{k^2}{8\pi^3d^2}+2\frac{d}{du}\log\G(m,d,z,u)\Biggl|_{u=0}+3\gamma+2\frac{\zeta'(1-z)}{\zeta(1-z)}+\frac{\pi}{2}.
\end{equation}
Substituting \eqref{Hf to Hf2} and \eqref{Linf(z)/Linfe q0} into \eqref{SZE+SIN rK MT+ET4}, we obtain
\begin{multline}\label{SZE+SIN rK MT+ET5}
\SZE(r,K)+\SIN(r,K)=
\sum_{k}h\left(\frac{k}{K}\right)\sum_{e|r^2}\frac{e_2}{\sqrt{re}}
\sum_{d|\frac{r}{e_2}}\sum_{m|\frac{r}{de_2}}\frac{\mu(m)}{dm}
\frac{1}{2\pi i}\int_{(a)}
\pi^{-3z/2}\zeta(z)\zeta(2z)\\\times\frac{\Gamma\left(3/4+z/2\right)}{\Gamma\left(3/4\right)}
\Gamma^2(z)\left(1-\sin(\pi z)+\cos(\pi z)\right)
\G(m,d,z,0)\left(\frac{kme}{r}\right)^{-z}\Hf_2(k,z)\\\times\left(1+O\left(\frac{|z|}{k}\right)\right)
\frac{dz}{z}+O\left(\frac{r^{3/2}}{K^{3/2-\epsilon}}\right).
\end{multline}
Note that the integrand is holomorphic at $z=1/2$, allowing us to shift the line of integration to any $a>0$.
Using the functional equation for the Riemann zeta function \cite{HMF} and subsequently applying \cite{HMF}, we find
\begin{multline}\label{zetazeta eq1}
\zeta(z)\zeta(2z)\Gamma^2(z)\Gamma\left(3/4+z/2\right)=\\=
2^{3z}\pi^{3z-2}\zeta(1-z)\zeta(1-2z)\sin(\pi z/2)\sin(\pi z)\Gamma(1-z)\Gamma^2(z)\Gamma(1-2z)
\Gamma\left(3/4+z/2\right)=\\=
2^{3z}\pi^{3z-1}\zeta(1-z)\zeta(1-2z)\sin(\pi z/2)\Gamma(z)\Gamma(1-2z)
\Gamma\left(3/4+z/2\right)=\\=
2^{z}\pi^{3z-3/2}\zeta(1-z)\zeta(1-2z)\sin(\pi z/2)\Gamma(z)\Gamma\left(3/4+z/2\right)\Gamma(1/2-z)\Gamma(1-z)=\\=
2^{z}\pi^{3z-1/2}\zeta(1-z)\zeta(1-2z)\frac{\sin(\pi z/2)}{\sin(\pi z)}\Gamma\left(3/4+z/2\right)\Gamma(1/2-z)=\\=
2^{-1/2}\pi^{3z-1}\zeta(1-z)\zeta(1-2z)\frac{\sin(\pi z/2)}{\sin(\pi z)}\Gamma\left(3/4+z/2\right)\Gamma\left(1/4-z/2\right)
\Gamma\left(3/4-z/2\right)=\\=
2^{-1/2}\pi^{3z}\zeta(1-z)\zeta(1-2z)\frac{\sin(\pi z/2)\Gamma\left(3/4-z/2\right)}{\sin(\pi z)\sin\left(\pi/4-\pi z/2\right)}
=\\=
2^{-1/2}\pi^{3z}\zeta(1-z)\zeta(1-2z)\frac{\Gamma\left(3/4-z/2\right)}{\sin(\pi/4)+\sin\left(\pi/4-\pi z\right)}=\\
=\pi^{3z}\zeta(1-z)\zeta(1-2z)\frac{\Gamma\left(3/4-z/2\right)}{1+\cos(\pi z)-\sin(\pi z)}.
\end{multline}

%%%%%%%%%%%%%%%%%%%%%%%%%%%%%%%%%%%%%%%%%%%%%%%%%%%%%%%%%%%%%%%%%%%%%%%

To prove Theorem \ref{thm:2mom average}, we need to evaluate the sum of \eqref{2mom MT2 eq5} and \eqref{SZE+SIN rK MT+ET5}. For this purpose, we must first evaluate the sums in \eqref{SZE+SIN rK MT+ET5}. It follows from \eqref{C1 def} and \eqref{Hf2 eq2} that
\begin{equation}\label{Hf2 eq4}
\Hf_2(k,z)=\Cc_1(-z)-2\log d+2\frac{d}{du}\log\G(m,d,z,u)\Biggl|_{u=0}.
\end{equation}
Let
\begin{equation}\label{SZE+SIN md sum def}
\SG(R,u,v,z):=
\sum_{d|R}\sum_{m|\frac{R}{d}}\frac{\mu(m)}{d^{1+v}m^{1+z}}\G(m,d,z,u)
\end{equation}
and
\begin{multline}\label{SZE+SIN sum e def}
\ES(r,k,z):=
\sum_{e_1e_2|r}\frac{|\mu(e_1)|}{\sqrt{e_1}}\left(\frac{r}{ke_1e_2^2}\right)^{z}\Biggl(\Cc_1(-z)
\SG\left(\frac{r}{e_2},0,0,z\right) + \\
+ 2\frac{d}{du}\SG\left(\frac{r}{e_2},u,0,z\right)\Biggl|_{u=0} +
2\frac{d}{dv}\SG\left(\frac{r}{e_2},0,v,z\right)\Biggl|_{v=0}
\Biggr).
\end{multline}
Since the $O$-term in the integrand of \eqref{SZE+SIN rK MT+ET5} does not depend on $m$, relations \eqref{Hf2 eq4}, \eqref{SZE+SIN md sum def}, and \eqref{SZE+SIN sum e def} imply that \eqref{SZE+SIN rK MT+ET5} can be rewritten as
\begin{multline}\label{SZE+SIN rK MT+ET8}
\SZE(r,K)+\SIN(r,K)=
\frac{1}{\sqrt{r}}\sum_{k}h\left(\frac{k}{K}\right)
\frac{1}{2\pi i}\int_{(a)}
\pi^{3z/2}\frac{\Gamma\left(3/4-z/2\right)}{\Gamma\left(3/4\right)}
\\\times
\zeta(1-z)\zeta(1-2z)\ES(r,k,z)\left(1+O\left(\frac{|z|}{k}\right)\right)
\frac{dz}{z}+O\left(\frac{r^{3/2}}{K^{3/2-\epsilon}}\right).
\end{multline}

%%%%%%%%%%%%%%%%%%%%%%%%%%%%%%%%%%%%%%%%%%%%%%%%%%%

To evaluate \eqref{SZE+SIN rK MT+ET8}, we must first analyze \eqref{SZE+SIN md sum def}.
It follows from \eqref{G(m,d,z,u)def} that for $(m_1a,m_2b)=1$, we have
\begin{equation}\label{G multiplicative}
\G(m_1m_2,ab,z,u)=\G(m_1,a,z,u)\G(m_2,b,z,u).
\end{equation}
By virtue of \eqref{G multiplicative} and \eqref{SZE+SIN md sum def}, the function $\SG(R,u,v,z)$ is multiplicative in $R$.
Since we assume that $r=r_1r_2^2$ with $r_1,r_2$ being coprime square-free numbers, it suffices to evaluate only $\SG(p,u,v,z)$ and $\SG(p^2,u,v,z)$.
We have
\begin{equation}\label{SG(p) eq1}
\SG(p,u,v,z)=\G(1,1,z,u)-\frac{\G(p,1,z,u)}{p^{1+z}}+\frac{\G(1,p,z,u)}{p^{1+v}},
\end{equation}
\begin{equation}\label{SG(p2) eq1}
\SG(p^2,u,v,z)=\G(1,1,z,u)-\frac{\G(p,1,z,u)}{p^{1+z}}+\frac{\G(1,p,z,u)}{p^{1+v}}-
\frac{\G(p,p,z,u)}{p^{2+v+z}}+\frac{\G(1,p^2,z,u)}{p^{2+2v}}.
\end{equation}
Furthermore, it follows from \eqref{G(m,d,z,u)def} that
\begin{equation}\label{SG(p) G11 1p}
\G(1,1,z,u)=1,\quad
\G(p,1,z,u)=\G_{1,0}(z)=p^{1-z}\left(1-\frac{1}{p}\right)\left(1-\frac{1}{p^{1+2z}}\right)^{-1},
\end{equation}
and
\begin{multline}\label{SG(p) Gp1}
\G(1,p,z,u)=\G_{0,1}(z,u)=\left(1-\frac{1}{p^{1+2z}}\right)^{-1}\\\times\left(
p^{1-z}\left(1-\frac{1}{p}\right)+\sqrt{p}\left(1-\frac{1}{p^{z-u}}\right)\left(1-\frac{1}{p^{2z}}\right)\right),
\end{multline}
\begin{equation}\label{SG(p2) Gpp}
\G(p,p,z,u)=\G_{1,1}(z)=p^{2}\left(1-\frac{1}{p}\right)\left(1-\frac{1}{p^{1+2z}}\right)^{-1},
\end{equation}
\begin{multline}\label{SG(p2) G1p2}
\G(1,p^2,z,u)=\G_{0,2}(z,u)=\left(1-\frac{1}{p^{1+2z}}\right)^{-1}
p^{2}\left(1-\frac{1}{p}\right)\left(\frac{1}{p^{z-u}}+\frac{1}{p^{2z}}-\frac{1}{p^{3z-u}}\right)+\\
+p^{3/2-z}\left(1-\frac{1}{p^{1+2z}}\right)^{-1}\left(1-\frac{1}{p^{z-u}}\right)\left(1-\frac{1}{p^{2z}}\right).
\end{multline}

Substituting \eqref{SG(p) G11 1p} and \eqref{SG(p) Gp1} into \eqref{SG(p) eq1}, we obtain
\begin{multline}\label{SG(p) eq2}
\SG(p,u,v,z)=\left(1-\frac{1}{p^{1+2z}}\right)^{-1}\Biggl(
1-\frac{1}{p^{2z}}+\frac{1}{p^{z+v}}-\frac{1}{p^{1+z+v}}+\\+
\frac{1}{p^{1/2+v}}
\left(1-\frac{1}{p^{z-u}}\right)\left(1-\frac{1}{p^{2z}}\right)
\Biggr):=\left(1-\frac{1}{p^{1+2z}}\right)^{-1}\SG_1(p,u,v,z).
\end{multline}
Similarly, substituting \eqref{SG(p) G11 1p}, \eqref{SG(p) Gp1}, \eqref{SG(p2) Gpp}, and \eqref{SG(p2) G1p2} into \eqref{SG(p2) eq1} yields
\begin{multline}\label{SG(p2) eq2}
\SG(p^2,u,v,z)=\left(1-\frac{1}{p^{1+2z}}\right)^{-1}\Biggl(
1-\frac{1}{p^{2z}}+
\frac{1}{p^{1/2+v}}\left(1-\frac{1}{p^{z-u}}\right)\left(1-\frac{1}{p^{2z}}\right)+\\+
\left(\frac{1}{p^{2v}}-\frac{1}{p^{1+2v}}\right)\left(\frac{1}{p^{z-u}}+\frac{1}{p^{2z}}-\frac{1}{p^{3z-u}}\right)+\\+
\frac{1}{p^{1/2+z+2v}}\left(1-\frac{1}{p^{z-u}}\right)\left(1-\frac{1}{p^{2z}}\right)
\Biggr):=\left(1-\frac{1}{p^{1+2z}}\right)^{-1}\SG_2(p,u,v,z).
\end{multline}

%%%%%%%%%%%%%%%%%%%

Now, to evaluate \eqref{SZE+SIN sum e def}, we consider
\begin{multline}\label{SZE+SIN sum e eq1}
\ES_0(r,k,u,v,z):=\sum_{e_1e_2|r}\frac{|\mu(e_1)|}{\sqrt{e_1}}\left(\frac{r}{ke_1e_2^2}\right)^{z}
\SG\left(\frac{r}{e_2},u,v,z\right)=\\=(rk)^{-z}\sum_{e_3|r}e_3^{2z}\SG\left(e_3,u,v,z\right)
\sum_{e_1|e_3}\frac{|\mu(e_1)|}{e_1^{1/2+z}}.
\end{multline}
Since $\SG(R,u,v,z)$ is multiplicative in $R$, we obtain
\begin{multline}\label{SZE+SIN sum e eq2}
\ES_0(r,k,u,v,z)=
(rk)^{-z}\prod_{p|r_1}\left(1+\left(1+\frac{1}{p^{1/2+z}}\right)p^{2z}\SG\left(p,u,v,z\right)\right)\times\\\times
\prod_{p|r_2}\left(1+\left(1+\frac{1}{p^{1/2+z}}\right)p^{2z}\SG\left(p,u,v,z\right)+\left(1+\frac{1}{p^{1/2+z}}\right)p^{4z}\SG\left(p^2,u,v,z\right)\right)=\\=
\left(\frac{r}{k}\right)^{z}
\prod_{p|r_1}\left(\frac{1}{p^{2z}}+\left(1+\frac{1}{p^{1/2+z}}\right)\SG\left(p,u,v,z\right)\right)\times\\\times
\prod_{p|r_2}\left(\frac{1}{p^{4z}}+\left(1+\frac{1}{p^{1/2+z}}\right)p^{-2z}\SG\left(p,u,v,z\right)+\left(1+\frac{1}{p^{1/2+z}}\right)
\SG\left(p^2,u,v,z\right)\right).
\end{multline}

Using \eqref{SG(p) eq2} and \eqref{SG(p2) eq2}, we have
\begin{multline}\label{SZE+SIN sum e eq3}
\ES_0(r,k,u,v,z)=\left(\frac{r}{k}\right)^{z}
\prod_{p|r_1}\left(\left(1-\frac{1}{p^{1/2+z}}\right)^{-1}\ES_1(p,u,v,z)\right)\\\times
\prod_{p|r_2}\left(1-\frac{1}{p^{1/2+z}}\right)^{-1}\ES_2(p,u,v,z),
\end{multline}
where
\begin{equation}\label{ES1 def}
\ES_1(p,u,v,z):=
\frac{1}{p^{2z}}-\frac{1}{p^{1/2+3z}}+\SG_1\left(p,u,v,z\right),
\end{equation}
and
\begin{multline}\label{ES2 def}
\ES_2(p,u,v,z):=
\frac{1}{p^{4z}}-\frac{1}{p^{1/2+5z}}+\frac{1}{p^{2z}}\SG_1\left(p,u,v,z\right)+
\SG_2\left(p,u,v,z\right)=\\=\frac{\ES_1(p,u,v,z)}{p^{2z}}+\SG_2\left(p,u,v,z\right).
\end{multline}
Applying \eqref{SG(p) eq2}, we obtain
\begin{equation}\label{ES1 eq1}
\ES_1(p,u,v,z)=
1-\frac{1}{p^{1/2+3z}}+\frac{1}{p^{z+v}}-\frac{1}{p^{1+z+v}}+
\frac{1}{p^{1/2+v}}
\left(1-\frac{1}{p^{z-u}}\right)\left(1-\frac{1}{p^{2z}}\right),
\end{equation}
\begin{equation}\label{ES1 eq2}
\ES_1(p,0,0,z)=
1+\frac{1}{p^{z}}+\frac{1}{p^{1/2}}-\frac{1}{p^{1/2+z}}-\frac{1}{p^{1/2+2z}}-\frac{1}{p^{1+z}}.
\end{equation}
Differentiating \eqref{ES1 eq1} then yields
\begin{equation}\label{ES1 u deriv}
\frac{d}{du}\ES_1(p,u,0,z)\Bigl|_{u=0}=
\frac{-\log p}{p^{1/2+z}}\left(1-\frac{1}{p^{2z}}\right),
\end{equation}
\begin{equation}\label{ES1 v deriv}
\frac{d}{dv}\ES_1(p,0,v,z)\Bigl|_{v=0}=
\frac{-\log p}{p^{z}}+\frac{\log p}{p^{1+z}}-
\frac{\log p}{p^{1/2}}
\left(1-\frac{1}{p^{z}}\right)\left(1-\frac{1}{p^{2z}}\right).
\end{equation}
Differentiating \eqref{ES2 def} by virtue of \eqref{SG(p2) eq2}, \eqref{ES1 u deriv}, and \eqref{ES1 v deriv}, we find
\begin{multline}\label{ES2 u deriv}
\frac{d}{du}\ES_2(p,u,0,z)\Bigl|_{u=0}=\\=-
\left(1-\frac{1}{p^{2z}}\right)\left(
\frac{1}{p^{1/2+3z}}+
\frac{1}{p^{1/2+z}}-
\left(1-\frac{1}{p}\right)\frac{1}{p^z}+\frac{1}{p^{1/2+2z}}
\right)\log p=\\=-
\left(-\frac{1}{p^{z}}+\frac{1}{p^{3z}}+\frac{1}{p^{1/2+z}}+\frac{1}{p^{1/2+2z}}-\frac{1}{p^{1/2+4z}}-\frac{1}{p^{1/2+5z}}+
\frac{1}{p^{1+z}}-\frac{1}{p^{1+3z}}
\right)\log p,
\end{multline}
\begin{multline}\label{ES2 v deriv}
\frac{d}{dv}\ES_2(p,0,v,z)\Bigl|_{v=0}=
-\frac{\log p}{p^{3z}}\left(1-\frac{1}{p}\right)-
\frac{\log p}{p^{1/2+2z}}
\left(1-\frac{1}{p^{z}}\right)\left(1-\frac{1}{p^{2z}}\right)-\\-
\frac{\log p}{p^{1/2}}\left(1-\frac{1}{p^{z}}\right)\left(1-\frac{1}{p^{2z}}\right)-2\log p\left(1-\frac{1}{p}\right)
\left(\frac{1}{p^{z}}+\frac{1}{p^{2z}}-\frac{1}{p^{3z}}\right)-\\-\frac{2\log p}{p^{1/2+z}}
\left(1-\frac{1}{p^{z}}\right)\left(1-\frac{1}{p^{2z}}\right)=
-\left(1-\frac{1}{p}\right)\left(\frac{2}{p^{z}}+\frac{2}{p^{2z}}-\frac{1}{p^{3z}}\right)\log p-\\-
\left(\frac{1}{p^{1/2}}+\frac{2}{p^{1/2+z}}+\frac{1}{p^{1/2+2z}}\right)
\left(1-\frac{1}{p^{z}}\right)\left(1-\frac{1}{p^{2z}}\right)\log p=\\=
-\Bigl(\frac{2}{p^{z}}+\frac{2}{p^{2z}}-\frac{1}{p^{3z}}+
\frac{1}{p^{1/2}}+\frac{1}{p^{1/2+z}}-\frac{2}{p^{1/2+2z}}-\frac{2}{p^{1/2+3z}}+\frac{1}{p^{1/2+4z}}+\frac{1}{p^{1/2+5z}}-\\
-\frac{2}{p^{1+z}}-\frac{2}{p^{1+2z}}+\frac{1}{p^{1+3z}}
  \Bigr)\log p.
\end{multline}
It follows from \eqref{SG(p2) eq2} that
\begin{equation}\label{SG2(p,0,0)}
\SG_2(p,0,0,z)=
1+\frac{1}{p^{z}}-\frac{1}{p^{3z}}+\frac{1}{p^{1/2}}-\frac{2}{p^{1/2+2z}}+\frac{1}{p^{1/2+4z}}-\frac{1}{p^{1+z}}-\frac{1}{p^{1+2z}}+
\frac{1}{p^{1+3z}}.
\end{equation}
Therefore, using \eqref{ES1 eq2} and \eqref{SG2(p,0,0)}, we rewrite \eqref{ES2 def} as
\begin{multline}\label{ES2 eq2}
\ES_2(p,0,0,z)=
\frac{1}{p^{2z}}+\frac{1}{p^{3z}}+\frac{1}{p^{1/2+2z}}-\frac{1}{p^{1/2+3z}}-\frac{1}{p^{1/2+4z}}-\frac{1}{p^{1+3z}}
+\SG_2\left(p,0,0,z\right)=\\=
1+\frac{1}{p^{z}}+\frac{1}{p^{2z}}+\frac{1}{p^{1/2}}-\frac{1}{p^{1/2+2z}}-\frac{1}{p^{1/2+3z}}-\frac{1}{p^{1+z}}-\frac{1}{p^{1+2z}}.
\end{multline}

%%%%%%%%%%%%5
Differentiating \eqref{SZE+SIN rK MT+ET3}, we find
\begin{multline}\label{SZE+SIN sum e u deriv}
\frac{d}{du}\ES_0(r,k,u,0,z)\Biggl|_{u=0}=\left(\frac{r}{k}\right)^{z}
\prod_{p|r_1}\left(\left(1-\frac{1}{p^{1/2+z}}\right)^{-1}\ES_1(p,0,0,z)\right)\times\\\times
\prod_{p|r_2}\left(1-\frac{1}{p^{1/2+z}}\right)^{-1}\ES_2(p,0,0,z)
\Biggl(\sum_{p|r_1}\frac{\frac{d}{du}\ES_1(p,u,0,z)\Bigl|_{u=0}}{\ES_1(p,0,0,z)}+
\sum_{p|r_2}\frac{\frac{d}{du}\ES_2(p,u,0,z)\Bigl|_{u=0}}{\ES_2(p,0,0,z)}
\Biggr),
\end{multline}
and
\begin{multline}\label{SZE+SIN sum e v deriv}
\frac{d}{dv}\ES_0(r,k,0,v,z)\Biggl|_{v=0}=\left(\frac{r}{k}\right)^{z}
\prod_{p|r_1}\left(\left(1-\frac{1}{p^{1/2+z}}\right)^{-1}\ES_1(p,0,0,z)\right)\times\\\times
\prod_{p|r_2}\left(1-\frac{1}{p^{1/2+z}}\right)^{-1}\ES_2(p,0,0,z)
\Biggl(\sum_{p|r_1}\frac{\frac{d}{dv}\ES_1(p,0,v,z)\Bigl|_{v=0}}{\ES_1(p,0,0,z)}+
\sum_{p|r_2}\frac{\frac{d}{dv}\ES_2(p,0,v,z)\Bigl|_{v=0}}{\ES_2(p,0,0,z)}
\Biggr).
\end{multline}
Using \eqref{SZE+SIN sum e eq1}, \eqref{SZE+SIN sum e eq3}, \eqref{SZE+SIN sum e u deriv}, and \eqref{SZE+SIN sum e v deriv}, we rewrite \eqref{SZE+SIN sum e def}  in the form
\begin{multline}\label{SZE+SIN sum e final1}
\ES(r,k,z)=
\left(\frac{r}{k}\right)^{z}
\prod_{p|r_1}\left(\left(1-\frac{1}{p^{1/2+z}}\right)^{-1}\ES_1(p,0,0,z)\right)
\prod_{p|r_2}\left(1-\frac{1}{p^{1/2+z}}\right)^{-1}\\\times
\ES_2(p,0,0,z)\Biggl(\Cc_1(-z)+
2\sum_{p|r_1}\frac{\frac{d}{du}\ES_1(p,u,0,z)\Bigl|_{u=0}}{\ES_1(p,0,0,z)}+
2\sum_{p|r_2}\frac{\frac{d}{du}\ES_2(p,u,0,z)\Bigl|_{u=0}}{\ES_2(p,0,0,z)}+\\
+2\sum_{p|r_1}\frac{\frac{d}{dv}\ES_1(p,0,v,z)\Bigl|_{v=0}}{\ES_1(p,0,0,z)}+
2\sum_{p|r_2}\frac{\frac{d}{dv}\ES_2(p,0,v,z)\Bigl|_{v=0}}{\ES_2(p,0,0,z)}
\Biggr).
\end{multline}
We are now ready to establish the identity
\begin{equation}\label{ES to MAP}
\ES(r,k,z)=\left(\frac{k}{r}\right)^{-z}\MAP(-z,r)
\end{equation}
defined in \eqref{MT sum e eq3}, namely
\begin{multline}\label{MT sum e eq4}
\MAP(-z,r)=
\prod_{p|r_1}\left(1+\frac{1}{p^{z}}+\frac{1}{p^{1/2}}\right)
\prod_{p|r_2}\left(1+\frac{1}{p^{z}}+\frac{1}{p^{2z}}+\frac{1}{p^{1/2}}+\frac{1}{p^{1/2+z}}\right)
\Biggl(\Cc_1(-z)-\\-
2\sum_{p|r_1}\frac{p^{-z}+p^{-1/2}}{1+p^{-z}+p^{-1/2}}\log p
-2\sum_{p|r_2}\frac{p^{-z}+p^{-1/2}+2p^{-2z}+2p^{-1/2-z}}{1+p^{-z}+p^{-1/2}+p^{-2z}+p^{-1/2-z}}\log p
\Biggr).
\end{multline}
First, it follows from \eqref{ES1 eq2} and \eqref{ES2 eq2} that
\begin{equation}\label{ES1 eq3}
\ES_1(p,0,0,z)=
\left(1-\frac{1}{p^{1/2+z}}\right)\left(1+\frac{1}{p^{z}}+\frac{1}{p^{1/2}}\right),
\end{equation}
\begin{equation}\label{ES2 eq3}
\ES_2(p,0,0,z)
=\left(1-\frac{1}{p^{1/2+z}}\right)\left(1+\frac{1}{p^{z}}+\frac{1}{p^{2z}}+\frac{1}{p^{1/2}}+\frac{1}{p^{1/2+z}}\right).
\end{equation}
Substituting \eqref{ES1 eq3} and \eqref{ES2 eq3} into \eqref{SZE+SIN sum e final1}, we show that
\begin{multline}\label{SZE+SIN sum e final2}
\ES(r,k,z)=
\left(\frac{r}{k}\right)^{z}
\prod_{p|r_1}\left(1+\frac{1}{p^{z}}+\frac{1}{p^{1/2}}\right)
\prod_{p|r_2}\left(1+\frac{1}{p^{z}}+\frac{1}{p^{2z}}+\frac{1}{p^{1/2}}+\frac{1}{p^{1/2+z}}\right)\\\times
\Biggl(\Cc_1(-z)+
2\sum_{p|r_1}\frac{\frac{d}{du}\ES_1(p,u,0,z)\Bigl|_{u=0}+
\frac{d}{dv}\ES_1(p,0,v,z)\Bigl|_{v=0}}{(1+p^{-z}+p^{-1/2})\left(1-p^{-1/2-z}\right)}+\\+
2\sum_{p|r_2}\frac{\frac{d}{du}\ES_2(p,u,0,z)\Bigl|_{u=0}+\frac{d}{dv}\ES_2(p,0,v,z)\Bigl|_{v=0}}{(1+p^{-z}+p^{-1/2}+p^{-2z}+p^{-1/2-z})\left(1-p^{-1/2-z}\right)}
\Biggr).
\end{multline}
Using \eqref{ES1 u deriv} and \eqref{ES1 v deriv}, we obtain
\begin{multline}\label{ES1 u+v deriv}
\frac{d}{du}\ES_1(p,u,0,z)\Bigl|_{u=0}+
\frac{d}{dv}\ES_1(p,0,v,z)\Bigl|_{v=0}=
\frac{-\log p}{p^{1/2+z}}\left(1-\frac{1}{p^{2z}}\right)-\frac{\log p}{p^{z}}+\frac{\log p}{p^{1+z}}-\\-
\frac{\log p}{p^{1/2}}\left(1-\frac{1}{p^{z}}\right)\left(1-\frac{1}{p^{2z}}\right)=
-\left(1-\frac{1}{p^{1/2+z}}\right)\left(\frac{1}{p^{z}}+\frac{1}{p^{1/2}}\right)\log p,
\end{multline}
Similarly, it follows from \eqref{ES2 u deriv} and \eqref{ES2 v deriv} that
\begin{multline}\label{ES2 u+v deriv}
\frac{d}{du}\ES_2(p,u,0,z)\Bigl|_{u=0}+
\frac{d}{dv}\ES_2(p,0,v,z)\Bigl|_{v=0}=
-\Bigl(
\frac{1}{p^{z}}+\frac{2}{p^{2z}}+\frac{1}{p^{1/2}}+\frac{2}{p^{1/2+z}}-\frac{1}{p^{1/2+2z}}-\\-\frac{2}{p^{1/2+3z}}
-\frac{1}{p^{1+z}}-\frac{2}{p^{1+2z}}\Bigr)\log p=
-\left(1-\frac{1}{p^{1/2+z}}\right)\left(
\frac{1}{p^{z}}+\frac{2}{p^{2z}}+\frac{1}{p^{1/2}}+\frac{2}{p^{1/2+z}}\right)\log p.
\end{multline}
Substituting \eqref{ES1 u+v deriv} and \eqref{ES2 u+v deriv} into \eqref{SZE+SIN sum e final2} and taking into account \eqref{MT sum e eq4}, we establish \eqref{ES to MAP}.
Therefore, we rewrite \eqref{SZE+SIN rK MT+ET8} in the form
\begin{multline}\label{SZE+SIN rK MT+ET9}
\SZE(r,K)+\SIN(r,K)=
\frac{1}{\sqrt{r}}\sum_{k}h\left(\frac{k}{K}\right)
\frac{1}{2\pi i}\int_{(a)}
\pi^{3z/2}\frac{\Gamma\left(3/4-z/2\right)}{\Gamma\left(3/4\right)}
\\\times
\zeta(1-z)\zeta(1-2z)\left(\frac{k}{r}\right)^{-z}\MAP(-z,r)\left(1+O\left(\frac{|z|}{k}\right)\right)
\frac{dz}{z}+O\left(\frac{r^{3/2}}{K^{3/2-\epsilon}}\right).
\end{multline}
%%%%%%%%%%%%%%%%%%%%%%%%%%%%%%%%%%%%%%%%%%%%%%%%

\section{Proof of Theorem \ref{thm:2mom average}: the special case}\label{sec:Proof of Theorem 2mom}

It follows from \eqref{SZE+SIN rK MT+ET9} and \eqref{2mom MT2 eq5} that
\begin{multline}\label{MT2+SZE+SIN eq0}
\MT_2(r,K)+\SZE(r,K)+\SIN(r,K)=\frac{1}{\sqrt{r}}\sum_{k}h\left(\frac{k}{K}\right)
\frac{1}{2\pi i}\int_{(a)}\left(H(z)+H(-z)\right)\frac{dz}{z}+\\+O\left(\frac{r^{3/2}}{K^{3/2-\epsilon}}+\frac{(rK)^{\epsilon}}{\sqrt{r}}\right),
\end{multline}
where
\begin{equation}\label{MT2+SZE+SIN Hdef}
H(z)=\pi^{-3z/2}\frac{\Gamma\left(3/4+z/2\right)}{\Gamma\left(3/4\right)} \zeta(1+2z)\zeta(1+z)\MAP(z,r)\left(\frac{k}{r}\right)^{z}.
\end{equation}
To evaluate the integral in \eqref{MT2+SZE+SIN eq0}, we apply the residue theorem, obtaining
\begin{align}\label{MT2+SZE+SIN eq2}
\frac{1}{2\pi i}\int_{(\epsilon)}\left(H(z)+H(-z)\right)\frac{dz}{z} &= \res_{z=0}\frac{H(z)+H(-z)}{z}+\frac{1}{2\pi i}\int_{(-\epsilon)}\left(H(z)+H(-z)\right)\frac{dz}{z} \notag \\
&= \res_{z=0}\frac{H(z)+H(-z)}{z}-\frac{1}{2\pi i}\int_{(\epsilon)}\left(H(z)+H(-z)\right)\frac{dz}{z}.
\end{align}
Therefore,
\begin{multline}\label{MT2+SZE+SIN eq1}
\MT_2(r,K)+\SZE(r,K)+\SIN(r,K)=\frac{1}{\sqrt{r}}\sum_{k}h\left(\frac{k}{K}\right)\frac{1}{2}\res_{z=0}\frac{H(z)+H(-z)}{z}+\\+O\left(\frac{r^{3/2}}{K^{3/2-\epsilon}}+\frac{(rK)^{\epsilon}}{\sqrt{r}}\right)=
\frac{1}{\sqrt{r}}\sum_{k}h\left(\frac{k}{K}\right)\res_{z=0}\frac{H(z)}{z}+O\left(\frac{r^{3/2}}{K^{3/2-\epsilon}}+\frac{(rK)^{\epsilon}}{\sqrt{r}}\right).
\end{multline}
Substituting \eqref{MT2+SZE+SIN eq1} and \eqref{2mom MTn est3} into \eqref{2 mom averaged to Zag}, we prove the following asymptotic formula for the twisted second moment \eqref{2 mom averaged}:
\begin{equation}\label{2 mom averaged final AF}
M^{a}_2(r,K)=\frac{1}{\sqrt{r}}\sum_{k}h\left(\frac{k}{K}\right)\res_{z=0}\frac{H(z)}{z}+O\left(\frac{r^{3/2}}{K^{3/2-\epsilon}}+\frac{(rK)^{\epsilon}}{\sqrt{r}}\right).
\end{equation}
In order to apply the result of Khan \cite[Theorem 1.1]{Khan2010} concerning the dependence of the non-vanishing proportion on the length of the mollifier, we must demonstrate that \eqref{2 mom averaged final AF} matches the main term in \cite[Theorem 3.1]{Khan2010}.
To this end, we apply \eqref{C1 def}, \eqref{2mom MT2 eq4}, \eqref{2mom MT2 eq5}, and \eqref{MT2+SZE+SIN Hdef}, obtaining
\begin{multline}\label{2 mom averaged final AF2}
M^{a}_2(r,K)=O\left(\frac{r^{3/2}}{K^{3/2-\epsilon}}+\frac{(rK)^{\epsilon}\sqrt{K}}{\sqrt{r}}\right)+\\+\sum_{k}h\left(\frac{k}{K}\right)\sum_{e|r^2}
\frac{1}{\sqrt{re_1}}\res_{z=0}\frac{F(z)}{z}\Biggl(
\zeta(1+z)\zeta(1+2z)\Cc_2+2\zeta'(1+z)\zeta(1+2z)
\Biggr),
\end{multline}
where
\begin{equation}\label{F1C2 def}
F(z)=\pi^{-3z/2}\frac{\Gamma\left(3/4+z/2\right)k^{z}}{\Gamma\left(3/4\right)(e_1e_2)^{z}},\quad
\Cc_2=2\log\frac{ke_2}{r}-3\log(2\pi)+3\gamma+\frac{\pi}{2}.
\end{equation}
Note that
\begin{equation}
\frac{\zeta(1+z)\zeta(1+2z)}{z}=\frac{1}{2z^3}+\frac{3\gamma}{2z^2}+\frac{\gamma^2-5\gamma_1/2}{z}+O(1),
\end{equation}
\begin{equation}
\frac{\zeta'(1+z)\zeta(1+2z)}{z}=
-\frac{1}{2z^4}-\frac{\gamma}{z^3}+\frac{3\gamma_1}{2z^2}-\frac{\gamma\gamma_1+3\gamma_2/2}{z}+O(1),
\end{equation}
which yields
\begin{multline}\label{2 mom averaged final AF3}
\res_{z=0}\frac{F(z)}{z}\Bigl(\zeta(1+z)\zeta(1+2z)\Cc_2+2\zeta'(1+z)\zeta(1+2z)\Bigr)=\\=
\Cc_2\left(\frac{F''(0)}{4}+\frac{3\gamma}{2}F'(0)+\left(\gamma^2-5\gamma_1/2\right)F(0)\right)-
\frac{F'''(0)}{6}-\gamma F''(0)+3\gamma_1F'(0)-\\-(2\gamma\gamma_1+3\gamma_2)F(0)=
\frac{1}{2}\log\frac{ke_2}{r}\log^2\frac{k}{e_1e_2}-\frac{1}{6}\log^3\frac{k}{e_1e_2}+
\log\frac{ke_2}{r}P_1(\log\frac{k}{e_1e_2})+P_2(\log\frac{k}{e_1e_2}),
\end{multline}
where $P_j(x)$ are some polynomials of degree $j.$ Substituting \eqref{2 mom averaged final AF3} into \eqref{2 mom averaged final AF2} and evaluating the sum over $k$ by means of the Poisson summation formula, we obtain \eqref{2momAF0}, which completes the proof of Theorem \ref{thm:2mom average}.

%%%%%%%%%%%%%%%%%%%%%%%%%%%%%%%%%%%%%%%%%%%%%%%%%%%%%%%%%%%%%%%%%%%%%%%%%%%%%%%%%%%%%%%%%%%%%%%%%%%%%%%%%%%%%%%%%%%%%%%%%%
\section{Proof of Theorem \ref{thm:2mom average}: the general case}\label{sec:Proof of Theorem 2mom General}
In this section, we will prove that \eqref{ES to MAP} holds for an arbitrary $r.$ To this end, we will employ the same Euler product strategy as in the verification of \eqref{ES to MAP} for specific $r$. The proof will proceed by induction, where the computations in Section~\ref{sec:Voronoi MT} establish the base case.

It follows from \eqref{MT sum e def} and \eqref{SZE+SIN sum e def} that, to establish \eqref{ES to MAP}, it suffices to show that
\begin{equation}\label{E0 to sum e}
\left(\frac{k}{r}\right)^{z}\ES_0(r,k,0,0,z)=
\sum_{e_3|r}e_3^{-z}\sum_{e_1|e_3}\frac{|\mu(e_1)|}{e_1^{1/2-z}},
\end{equation}
\begin{multline}\label{sum G deriv to sum e deriv}
\sum_{e_1e_2|r}\frac{|\mu(e_1)|}{\sqrt{e_1}}\left(\frac{1}{e_1e_2^2}\right)^{z}\Biggl(
\frac{d}{du}\SG\left(\frac{r}{e_2},u,0,z\right)\Biggl|_{u=0}+
\frac{d}{dv}\SG\left(\frac{r}{e_2},0,v,z\right)\Biggl|_{v=0}
\Biggr)=\\=
\sum_{e_3|r}
\left(-\frac{d}{du}\right)e_3^{-z+u}\Bigl|_{u=0}
\sum_{e_1|e_3}\frac{|\mu(e_1)|}{e_1^{1/2+z}},
\end{multline}
where  $\ES_0(r,k,u,v,z)$ is given by \eqref{SZE+SIN sum e eq1}.
Let us first consider \eqref{E0 to sum e}.
To simplify the notation, let $\nu_p(r)$ denote the $p$-adic valuation of $r$.
We then have
\begin{multline}\label{E0 to sum e eq1}
\sum_{e_3|r}e_3^{-z}\sum_{e_1|e_3}\frac{|\mu(e_1)|}{e_1^{1/2-z}}=\prod_{p|r}\left(1+\sum_{j=1}^{\nu_p(r)}\frac{1}{p^{jz}}\left(1+\frac{1}{p^{1/2-z}}\right)\right)=\\=
\prod_{p|r}\left(1+\frac{1}{p^{z}}\left(1+\frac{1}{p^{1/2-z}}\right)
\frac{1-p^{-\nu_p(r)z}}{1-p^{-z}}
\right).
\end{multline}
Arguing as in the proof of \eqref{SZE+SIN sum e eq2}, we find that the left-hand side of \eqref{E0 to sum e} is equal to
\begin{multline}\label{E0 to sum e eq2}
\left(\frac{k}{r}\right)^{z}\ES_0(r,k,u,v,z)=
\sum_{e_3|r}\frac{e_3^{2z}}{r^{2z}}\SG\left(e_3,u,v,z\right)
\sum_{e_1|e_3}\frac{|\mu(e_1)|}{e_1^{1/2+z}}=\\=
\prod_{p|r}\frac{1}{p^{2\nu_p(r)z}}\left(1+
\left(1+\frac{1}{p^{1/2+z}}\right)\sum_{j=1}^{\nu_p(r)}p^{2jz}
\SG\left(p^{j},u,v,z\right)\right):=\prod_{p|r}\ES_p(\nu_p(r),u,v,z).
\end{multline}
Substituting \eqref{E0 to sum e eq1} and \eqref{E0 to sum e eq2} into \eqref{E0 to sum e}, we find that it suffices to prove the following result.
%%%%%%%%%%%%%%%%%%%%%%%%%%%%%%%%%%%%%%%%%%%%%%%%%%%%%%%%%%%%%%%%%%%%%%%%%%%%%%%%%%%%%%%%
\begin{lem}\label{lem:sum SG pj}
The following identity holds:
\begin{equation}\label{sum SG pj}
1+\left(1+\frac{1}{p^{1/2+z}}\right)\sum_{j=1}^{\nu}p^{2jz}
\SG\left(p^{j},0,0,z\right)=
p^{2\nu z}\left(1+\frac{1}{p^{z}}\left(1+\frac{1}{p^{1/2-z}}\right)
\frac{1-p^{-\nu z}}{1-p^{-z}}
\right).
\end{equation}
\end{lem}
\begin{proof}
We will prove \eqref{sum SG pj} by induction on $\nu$. The case $\nu=1$ follows from \eqref{SG(p) eq2}, \eqref{ES1 def}, and \eqref{ES1 eq2}. Suppose that \eqref{sum SG pj} holds for $\nu$; we shall show that it also holds for $\nu+1$. For this purpose, it suffices to show that
\begin{multline}\label{SG p nu+1 eq1}
\left(1+\frac{1}{p^{1/2+z}}\right)
\SG\left(p^{\nu+1},0,0,z\right)=
1+\frac{1}{p^{z}}\left(1+\frac{1}{p^{1/2-z}}\right)\frac{1-p^{-(\nu+1)z}}{1-p^{-z}} - \\
- \frac{1}{p^{2z}}\left(1+\frac{1}{p^{z}}\left(1+\frac{1}{p^{1/2-z}}\right)
\frac{1-p^{-\nu z}}{1-p^{-z}}\right).
\end{multline}
Straightforward calculations allow us to rewrite \eqref{SG p nu+1 eq1} as
\begin{multline}\label{SG p nu+1 eq2}
\left(1+\frac{1}{p^{1/2+z}}\right)
\SG\left(p^{\nu+1},0,0,z\right)= \\
= 1-\frac{1}{p^{2z}}+\frac{1}{p^{z}}\left(1+\frac{1}{p^{1/2-z}}\right)\left(1+\frac{1}{p^{z}}\right)-\frac{1}{p^{(\nu+2)z}}\left(1+\frac{1}{p^{1/2-z}}\right).
\end{multline}
To prove \eqref{SG p nu+1 eq2}, we again employ the induction principle. For $\nu=0$ and $\nu=1$, equation \eqref{SG p nu+1 eq2} follows from \eqref{SG(p) eq2} and \eqref{SG(p2) eq2}. Using \eqref{SZE+SIN md sum def}, we deduce that
\begin{equation}\label{SG p nu+1 to p nu0}
\SG(p^{\nu},u,v,z)=\frac{\G(1,p^{\nu},z,u)}{p^{\nu(1+v)}}+
\sum_{j=0}^{\nu-1}\sum_{m|p}\frac{\mu(m)\G(m,p^{j},z,u)}{m^{1+z}p^{j(1+v)}},
\end{equation}
\begin{multline}\label{SG p nu+1 to p nu}
\SG(p^{\nu+1},u,v,z)=\frac{\G(1,p^{\nu+1},z,u)}{p^{(\nu+1)(1+v)}}+
\sum_{m|p}\frac{\mu(m)\G(m,p^{\nu},z,u)}{m^{1+z}p^{\nu(1+v)}}+ \\
+ \sum_{j=0}^{\nu-1}\sum_{m|p}\frac{\mu(m)\G(m,p^{j},z,u)}{m^{1+z}p^{j(1+v)}}=
\frac{\G(1,p^{\nu+1},z,u)}{p^{(\nu+1)(1+v)}}-
\frac{\G(p,p^{\nu},z,u)}{p^{\nu(1+v)}p^{1+z}}+\SG(p^{\nu},u,v,z).
\end{multline}
Therefore, assuming that \eqref{SG p nu+1 eq2} holds with $\nu$ replaced by $\nu-1$, we deduce from \eqref{SG p nu+1 to p nu} that \eqref{SG p nu+1 eq2} is equivalent to the identity
\begin{equation}\label{G(1,pnu+1)-G(p,pnu) eq1}
\left(1+\frac{1}{p^{1/2+z}}\right)\left(
\frac{\G(1,p^{\nu+1},z,0)}{p^{\nu+1}}-\frac{\G(p,p^{\nu},z,0)}{p^{1+z+\nu}}\right)
=\frac{1}{p^{(\nu+1)z}}\left(1+\frac{1}{p^{1/2-z}}\right)
\left(1-\frac{1}{p^{z}}\right).
\end{equation}
It follows from \eqref{Sound series 0}, \eqref{Sound series 1}, and \eqref{Sp(0,0)} that
\begin{multline}\label{G(1,pnu+1)-G(p,pnu) eq2}
\frac{\G(1,p^{\nu+1},z,0)}{p^{\nu+1}}-\frac{\G(p,p^{\nu},z,0)}{p^{1+z+\nu}}=
\frac{(1-p^{-z})(1-p^{-2z})}{(1-p^{-1-2z})p^{\nu+1}} \\
\times \sum_{q=0}^{\infty}\frac{1}{p^{q(1+z)}}\left(
\sum_{n=0}^{\infty}\frac{G(p^{2n+2q},p^{q+\nu+1})}{p^{nz}}-
\sum_{n=0}^{\infty}\frac{G(p^{2(n+1)+2q},p^{q+\nu+1})}{p^{(n+1)z}}\right)= \\
= \frac{(1-p^{-z})(1-p^{-2z})}{(1-p^{-1-2z})p^{\nu+1}}
\sum_{q=0}^{\infty}\frac{G(p^{2q},p^{q+\nu+1})}{p^{q(1+z)}}.
\end{multline}
It follows from \eqref{G(n, p gamma)} and \eqref{G(n, 2 gamma)} that
\begin{multline}\label{G(1,pnu+1)-G(p,pnu) eq3}
\sum_{q=0}^{\infty}\frac{G(p^{2q},p^{q+\nu+1})}{p^{q(1+z)}}=
\frac{p^{2\nu+1/2}}{p^{\nu(1+z)}}+
\sum_{\substack{q=\nu+1\\q+\nu+1\equiv0\Mod{2}}}^{\infty}\frac{\phi(p^{\nu+1+q})}{p^{q(1+z)}}= \\
= \frac{p^{2\nu+1/2}}{p^{\nu(1+z)}}+
\frac{(1-p^{-1})p^{2\nu+2}}{(1-p^{-2z})p^{(\nu+1)(1+z)}}.
\end{multline}
Using \eqref{G(1,pnu+1)-G(p,pnu) eq2} and \eqref{G(1,pnu+1)-G(p,pnu) eq3}, we obtain
\begin{multline}\label{G(1,pnu+1)-G(p,pnu) eq4}
\left(1+\frac{1}{p^{1/2+z}}\right)\left(
\frac{\G(1,p^{\nu+1},z,0)}{p^{\nu+1}}-\frac{\G(p,p^{\nu},z,0)}{p^{1+z+\nu}}\right)= \\
= \frac{(1-p^{-z})(1-p^{-2z})}{(1-p^{-1/2-z})}
\left(\frac{p^{\nu-1/2}}{p^{\nu(1+z)}}+
\frac{(1-p^{-1})p^{\nu+1}}{(1-p^{-2z})p^{(\nu+1)(1+z)}}\right)= \\
= \frac{(1-p^{-z})}{(1-p^{-1/2-z})p^{(\nu+1)z}}
\left(p^{z-1/2}(1-p^{-2z})+1-p^{-1}\right)=\frac{1}{p^{(\nu+1)z}}\left(1+\frac{1}{p^{1/2-z}}\right)\left(1-\frac{1}{p^{z}}\right).
\end{multline}
Therefore, \eqref{G(1,pnu+1)-G(p,pnu) eq1} is proved, which finalizes the proof of the lemma.

\end{proof}
Since we have completed the proof of \eqref{E0 to sum e}, it remains to prove \eqref{sum G deriv to sum e deriv}. Let us first evaluate the right-hand side of \eqref{sum G deriv to sum e deriv}. Similarly to \eqref{E0 to sum e eq1}, we obtain
\begin{multline}\label{sum e deriv}
\sum_{e_3|r}\left(\frac{d}{du}\right)e_3^{-z+u}\Bigl|_{u=0}
\sum_{e_1|e_3}\frac{|\mu(e_1)|}{e_1^{1/2+z}}=\\=
\frac{d}{du}\Biggl(
\prod_{p|r}\left(1+\frac{1}{p^{z-u}}\left(1+\frac{1}{p^{1/2-z}}\right)
\frac{1-p^{-\nu_p(r)(z-u)}}{1-p^{-(z-u)}}\right)\Biggr)\Bigl|_{u=0}=\\=
\prod_{p|r}\left(1+\frac{1}{p^{z}}\left(1+\frac{1}{p^{1/2-z}}\right)
\frac{1-p^{-\nu_p(r)z}}{1-p^{-z}}\right)\sum_{p|r}\frac{(1+p^{-1/2+z})\log p}{p^z\left(1+\frac{1}{p^{z}}\left(1+\frac{1}{p^{1/2-z}}\right)
\frac{1-p^{-\nu_p(r)z}}{1-p^{-z}}\right)}\\\times\left(\frac{1-p^{-\nu_p(r)z}}{1-p^{-z}}-
\frac{\nu p^{-\nu_p(r)z}}{1-p^{-z}}+\frac{1-p^{-\nu_p(r)z}}{(1-p^{-z})^2p^z}
\right)=\prod_{p|r}\Bigl(1+\frac{1}{p^{z}}\left(1+\frac{1}{p^{1/2-z}}\right)
\frac{1-p^{-\nu_p(r)z}}{1-p^{-z}}\Bigr)\\\times\sum_{p|r}\frac{(1+p^{-1/2+z})(1-(\nu+1)p^{-\nu_p(r)z}+\nu p^{-(\nu_p(r)+1)z})\log p}{p^z\Bigl(1+\frac{1}{p^{z}}\left(1+\frac{1}{p^{1/2-z}}\right)
\frac{1-p^{-\nu_p(r)z}}{1-p^{-z}}\Bigr)(1-p^{-z})^2}.
\end{multline}

It follows from \eqref{SZE+SIN sum e eq1} that the left-hand side of \eqref{sum G deriv to sum e deriv} is equal to
\begin{equation}
\Biggl(\frac{d}{du}+\frac{d}{dv}\Biggr)
\left(\frac{k}{r}\right)^{z}\ES_0(r,k,u,v,z)\Biggl|_{u=v=0}.
\end{equation}
Evaluating the derivative of \eqref{E0 to sum e eq2}, we find that the left-hand side of \eqref{sum G deriv to sum e deriv} equals
\begin{multline}\label{sum G deriv}
\prod_{p|r} \ES_p(\nu_p(r),0,0,z)\sum_{p|r}\frac{1}{ \ES_p(\nu_p(r),0,0,z)} \\
\times \frac{1+p^{-1/2-z}}{p^{2\nu_p(r)z}}
\Bigl(\frac{d}{du}+\frac{d}{dv}\Bigr)\sum_{j=1}^{\nu_p(r)}p^{2jz}
\SG\left(p^{j},u,v,z\right)\Biggl|_{u=v=0}.
\end{multline}
By combining \eqref{sum e deriv}, \eqref{sum G deriv}, and Lemma \ref{lem:sum SG pj}, we conclude that \eqref{sum G deriv to sum e deriv} holds provided that the following identity is satisfied.

%%%%%%%%%%%%%%%%%%%%%%%%%%%%%%%%%%%%%%%%%%%%%%%%%%%%%%%%%%%%%%%%%%%%%%%%%%%%%%%%%%%%%%
\begin{lem}\label{lem:deriv sum SG pj}
The following identity holds:
\begin{multline}\label{deriv sum SG pj}
\left(1+\frac{1}{p^{1/2+z}}\right)\Bigl(\frac{d}{du}+\frac{d}{dv}\Bigr)\sum_{j=1}^{\nu}p^{2jz}
\SG\left(p^{j},u,v,z\right)\Biggl|_{u=v=0}=\\=-
\frac{p^{2\nu z}(1+p^{-1/2+z})\log p}{p^z(1-p^{-z})^2}
\left(1-\frac{\nu+1}{p^{\nu z}}+\frac{\nu}{p^{(\nu+1)z}}\right).
\end{multline}
\end{lem}
\begin{proof}
We will prove \eqref{deriv sum SG pj} by induction. The case $\nu=1$ follows from \eqref{SG(p) eq2}, \eqref{ES1 def}, and \eqref{ES1 u+v deriv}. Suppose that \eqref{deriv sum SG pj} holds for $\nu$; then it holds for $\nu+1$ provided that
\begin{multline}\label{deriv SG pj eq1}
\left(1+\frac{1}{p^{1/2+z}}\right)\Bigl(\frac{d}{du}+\frac{d}{dv}\Bigr)p^{2(\nu+1)z}
\SG\left(p^{\nu+1},u,v,z\right)\Biggl|_{u=v=0} = \\
= \frac{p^{2(\nu+1) z}(1+p^{-1/2+z})\log p}{p^z(1-p^{-z})^2}
\left(-1+\frac{1}{p^{2z}}+\frac{\nu+2}{p^{(\nu+1)z}}-\frac{2\nu+2}{p^{(\nu+2)z}}+\frac{\nu}{p^{(\nu+3)z}}\right).
\end{multline}
Factoring out the factor $(1-p^{-z})$ from the last bracket, we find that it suffices to prove
\begin{multline}\label{deriv SG pj eq2}
\left(1+\frac{1}{p^{1/2+z}}\right)\Bigl(\frac{d}{du}+\frac{d}{dv}\Bigr)
\SG\left(p^{\nu+1},u,v,z\right)\Biggl|_{u=v=0} = \\
= \frac{(1+p^{-1/2+z})\log p}{p^z(1-p^{-z})}
\left(-1-\frac{1}{p^{z}}+\frac{\nu}{p^{(\nu+1)z}}\left(1-\frac{1}{p^z}\right)+\frac{2}{p^{(\nu+1)z}}\right).
\end{multline}
Next, we prove \eqref{deriv SG pj eq2} by induction. For $\nu=0$, equation \eqref{deriv SG pj eq2} follows from \eqref{SG(p) eq2}. Assuming that \eqref{deriv SG pj eq2} holds with $\nu$ replaced by $\nu-1$, we deduce from \eqref{SG p nu+1 to p nu} that \eqref{deriv SG pj eq2} is equivalent to the identity
\begin{multline}\label{deriv G(1,pnu+1)-G(p,pnu) eq1}
\left(1+\frac{1}{p^{1/2+z}}\right)\Bigl(\frac{d}{du}+\frac{d}{dv}\Bigr)
\left(\frac{\G(1,p^{\nu+1},z,u)}{p^{(\nu+1)(1+v)}}-
\frac{\G(p,p^{\nu},z,u)}{p^{\nu(1+v)}p^{1+z}}\right)\Biggl|_{u=v=0} = \\
= \frac{(1+p^{-1/2+z})\log p}{p^z(1-p^{-z})}
\left(\frac{\nu}{p^{(\nu+1)z}}\left(1-\frac{1}{p^z}\right)+\frac{2}{p^{(\nu+1)z}}
-\frac{\nu-1}{p^{\nu z}}\left(1-\frac{1}{p^z}\right)-\frac{2}{p^{\nu z}}
\right).
\end{multline}
Simplifying the right-hand side of \eqref{deriv G(1,pnu+1)-G(p,pnu) eq1}, we find that it equals
\begin{equation*}
\frac{-\log p}{p^{(\nu+1)z}}
\left(1+p^{-1/2+z}\right)\left(1+\nu-\frac{\nu}{p^z}\right).
\end{equation*}
Rewriting the left-hand side of \eqref{deriv G(1,pnu+1)-G(p,pnu) eq1}, we see that it remains to prove
\begin{multline}\label{deriv G(1,pnu+1)-G(p,pnu) eq2}
\Delta\G(\nu,p,z):=\frac{1}{p^{\nu+1}}\left(\frac{d}{du}\G(1,p^{\nu+1},z,u)\Bigl|_{u=0}
-\frac{1}{p^z}\frac{d}{du}\G(p,p^{\nu},z,u)\Bigl|_{u=0}\right) - \\
- \frac{\log p}{p^{\nu+1}}\Bigl((\nu+1)\G(1,p^{\nu+1},z,0)-\frac{\nu}{p^z}\G(p,p^{\nu},z,0)\Bigr)=
\frac{-\log p}{p^{(\nu+1)z}}
\frac{(1+p^{-1/2+z})}{(1+p^{-1/2-z})}\left(1+\nu-\frac{\nu}{p^z}\right).
\end{multline}
Using \eqref{SG(p2) Gpp} and \eqref{SG(p2) G1p2}, one can verify that \eqref{deriv G(1,pnu+1)-G(p,pnu) eq2} holds for $\nu=1$. In order to prove \eqref{deriv G(1,pnu+1)-G(p,pnu) eq2} for a general $\nu$, we must evaluate $\G(1,p^{\nu+1},z,u)$ and $\G(p,p^{\nu},z,u)$ by arguing as in Lemma \ref{lem:Sound series}. It follows from \eqref{Sound series 0}, \eqref{Sound series 2}, \eqref{Sp(0,0)}, and \eqref{S2(0,0)} that
\begin{multline}\label{G1 pnu+1 eq1}
\G(1,p^{\nu+1},z,u)=\frac{\zeta(1+2z)}{\zeta(z-u)\zeta(2z)}
\sum_{n,q=1}^{\infty}\frac{G(n^2q^2,p^{\nu+1}q)}{n^{z-u}q^{1+z}} = \\
= \frac{(1-p^{-z+u})(1-p^{-2z})}{1-p^{-1-2z}}
\sum_{n,q=0}^{\infty}\frac{G(p^{2n+2q},p^{q+\nu+1})}{p^{n(z-u)+q(1+z)}}.
\end{multline}
Consider first the case where $\nu$ is even, say $\nu=2\nu_2$. Using \eqref{G(n, p gamma)} if $p>2$ (and \eqref{G(n, 2 gamma)} if $p=2$), we obtain
\begin{multline}\label{G1 pnu+1 eq2}
\sum_{n,q=0}^{\infty}\frac{G(p^{2n+2q},p^{q+\nu+1})}{p^{n(z-u)+q(1+z)}} =
\sum_{n=0}^{\nu_2}\frac{1}{p^{n(z-u)}}\left(\frac{p^{\nu+1/2}}{p^{(\nu-2n)z}} +
\sum_{\substack{q=\nu-2n+1\\q\equiv1\Mod{2}}}^{\infty}\frac{\phi(p^{q+\nu+1})}{p^{q(1+z)}}
\right) + \\
+ \sum_{n=\nu_2+1}^{\infty}\frac{1}{p^{n(z-u)}}
\sum_{\substack{q=0\\q\equiv1\Mod{2}}}^{\infty}\frac{\phi(p^{q+\nu+1})}{p^{q(1+z)}} =
\sum_{n=0}^{\nu_2}\frac{p^{\nu+1/2}}{p^{n(z-u)}p^{(\nu-2n)z}} + \\
+ \frac{p^{\nu+1}(1-p^{-1})}{(1-p^{-2z})p^z}\sum_{n=0}^{\nu_2}\frac{p^{\nu+1/2}}{p^{n(z-u)}p^{(\nu-2n)z}} + \frac{p^{\nu+1}(1-p^{-1})}{p^zp^{(\nu_2+1)(z-u)}(1-p^{-2z})(1-p^{-z+u})}.
\end{multline}
Performing the change of variables $n=\nu_2-m$, we obtain
\begin{multline}\label{G1 pnu+1 eq3}
\sum_{n,q=0}^{\infty}\frac{G(p^{2n+2q},p^{q+\nu+1})}{p^{n(z-u)+q(1+z)}} =
\left(\frac{p^{\nu+1/2}}{p^{\nu_2(z-u)}}+\frac{p^{\nu+1}(1-p^{-1})}{(1-p^{-2z})p^{z+\nu_2(z-u)}}\right)\sum_{m=0}^{\nu_2}\frac{1}{p^{m(z+u)}} + \\
+ \frac{p^{\nu+1}(1-p^{-1})}{p^zp^{(\nu_2+1)(z-u)}(1-p^{-2z})(1-p^{-z+u})}.
\end{multline}
Substituting \eqref{G1 pnu+1 eq3} into \eqref{G1 pnu+1 eq1}, we have
\begin{multline}\label{G1 pnu+1 eq4}
\G(1,p^{\nu+1},z,u) =
\frac{(1-p^{-z+u})(1-p^{-2z})}{(1-p^{-1-2z})p^{\nu_2(z-u)}}
\left(p^{\nu+1/2}+\frac{p^{\nu+1}(1-p^{-1})}{(1-p^{-2z})p^{z}}\right)\sum_{m=0}^{\nu_2}\frac{1}{p^{m(z+u)}} + \\
+ \frac{p^{\nu+1}(1-p^{-1})}{p^zp^{(\nu_2+1)(z-u)}(1-p^{-1-2z})}.
\end{multline}
Taking the derivative, we show that for $\nu=2\nu_2$,
\begin{multline}\label{G1 pnu+1 eq5}
\frac{d}{du}\G(1,p^{\nu+1},z,u)\Bigl|_{u=0} =
\frac{-(1-p^{-2z})\log p}{(1-p^{-1-2z})p^{(\nu_2+1)z}}\left(p^{\nu+1/2}+\frac{p^{\nu+1}(1-p^{-1})}{(1-p^{-2z})p^{z}}\right)\sum_{m=0}^{\nu_2}\frac{1}{p^{mz}} + \\
+ \frac{(1-p^{-z})(1-p^{-2z})}{(1-p^{-1-2z})p^{\nu_2z}}
\left(p^{\nu+1/2}+\frac{p^{\nu+1}(1-p^{-1})}{(1-p^{-2z})p^{z}}\right)\sum_{m=0}^{\nu_2}\frac{(\nu_2-m)\log p}{p^{mz}} + \\
+ \frac{p^{\nu+1}(1-p^{-1})(\nu_2+1)\log p}{p^{(\nu_2+2)z}(1-p^{-1-2z})}.
\end{multline}
Now consider the case where $\nu=2\nu_1+1$ is odd. Using \eqref{G(n, p gamma)} if $p>2$ (and \eqref{G(n, 2 gamma)} if $p=2$) and making the change of variables $n=\nu_1-m$ again, we obtain
\begin{multline}\label{G1 pnu+1 eq6}
\sum_{n,q=0}^{\infty}\frac{G(p^{2n+2q},p^{q+\nu+1})}{p^{n(z-u)+q(1+z)}} =
\sum_{n=0}^{\nu_1}\frac{1}{p^{n(z-u)}}\left(\frac{p^{\nu+1/2}}{p^{(\nu-2n)z}} +
\sum_{\substack{q=\nu-2n+1\\q\equiv0\Mod{2}}}^{\infty}\frac{\phi(p^{q+\nu+1})}{p^{q(1+z)}}
\right) + \\
+ \sum_{n=\nu_1+1}^{\infty}\frac{1}{p^{n(z-u)}}
\sum_{\substack{q=0\\q\equiv0\Mod{2}}}^{\infty}\frac{\phi(p^{q+\nu+1})}{p^{q(1+z)}} =
\left(\frac{p^{\nu+1/2}}{p^{\nu_1(z-u)+z}}+\frac{p^{\nu+1}(1-p^{-1})}{(1-p^{-2z})p^{2z+\nu_1(z-u)}}\right)\sum_{m=0}^{\nu_1}\frac{1}{p^{m(z+u)}} + \\
+ \frac{p^{\nu+1}(1-p^{-1})}{p^{(\nu_1+1)(z-u)}(1-p^{-2z})(1-p^{-z+u})}.
\end{multline}
Substituting \eqref{G1 pnu+1 eq6} into \eqref{G1 pnu+1 eq1}, we have
\begin{multline}\label{G1 pnu+1 eq7}
\G(1,p^{\nu+1},z,u) =
\frac{(1-p^{-z+u})(1-p^{-2z})}{(1-p^{-1-2z})p^{\nu_1(z-u)+z}}
\left(p^{\nu+1/2}+\frac{p^{\nu+1}(1-p^{-1})}{(1-p^{-2z})p^{z}}\right)\sum_{m=0}^{\nu_1}\frac{1}{p^{m(z+u)}} + \\
+ \frac{p^{\nu+1}(1-p^{-1})}{p^{(\nu_1+1)(z-u)}(1-p^{-1-2z})}.
\end{multline}
Taking the derivative, we find that for $\nu=2\nu_1+1$,
\begin{multline}\label{G1 pnu+1 eq8}
\frac{d}{du}\G(1,p^{\nu+1},z,u)\Bigl|_{u=0} =
\frac{-(1-p^{-2z})\log p}{(1-p^{-1-2z})p^{(\nu_1+2)z}}\left(p^{\nu+1/2}+\frac{p^{\nu+1}(1-p^{-1})}{(1-p^{-2z})p^{z}}\right)\sum_{m=0}^{\nu_1}\frac{1}{p^{mz}} + \\
+ \frac{(1-p^{-z})(1-p^{-2z})}{(1-p^{-1-2z})p^{(\nu_1+1)z}}
\left(p^{\nu+1/2}+\frac{p^{\nu+1}(1-p^{-1})}{(1-p^{-2z})p^{z}}\right)\sum_{m=0}^{\nu_1}\frac{(\nu_1-m)\log p}{p^{mz}} + \\
+ \frac{p^{\nu+1}(1-p^{-1})(\nu_1+1)\log p}{p^{(\nu_1+1)z}(1-p^{-1-2z})}.
\end{multline}
It follows from \eqref{Sound series 0}, \eqref{Sound series 2}, \eqref{Sp(0,0)}, and \eqref{S2(0,0)} that
\begin{multline}\label{G1 p pnu eq1}
\G(p,p^{\nu},z,u)=\frac{\zeta(1+2z)}{\zeta(z-u)\zeta(2z)}
\sum_{n,q=1}^{\infty}\frac{G(p^2n^2q^2,p^{\nu+1}q)}{n^{z-u}q^{1+z}} = \\
= \frac{(1-p^{-z+u})(1-p^{-2z})}{1-p^{-1-2z}}
\sum_{n,q=0}^{\infty}\frac{G(p^{2n+2q+2},p^{q+\nu+1})}{p^{n(z-u)+q(1+z)}}.
\end{multline}
Let $\nu=2\nu_2$. Using \eqref{G(n, p gamma)} if $p>2$ (and \eqref{G(n, 2 gamma)} if $p=2$) and performing the change of variables $n=\nu_2-1-m$, we obtain
\begin{multline}\label{G1 p pnu eq2}
\sum_{n,q=0}^{\infty}\frac{G(p^{2n+2q+2},p^{q+\nu+1})}{p^{n(z-u)+q(1+z)}} =
\sum_{n=0}^{\nu_2-1}\frac{1}{p^{n(z-u)}}\left(\frac{p^{\nu+1/2}}{p^{(\nu-2-2n)z}} +
\sum_{\substack{q=\nu-2n-1\\q\equiv1\Mod{2}}}^{\infty}\frac{\phi(p^{q+\nu+1})}{p^{q(1+z)}}
\right) + \\
+ \sum_{n=\nu_2}^{\infty}\frac{1}{p^{n(z-u)}}
\sum_{\substack{q=0\\q\equiv1\Mod{2}}}^{\infty}\frac{\phi(p^{q+\nu+1})}{p^{q(1+z)}} =
\left(\frac{p^{\nu+1/2}}{p^{(\nu_2-1)(z-u)}}+\frac{p^{\nu+1}(1-p^{-1})}{(1-p^{-2z})p^{z+(\nu_2-1)(z-u)}}\right)\sum_{m=0}^{\nu_2-1}\frac{1}{p^{m(z+u)}} + \\
+ \frac{p^{\nu+1}(1-p^{-1})}{p^zp^{\nu_2(z-u)}(1-p^{-2z})(1-p^{-z+u})}.
\end{multline}
Substituting \eqref{G1 p pnu eq2} into \eqref{G1 p pnu eq1}, we have
\begin{multline}\label{G1 p pnu eq3}
\G(p,p^{\nu},z,u) =
\frac{(1-p^{-z+u})(1-p^{-2z})}{(1-p^{-1-2z})p^{(\nu_2-1)(z-u)}}
\left(p^{\nu+1/2}+\frac{p^{\nu+1}(1-p^{-1})}{(1-p^{-2z})p^{z}}\right)\sum_{m=0}^{\nu_2-1}\frac{1}{p^{m(z+u)}} + \\
+ \frac{p^{\nu+1}(1-p^{-1})}{p^zp^{\nu_2(z-u)}(1-p^{-1-2z})}.
\end{multline}
Taking the derivative, we show that for $\nu=2\nu_2$,
\begin{multline}\label{G1 p pnu eq4}
\frac{d}{du}\G(p,p^{\nu},z,u)\Bigl|_{u=0} =
\frac{-(1-p^{-2z})\log p}{(1-p^{-1-2z})p^{\nu_2z}}\left(p^{\nu+1/2}+\frac{p^{\nu+1}(1-p^{-1})}{(1-p^{-2z})p^{z}}\right)\sum_{m=0}^{\nu_2-1}\frac{1}{p^{mz}} + \\
+ \frac{(1-p^{-z})(1-p^{-2z})}{(1-p^{-1-2z})p^{(\nu_2-1)z}}
\left(p^{\nu+1/2}+\frac{p^{\nu+1}(1-p^{-1})}{(1-p^{-2z})p^{z}}\right)\sum_{m=0}^{\nu_2-1}\frac{(\nu_2-1-m)\log p}{p^{mz}} + \\
+ \frac{p^{\nu+1}(1-p^{-1})\nu_2\log p}{p^{(\nu_2+1)z}(1-p^{-1-2z})}.
\end{multline}
Let $\nu=2\nu_1+1$. Using \eqref{G(n, p gamma)} if $p>2$ (and \eqref{G(n, 2 gamma)} if $p=2$) and performing the change of variables $n=\nu_1-1-m$, we obtain
\begin{multline}\label{G1 p pnu eq5}
\sum_{n,q=0}^{\infty}\frac{G(p^{2n+2q+2},p^{q+\nu+1})}{p^{n(z-u)+q(1+z)}} =
\sum_{n=0}^{\nu_1-1}\frac{1}{p^{n(z-u)}}\left(\frac{p^{\nu+1/2}}{p^{(\nu-2-2n)z}} +
\sum_{\substack{q=\nu-2n-1\\q\equiv0\Mod{2}}}^{\infty}\frac{\phi(p^{q+\nu+1})}{p^{q(1+z)}}
\right) + \\
+ \sum_{n=\nu_1}^{\infty}\frac{1}{p^{n(z-u)}}
\sum_{\substack{q=0\\q\equiv0\Mod{2}}}^{\infty}\frac{\phi(p^{q+\nu+1})}{p^{q(1+z)}} =
\left(\frac{p^{\nu+1/2}}{p^{(\nu_1-1)(z-u)+z}}+\frac{p^{\nu+1}(1-p^{-1})}{(1-p^{-2z})p^{2z+(\nu_1-1)(z-u)}}\right)\sum_{m=0}^{\nu_1-1}\frac{1}{p^{m(z+u)}} + \\
+ \frac{p^{\nu+1}(1-p^{-1})}{p^{\nu_1(z-u)}(1-p^{-2z})(1-p^{-z+u})}.
\end{multline}
Substituting \eqref{G1 p pnu eq5} into \eqref{G1 p pnu eq1}, we have
\begin{multline}\label{G1 p pnu eq6}
\G(p,p^{\nu},z,u) =
\frac{(1-p^{-z+u})(1-p^{-2z})}{(1-p^{-1-2z})p^{(\nu_1-1)(z-u)+z}}
\left(p^{\nu+1/2}+\frac{p^{\nu+1}(1-p^{-1})}{(1-p^{-2z})p^{z}}\right)\sum_{m=0}^{\nu_2-1}\frac{1}{p^{m(z+u)}} + \\
+ \frac{p^{\nu+1}(1-p^{-1})}{p^{\nu_1(z-u)}(1-p^{-1-2z})}.
\end{multline}
Taking the derivative yields, for $\nu=2\nu_1+1$,
\begin{multline}\label{G1 p pnu eq7}
\frac{d}{du}\G(p,p^{\nu},z,u)\Bigl|_{u=0} =
\frac{-(1-p^{-2z})\log p}{(1-p^{-1-2z})p^{(\nu_1+1)z}}\left(p^{\nu+1/2}+\frac{p^{\nu+1}(1-p^{-1})}{(1-p^{-2z})p^{z}}\right)\sum_{m=0}^{\nu_1-1}\frac{1}{p^{mz}} + \\
+ \frac{(1-p^{-z})(1-p^{-2z})}{(1-p^{-1-2z})p^{\nu_1z}}
\left(p^{\nu+1/2}+\frac{p^{\nu+1}(1-p^{-1})}{(1-p^{-2z})p^{z}}\right)\sum_{m=0}^{\nu_1-1}\frac{(\nu_1-1-m)\log p}{p^{mz}} + \\
+ \frac{p^{\nu+1}(1-p^{-1})\nu_1\log p}{p^{\nu_1z}(1-p^{-1-2z})}.
\end{multline}
For $\nu=2\nu_2$, it follows from \eqref{G1 pnu+1 eq4} and \eqref{G1 p pnu eq3} that
\begin{multline}\label{G1pnu+1-Gppnu eq1}
(\nu+1)\G(1,p^{\nu+1},z,0)-\frac{\nu}{p^z}\G(p,p^{\nu},z,0) =
\frac{p^{\nu+1}(1-p^{-1})}{p^{(\nu_2+2)z}(1-p^{-1-2z})} + \\
+ \frac{(1-p^{-z})(1-p^{-2z})}{(1-p^{-1-2z})p^{\nu_2z}}
\left(p^{\nu+1/2}+\frac{p^{\nu+1}(1-p^{-1})}{(1-p^{-2z})p^{z}}\right)\left(\sum_{m=0}^{\nu_2-1}\frac{1}{p^{mz}}+\frac{\nu+1}{p^{\nu_2z}}
\right).
\end{multline}
For $\nu=2\nu_1+1$, it follows from \eqref{G1 pnu+1 eq7} and \eqref{G1 p pnu eq6} that
\begin{multline}\label{G1pnu+1-Gppnu eq2}
(\nu+1)\G(1,p^{\nu+1},z,0)-\frac{\nu}{p^z}\G(p,p^{\nu},z,0) =
\frac{p^{\nu+1}(1-p^{-1})}{p^{(\nu_1+1)z}(1-p^{-1-2z})} + \\
+ \frac{(1-p^{-z})(1-p^{-2z})}{(1-p^{-1-2z})p^{(\nu_1+1)z}}
\left(p^{\nu+1/2}+\frac{p^{\nu+1}(1-p^{-1})}{(1-p^{-2z})p^{z}}\right)\left(\sum_{m=0}^{\nu_1-1}\frac{1}{p^{mz}}+\frac{\nu+1}{p^{\nu_1z}}
\right).
\end{multline}
For $\nu=2\nu_2$, it follows from \eqref{G1 pnu+1 eq5} and \eqref{G1 p pnu eq4} that
\begin{multline}\label{G1pnu+1-Gppnu eq3}
\frac{d}{du}\G(1,p^{\nu+1},z,u)\Bigl|_{u=0}-\frac{1}{p^z}\frac{d}{du}\G(p,p^{\nu},z,u)\Bigl|_{u=0} = \frac{p^{\nu+1}(1-p^{-1})\log p}{p^{(\nu_2+2)z}(1-p^{-1-2z})} - \\
- \frac{(1-p^{-2z})\log p}{(1-p^{-1-2z})p^{(\nu_2+1)z}}\left(p^{\nu+1/2}+\frac{p^{\nu+1}(1-p^{-1})}{(1-p^{-2z})p^{z}}\right)\frac{1}{p^{\nu_2z}} + \\
+ \frac{(1-p^{-z})(1-p^{-2z})}{(1-p^{-1-2z})p^{\nu_2z}}
\left(p^{\nu+1/2}+\frac{p^{\nu+1}(1-p^{-1})}{(1-p^{-2z})p^{z}}\right)\sum_{m=0}^{\nu_2-1}\frac{\log p}{p^{mz}}.
\end{multline}
For $\nu=2\nu_1+1$, it follows from \eqref{G1 pnu+1 eq8} and \eqref{G1 p pnu eq7} that
\begin{multline}\label{G1pnu+1-Gppnu eq4}
\frac{d}{du}\G(1,p^{\nu+1},z,u)\Bigl|_{u=0}-\frac{1}{p^z}\frac{d}{du}\G(p,p^{\nu},z,u)\Bigl|_{u=0} = \frac{p^{\nu+1}(1-p^{-1})\log p}{p^{(\nu_1+1)z}(1-p^{-1-2z})} - \\
- \frac{(1-p^{-2z})\log p}{(1-p^{-1-2z})p^{(\nu_1+2)z}}\left(p^{\nu+1/2}+\frac{p^{\nu+1}(1-p^{-1})}{(1-p^{-2z})p^{z}}\right)\frac{1}{p^{\nu_1z}} + \\
+ \frac{(1-p^{-z})(1-p^{-2z})}{(1-p^{-1-2z})p^{(\nu_1+1)z}}
\left(p^{\nu+1/2}+\frac{p^{\nu+1}(1-p^{-1})}{(1-p^{-2z})p^{z}}\right)\sum_{m=0}^{\nu_1-1}\frac{\log p}{p^{mz}}.
\end{multline}
Recall that to prove the lemma, we must verify \eqref{deriv G(1,pnu+1)-G(p,pnu) eq2}.
For $\nu=2\nu_2$, it follows from \eqref{G1pnu+1-Gppnu eq1} and \eqref{G1pnu+1-Gppnu eq3} that
\begin{multline}\label{DeltaG nu ev eq1}
\Delta\G(\nu,p,z) = -
\left(\frac{1}{p^{1/2}}+\frac{(1-p^{-1})}{(1-p^{-2z})p^{z}}\right)
\frac{(1-p^{-2z})}{(1-p^{-1-2z})}\left(\frac{1}{p^{(2\nu_2+1)z}}+\left(1-\frac{1}{p^z}\right)\frac{\nu+1}{p^{\nu z}}\right)\log p = \\
= \frac{-\log{p}}{(1-p^{-1-2z})p^{\nu z}}\left(\frac{1}{p^{1/2}}\left(1-\frac{1}{p^{2z}}\right)+\frac{1}{p^{z}}\left(1-\frac{1}{p}\right)\right)
\left(1+\nu-\frac{\nu}{p^z}\right) = \\
= \frac{-\log p}{p^{(\nu+1)z}}
\frac{(1+p^{-1/2+z})}{(1+p^{-1/2-z})}\left(1+\nu-\frac{\nu}{p^z}\right),
\end{multline}
thereby establishing \eqref{deriv G(1,pnu+1)-G(p,pnu) eq2}.
Similarly, for $\nu=2\nu_1+1$, we obtain \eqref{deriv G(1,pnu+1)-G(p,pnu) eq2} from \eqref{G1pnu+1-Gppnu eq2} and \eqref{G1pnu+1-Gppnu eq4}. This completes the proof of the lemma.

\end{proof}
%%%%%%%%%%%%%%%%%%%%%%%%%%%%%%%%%%%%%%%%%%%%%%%%%%%%%%%%%%%%%%%%%%%%%%%%%%%%%%%%%%%%%%%%%%%
%%%%%%%%%%%%%%%%%%%%

\end{document}